\documentclass[11pt]{article}

\usepackage{fullpage}
\usepackage{amsmath,amsthm,amsfonts,dsfont}
\usepackage{amssymb,latexsym}
\usepackage{palatino}
\usepackage{mathpazo}
\usepackage{stmaryrd}
\usepackage{mathtools}
\usepackage{hyperref}
\usepackage{xspace}
\usepackage{graphicx}
\usepackage{stackrel}

\newtheorem{theorem}{Theorem}[section]

\newtheorem*{remark}{Remark}
\newtheorem{proposition}[theorem]{Proposition}
\newtheorem{lemma}[theorem]{Lemma}
\newtheorem{conjecture}[theorem]{Conjecture}
\newtheorem{claim}[theorem]{Claim}

\newtheorem{corollary}[theorem]{Corollary}
\newtheorem{definition}[theorem]{Definition}

\newcommand{\N}{\ensuremath{\mathbb{N}}}

\newcommand{\R}{\ensuremath{\mathbb{R}}}
\newcommand{\Z}{\ensuremath{\mathbb{Z}}}
\newcommand{\PSD}{\ensuremath{\mathbb{S}_+}}

 \newcommand{\eps}{\varepsilon} 
\renewcommand{\epsilon}{\varepsilon} 
\newcommand{\calV}{\mathcal{V}}

\newcommand{\eqdef}{\mathbin{\stackrel{\rm def}{=}}}

\renewcommand{\vec}[1]{\ensuremath{\mathbf{#1}}}

\newcommand{\problem}[1]{\mbox{#1}\xspace}

\newcommand{\poly}{\mathrm{poly}}

\DeclareMathOperator*{\E}{\mathbb{E}}

\DeclareMathOperator*{\tr}{\mathrm{tr}}

\newcommand{\GapSPP}{\textrm{GapSPP}}
\newcommand{\GapCRP}{\textrm{GapCRP}}

\newcommand{\DSGS}[2]{\ensuremath{#1\text{-}\problem{DSGS}^{#2}}}

\newcommand{\da}{\downarrow}
\newcommand{\setminuszero}{\setminus \set{\vec{0}}}

\newcommand{\NP}{\textsf{NP}}
\newcommand{\coNP}{\textsf{coNP}}
\newcommand{\AM}{\textsf{AM}}
\newcommand{\coAM}{\textsf{coAM}}

\newcommand{\Ball}{B}

\newcommand{\T}{\mathsf{T}}
\newcommand{\pr}[2]{\langle{#1, #2}\rangle}
\makeatletter
\def\imod#1{\allowbreak\mkern4mu({\operator@font mod}\,\,#1)}
\makeatother

\DeclareMathOperator{\ran}{im}

\newcommand{\lat}{\mathcal{L}}

\newcommand{\gs}[1]{\ensuremath{\widetilde{#1}}}

\DeclareMathOperator{\vol}{vol}

\DeclareMathOperator{\dist}{dist}

\DeclarePairedDelimiter\inner{\langle}{\rangle}
\DeclarePairedDelimiter\abs{\lvert}{\rvert}
\DeclarePairedDelimiter\set{\{}{\}}
\DeclarePairedDelimiter\parens{(}{)}

\DeclarePairedDelimiter\bracks{[}{]}

\DeclarePairedDelimiter\floor{\lfloor}{\rfloor}
\DeclarePairedDelimiter\ceil{\lceil}{\rceil}
\DeclarePairedDelimiter\length{\lVert}{\rVert}

\newif\ifnotes\notesfalse

\ifnotes
\usepackage{color}
\definecolor{mygrey}{gray}{0.50}
\newcommand{\notename}[2]{{\textcolor{mygrey}{\footnotesize{\bf (#1:} {#2}{\bf ) }}}}
\newcommand{\noteswarning}{{\begin{center} {\Large WARNING: NOTES ON}\end{center}}}

\else

\newcommand{\notename}[2]{{}}
\newcommand{\noteswarning}{{}}

\fi

\newcommand{\onote}[1]{{\notename{Oded}{#1}}}
\newcommand{\dnote}[1]{{\notename{Daniel}{#1}}}

\begin{document}

\title{Towards Strong Reverse Minkowski-type Inequalities for Lattices}
\author{
Daniel Dadush\thanks{Centrum Wiskunde \& Informatica, Amsterdam. Supported by the NWO Veni grant 639.071.510. \texttt{dadush@cwi.nl}}
\and
Oded Regev\thanks{Courant Institute of Mathematical Sciences, New York
 University. Supported by the Simons Collaboration on Algorithms and Geometry and by the National Science Foundation (NSF) under Grant No.~CCF-1320188. Any opinions, findings, and conclusions or recommendations expressed in this material are those of the authors and do not necessarily reflect the views of the NSF.}
}
\date{}
\maketitle

\noteswarning

%

\begin{abstract}
We present a natural reverse Minkowski-type inequality for lattices, which gives
\emph{upper bounds} on the number of lattice points in a Euclidean ball in terms
of sublattice determinants, and conjecture its optimal form. 
The conjecture exhibits a surprising wealth of connections to various areas
in mathematics and computer science, including a conjecture motivated 
by integer programming by Kannan and Lov{\'a}sz (Annals of Math.\ 1988),
a question from additive combinatorics asked by Green,
a question on Brownian motions asked by Saloff-Coste (Colloq.\ Math.\ 2010),
a theorem by Milman and Pisier from convex geometry (Ann.\ Probab.\ 1987), 
worst-case to average-case reductions in lattice-based cryptography,  
and more. We present these connections, provide evidence for the conjecture,
and discuss possible approaches towards a proof. Our main technical 
contribution is in proving that our conjecture implies the $\ell_2$ case of
the Kannan and Lov{\'a}sz conjecture. The proof relies on a novel 
convex relaxation for the covering radius, and a rounding procedure
for based on ``uncrossing'' lattice subspaces. 
\end{abstract}


\tableofcontents
\newpage


\section{Introduction}

A lattice $\lat \subset \R^n$ is defined as the set of all
integer linear combinations of $n$ linearly independent vectors $B =
(\vec{b}_1,\dots,\vec{b}_n)$ in $\R^n$, where we call $B$ a basis for $\lat$.
The determinant of $\lat$ is defined as $\det(\lat) = |\det(B)|$, which is
invariant to the choice of basis for $\lat$. The determinant measures 
the ``global density'' of the lattice. More precisely, $\det(\lat)^{-1}$ 
is the asymptotic number of lattice points per unit volume.

One of the earliest and most important results in this area is Minkowski's
(First) Theorem from 1889, which guarantees the existence of lattice points in
any large enough symmetric convex body (see
Theorem~\ref{thm:mink-first} for the precise statement). In particular, it
implies that any lattice $\lat$ with $\det(\lat) \le 1$ must contain
exponentially in $n$ many points of Euclidean norm at most $\sqrt{n}$.
Informally, Minkowski's theorem can be described as saying that ``global density
implies local density''.  

The starting point of our investigation is an attempt to \emph{reverse} this implication. 
Assume $\lat$ has local density, i.e., exponentially many points of norm at most $\sqrt{n}$.
Can we conclude that it must also have global density, i.e., $\det(\lat) \le 1$? 
A moment's thought reveals that we cannot hope for such a strong conclusion. 
Indeed, consider the lattice $\lat$ generated by the basis $\set{\vec e_1,\ldots, \vec e_{n-1}, M \vec e_n}$ 
where $M$ is arbitrarily large. Then $\det(\lat)=M$ yet $\lat$ has exponentially many
points of norm at most $\sqrt{n}$. 

A natural conjecture would therefore say that local density implies global 
density \emph{in a subspace}, or in other words, that there is a (not necessarily full-rank) sublattice with low determinant. This is the essence of our main conjecture. 

To make this more precise, consider the lattice $\Z^n$. 
Obviously, all its sublattices have determinant at least $1$ (as the determinant
of any integer matrix is integer). 
Moreover, it is not difficult to see that the number of integer points in a ball
of radius $r$ is approximately $n^{r^2}$ for $r \ll \sqrt{n}$.
Our main conjecture basically says that 
\emph{$\Z^n$ has the highest local density among
all lattices whose sublattices all have determinant at least $1$.} See Conjecture~\ref{con:maineta} for the formal statement, which also makes the connection
to the so-called smoothing parameter explicit. 

\begin{conjecture}\label{con:mainetaintro}(Main conjecture, informal)
Let $\lat \subset \R^n$ be an $n$-dimensional lattice whose sublattices all have
determinant at least $1$. Then, for any $r>0$, the number of lattice points of Euclidean
norm at most $r$ is at most $\exp( \poly\log n \cdot r^2)$. 
\end{conjecture}

While one can consider other variants of the conjecture 
(see the stronger variant in Section~\ref{sec:morepreciserevminko}
and the weaker variants in Section~\ref{sec:variants}), 
the main conjecture is the most appealing, as described next.

\paragraph{Connections.}
The conjecture exhibits a surprising wealth of connections to various areas
in mathematics, ranging from additive combinatorics to convex geometry
and to heat diffusion on manifolds. We mention some of those next. 
\begin{itemize}
\item
The Kannan-Lov{\'a}sz (KL) conjecture~\cite{KL88} is a nearly tight 
characterization of the covering radius of a lattice with respect to a given
convex body in terms of projection volumes and determinants. 
The main technical contribution of our paper, appearing in Section~\ref{sec:klmain}, 
is to show that the $\ell_2$ version of the KL conjecture~\cite{KL88} 
is implied up to poly-logarithmic factors by the main conjecture.  
This implication is quite nontrivial, and a proof overview will be given below. 
Very briefly, we rely on a novel convex relaxation for the covering radius 
and a rounding strategy for the corresponding dual program to extract the 
relevant subspace. 
We remark that it is possible that our main conjecture in fact implies
the \emph{general} KL conjecture. See the discussion at the end of Section~\ref{sec:klorigin}.
\item
In 2007, it was suggested by Green~\cite{GreenBlog} that a certain 
strong variant of the polynomial Freiman-Ruzsa conjecture over the integers
is false. In work in progress of the second-named author with Lovett~\cite{LovettRegev16},
we prove that this is indeed the case assuming the main conjecture.\footnote{We remark
that for this it would suffice to prove the conclusion in Conjecture~\ref{con:mainetaintro} 
for $r=\sqrt{n}$.} 
No unconditional proof of this is currently known.
\item
The following question was considered by Saloff-Coste (see, e.g.,~\cite{SaloffCoste94,SaloffCoste04,BendikovSaloffCoste03} and especially~\cite[Problem 11]{SaloffCoste10}).
For a lattice $\lat \subset \R^n$, consider the Brownian motion on the flat torus $\R^n/\lat$ starting
from the origin. Its density at any time $t>0$ is a Gaussian of standard deviation $\sqrt{t}$ 
reduced modulo $\lat$. Saloff-Coste asked whether its mixing time in the total variation (or equivalently, $L_1$) 
sense is approximately the same as that in the (easier to analyze) $L_2$ sense. 
This is equivalent to asking whether the smoothing parameter of a lattice
is approximately the same as the ``$L_1$ smoothing parameter.''
In Section~\ref{sec:mixingtime} we provide more background and show that such a result follows from our main conjecture. 
We remark that this seems to be the weakest implication of the main conjecture that
is already an interesting question in its own right. 
\item 
In Section~\ref{sec:complexity} we describe some connections to computational complexity.
We show that our main conjecture implies a very tight reduction from the problem
of approximation the smoothing parameter to the problem of sampling points from
a given coset of a lattice according to a Gaussian (or even subgaussian) distribution. 
Combined with known reductions, this shows that the hardness of 
SIS, a central average-case cryptographic problem, can be based on the worst-case hardness 
of approximation the smoothing parameter to within $\tilde{O}(\sqrt{n})$. 
We also mention an easy implication to the complexity of the problem of 
approximating the covering radius. 
\end{itemize}

\paragraph{Evidence for the conjecture.}
First, as we show in Section~\ref{sec:sanity}, the conjecture passes 
some basic sanity checks. For instance, it is true for natural families of lattices 
including ``rectangular'' lattices (or more generally, direct sum of lattices for which
the conjecture holds) as well as random lattices. We also show there
that a weaker bound on the number of lattice points is true (namely, with
$\poly \log n$ replaced by $\sqrt{n}$).

Second, as we describe in Section~\ref{sec:pisier}, 
the conjecture has a certain ``continuous relaxation'' which 
talks about volumes of sections instead of determinants 
of lattice subspaces. Somewhat surprisingly, that 
relaxation is known to be true and follows from a 
celebrated theorem by Milman and Pisier. 
This is the most non-trivial implication of the main 
conjecture that we know is true.

Finally, we believe that there are several approaches for proving
the main conjecture itself that are worth exploring in more detail,
including one based on additive combinatorics. 
We also tried exploring whether there are natural weaker variants of the conjecture
that might be easier to prove. Those variants are described in Section~\ref{sec:variants},
and some are quite natural in their own right. The variant described in 
Section~\ref{sec:musmoothtight} is particularly appealing, and we believe
that it can be attacked using a Fourier analytic approach. 

\paragraph{Related work.}
There has been considerable work on the problem of bounding or counting 
the number of lattice points in a ball, perhaps the earliest
reference being Gauss's circle problem. Most of the work in this area, however,
considers ``large'' convex bodies and shows that for such bodies 
the number of lattice points is very close to what one would expect 
by a volume heuristic. As such, this seems too coarse to capture the 
subspace structure highlighted by the main conjecture. 
See, e.g.,~\cite{BentkusG97} and references therein for the kind of
results proved in this area since the early 20th century. 

Also related is the work on \emph{stable lattices} (sometimes known as ``semistable''). These are lattices of determinant $1$ whose sublattices all have determinant at least $1$. 
Properties of stable lattices have been studied since the 1970s in 
connection with algebra, topology, and geometry. See~\cite{ShapiraW16}
and references therein. In particular, Shapira and Weiss~\cite{ShapiraW16} consider
a strong quantitative variant of the $\ell_2$ version of the KL conjecture, and
show that it implies the so-called Minkowski conjecture. 
It remains to be seen how exactly our main 
conjecture connects to the work on stable lattices. 

\paragraph{Proof overview of the main technical theorem.}
Our main technical theorem shows how to derive the $\ell_2$ case of the
Kannan-Lov{\'a}sz (KL) conjecture from our main conjecture. 
To explain the KL conjecture, recall that the \emph{covering radius} $\mu$ of 
a lattice is the maximum distance a point in space can be from the lattice.
So for instance $\mu(\Z^n) = \sqrt{n}/2$. An easy lower bound on $\mu(\lat)$ is 
given in terms of the determinant of $\lat$. Namely, since
balls of radius $\mu(\lat)$ centered at all points of $\lat$ cover
$\R^n$, we obtain by volume considerations that 
$\mu(\lat) \gtrsim  \sqrt{n} (\det(\lat))^{1/n}$.
This bound can be far from tight, e.g.,  for the lattice
generated by the basis $\set{\vec e_1,\ldots, \vec e_{n-1}, M \vec e_n}$ 
where $M$ is large. As was the case for the main conjecture,
we can tighten the bound by maximizing over subspaces, namely,
we consider the maximum of $\sqrt{\dim(W)} (\det(\pi_W(\lat)))^{1/\dim(W)}$
over all subspaces $W$ where $\pi_W$ denotes the projection on $W$. 
Since $\mu(\lat) \ge \mu(\pi_W(\lat))$ this is clearly a lower bound on $\mu(\lat)$. 
The $\ell_2$ KL conjecture says that this is nearly tight, i.e., that we can also
\emph{upper bound} $\mu(\lat)$ by the same maximum up to polylogarithmic terms. 

As should be obvious by now, the $\ell_2$ KL conjecture has a similar flavor 
to our main conjecture (and even more so when comparing the formal definitions, see Conjecture~\ref{con:maineta} and Conjecture~\ref{conj:kl-l2}). 
The difference is that our conjecture tries to capture
the number of points in a ball (or to be precise, the smoothing parameter), whereas
$\ell_2$ KL tries to capture the covering radius. 

The first step in our proof is to bound from above the covering radius by a convex
program we call $\mu_{\textnormal{sm}}$.
Intuitively and informally, the convex program tries
to find the ``smallest'' covariance matrix $A$ such that the distribution
obtained by picking a random lattice point and adding a Gaussian random variable
with covariance $A$ is close to uniform over $\R^n$. 
Here by smallest we mean the expected squared norm of a Gaussian with that
covariance, which is simply $\tr(A)$. 
A Gaussian that satisfies this property is said to ``smooth'' $\lat$. 
It is intuitively clear (and not difficult to prove formally) that this program bounds $\mu(\lat)^2$ from above,
since in order for the Gaussian to smooth $\lat$, it
must ``reach'' all points in space, and hence its norm must be at least $\mu(\lat)$. 
To state this program formally, we recall that smoothness has an elegant
equivalent definition (which follows from the Poisson summation formula)
in terms of the a Gaussian sum on dual lattice points, namely
\begin{equation}
\label{eq:smooth-mu-intro}
\mu_{\rm{sm}}(\lat)^2 = \min \set[\Big]{\tr(A) : A \succeq 0, \sum_{\vec{y} \in \lat^*
\setminuszero} e^{-\pi \vec{y}^\T A \vec{y}} \leq 1/2} \text{.}
\end{equation}

At this point the natural thing to do would be to consider the dual
of $\mu_{\textnormal{sm}}$, which is a maximization problem, and use
its optimal solution to find a subspace the projection on which has large determinant, as needed for the $\ell_2$ KL conjecture. 
Unfortunately, the dual program seems difficult to deal with since 
the single constraint in~\eqref{eq:smooth-mu-intro} ``entangles'' all the information about lattice points in a complicated way, and it is not clear how
it would help in identifying such a subspace.

This is were our main conjecture comes in. 
We formulate 
another convex program $\mu_{\textnormal{det}}$ which, assuming the main conjecture, bounds 
$\mu_{\textnormal{sm}}$ from above up to polylogarithmic factors. In more detail, 
that program also tries to find the covariance matrix with smallest trace, 
but its constraints basically say that $A$ should be such that 
in the dual lattice, all sublattice determinants are large
(relative to $A$). By the main conjecture, large
sublattice determinants imply small number of points
in balls, which in turn, implies smoothing.
(In fact, the formal statement of the main conjecture in Conjecture~\ref{con:maineta}
is already stated in terms of a Gaussian sum over lattice points as in~\eqref{eq:smooth-mu-intro}, so there
is no need for this detour through number of points in balls.)
Since under the main conjecture, the constraints in $\mu_{\textnormal{det}}$ imply those in 
$\mu_{\textnormal{sm}}$ (i.e., are weaker),
we obtain that $\mu_{\textnormal{det}}$ bounds $\mu_{\textnormal{sm}}$ from
above up to polylogarithmic factors, as desired. 

Since $\mu_{\textnormal{det}}$ directly puts in a constraint for every subspace, 
the subspace structure comes out explicitly, and we are finally in position take the dual. 
It turns out that the dual of $\mu_{\textnormal{det}}$ has a reasonably nice form,
and a solution to it can be seen as some kind of mixture of various lattice
subspaces. The last and most technically demanding part of the proof is to ``round'' that dual solution, 
i.e., we show how to take an arbitrary mixture of lattice subspaces and extract from
it just one lattice subspace that is nearly as good. There are several steps 
to this proof, the most interesting one being a sort of ``uncrossing inequality,''
showing that if the solution includes two subspaces $V$ and $W$, we can replace them
with the subspaces $V+W$ and $V \cap W$ in a way that does not decrease the goal
function. We now repeat this uncrossing step over and over again. Notice that we make
progress as long as there are two subspaces such that neither is contained in the other. 
Therefore, after sufficiently many iterations, we arrive at a chain, i.e., a sequence
of subspaces $W_1 \subseteq W_2 \subseteq \cdots \subseteq W_m$. Using careful bucketing,
we show that one of these subspaces must be nearly as good as the mixture, and this
completes the proof. 

\paragraph{Outline.}
We include some preliminaries in Section~\ref{sec:prelims}
and state the main conjecture formally in Section~\ref{sec:mainconj}. The remaining
sections are mostly independent. First, in Section~\ref{sec:klmain} we include
our main technical theorem on the $\ell_2$ KL conjecture.  We then discuss in
Section~\ref{sec:variants} some weaker variants of the conjecture, 
and in Section~\ref{sec:sm-cov-body} a unified way to view them all using 
certain universal convex bodies. In
Section~\ref{sec:mixingtime} we discuss the mixing time of Brownian motions. In
Section~\ref{sec:complexity}, we give some implications of our conjectures to
computational complexity. Section~\ref{sec:morepreciserevminko} explores a more
precise type of reverse Minkowski inequality. In Section~\ref{sec:pisier} we
discuss the continuous analogue of our conjecture.  We conclude in
Section~\ref{sec:sanity}, with some basic sanity checks of the conjecture.
Since the proof in
Section~\ref{sec:klmain} is somewhat involved, the reader might prefer to read
some of the following sections first, as most of them are lighter. 

\paragraph{Acknowledgements.} We thank 
Mark Rudelson for pointing us to the Milman-Pisier theorem and
Chris Peikert for suggesting an early version of the weak conjecture and its relation to $L_1$ smoothing. We also thank them and Kai-Min Chung, Shachar Lovett, and Noah Stephens-Davidowitz for useful discussions.


\section{Preliminaries}
\label{sec:prelims}

We write $X \lesssim Y$ to mean that there exists a universal constant
$C>0$ such that $X \leq CY$, and
similarly for $X \gtrsim Y$.

\paragraph{{\bf Basic Concepts.}} Let $\Ball_2^n = \set{\vec{x} \in \R^n:
\|\vec{x}\|_2 \leq 1}$ denote the unit Euclidean ball and $S^{n-1} =
\set{\vec{x} \in \R^n: \|\vec{x}\|_2 = 1} = \partial \Ball_2^n$ denote the unit
sphere in $\R^n$ . For subsets $A,B \subseteq \R^n$, we denote their Minkowski
sum $A+B = \set{a+b: a \in A, b \in B}$. Define ${\rm span}(A)$ to be the
smallest linear subspace containing $A$.

A set $K \subseteq \R^n$ is \emph{convex} if for all $\vec{x}$, $\vec{y} \in K$, $\lambda
\in [0,1]$, we have $\lambda \vec{x} + (1-\lambda)\vec{y} \in K$. $K$ is a \emph{convex body}
if additionally it is compact and has non-empty interior. $K$ is
\emph{symmetric} if $K=-K$. If $K$ is a symmetric convex body, then the functional
$\|\vec{x}\|_K = \min \set{s \geq 0: \vec{x} \in sK}$, $\forall \vec{x} \in
\R^n$, defines a norm on $\R^n$, which we call the norm induced by $K$ or
the $K$-norm.

\paragraph{{\bf Linear Algebra.}} We denote the $n \times n$ identity matrix by
$I_n$, or simply $I$ when the dimension is clear.  For a linear subspace $W
\subseteq \R^n$, define $\pi_W$ as $n \times n$ orthogonal projection matrix
onto $W$. Note that if the
columns of a matrix $O_W$ form an orthornormal basis for $W$, then $\pi_W = O_W
O_W^\T$. We define $W^\perp = \set{x \in \R^n: \pr{\vec{x}}{\vec{w}} = 0,
\forall \vec{w} \in W}$.  For a subspace $V \subseteq \R^n$, we write $W \perp
V$ if $W$ and $V$ are orthogonal subspaces, i.e., if $W \subseteq V^\perp$.  For
a matrix $B \in \R^{n \times d}$, let $\ran(B)$ and ${\rm ker}(B)$ denote the
image (also known as range) and kernel of $B$ respectively. Let $B^\T$ denote
the transpose of $B$. A matrix is symmetric if $B^\T = B$. If $B$ is square,
i.e., $n=d$, define $\tr(B) = \sum_{i \in [n]} B_{ii}$. Define the Moore-Penrose
pseudoinverse $B^+ \in \R^{d \times n}$ of $B$ to be the unique matrix such that
$B B^+ = B^{+\T}B^\T = \pi_{\ran(B)}$, $B^+ B = B^\T B^{+\T} = \pi_{{\rm
ker}(B)^\perp}$. Note that $B^{++} = B$. If $B$ is symmetric, note that $\ran(B) = {\rm ker}(B)^\perp$.
If $B$ is square (i.e.,~$n=d$) and non-singular then $B^+ = B^{-1}$, the
standard matrix inverse.  
Define the operator norm $\|B\| = \max_{\|\vec{x}\|=1}
\|B\vec{x}\|_2$ and Frobenius norm $\|B\|_F = \sqrt{\tr(B^T B)}$.

Given the columns of $B = (\vec{b}_1,\ldots,\vec{b}_d)$, we define its
Gram-Schmidt orthogonalization $(\gs{\vec{b}}_1,\ldots, \gs{\vec{b}}_d)$ by
$\gs{\vec{b}}_i = \pi_i(\vec{b}_i)$ where $\pi_i = \pi_{{\rm
span}(b_1,\dots,b_{i-1})^\perp}$, $i \in [d]$. We shall call $\pi_i$, $i \in
[d]$, the $i^{th}$ Gram-Schmidt projection of $B$. A useful fact is that $B$ is
non-singular iff all the Gram-Schmidt vectors are all non-zero. 

\subsection{Properties of positive semidefinite matrices}

Here we include some basic facts regarding positive semidefinite matrices. 
See, e.g., Bhatia's book~\cite{BhatiaPSDBook} for proofs and further discussion. 

A matrix $X \in \R^{n \times n}$ is positive semidefinite (PSD) if $X$ is
symmetric and $\vec{y}^\T X \vec{y} \geq 0$, $\forall \vec{y} \in \R^n$.
Equivalently, $X$ is PSD $\Leftrightarrow$ $X = B^\T B$ for some matrix $B$
$\Leftrightarrow$ $X$ is symmetric and all its eigenvalues are non-negative. We
note that any PSD matrix $X$ has a unique PSD square root, which we denote
$X^{1/2}$. We denote the cone of $n \times n$ PSD matrices by $\PSD^n$. For
symmetric $X,Y \in \R^{n \times n}$, we shall write $X \preceq Y$ in the Loewner
ordering if $Y-X$ is positive semidefinite. We write
$\lambda_1(X) \geq \lambda_2(X) \geq \dots \geq \lambda_n(X)$ for the eigenvalues 
of $X$ ordered from largest to smallest (note they are all real by
symmetry of $X$).  

Let $W \subseteq \R^n$ be a linear subspace and let $X \in \PSD^n$. 

\begin{definition}[Matrix Ellipsoid]
\label{def:mat-ell}
We define the ellipsoid induced by $X$ to be $E(X) = \{\vec{y} \in \ran(X):
\vec{y}^\T X^+ \vec{y} \leq 1\} = X^{1/2} \Ball_2^n$.
\end{definition}

\begin{definition}[Matrix Projection]
\label{def:mat-proj} 
We define the projection of $X$ to $W$ by $X^{\da W} \eqdef \pi_W X \pi_W$.
Here we have that $E(X^{\da W}) = \pi_W E(X)$, which justifies the term
matrix projection.
\end{definition}

\begin{definition}[Matrix Slice]
\label{def:schur-compl}
We denote $X^{\cap W}$, the slice of $X$ on $W$, to be the unique PSD matrix
satisfying
\[
\vec{y}^\T(X^{\cap W})\vec{y} = \min_{\vec{w} \in W^\perp} (\vec{y}+\vec{w})^\T X
(\vec{y}+\vec{w}), \quad \forall \vec{y} \in \R^n \text{ .}
\]
Here we have that $E(X^{\cap W}) = E(X) \cap W$, which justifies the term matrix
slice. 

If $X = \begin{pmatrix} A & C \\ C^\T & B \end{pmatrix} \in \PSD^n$, $A \in
\R^{k \times k}$, $C \in \R^{k \times (n-k)}$, $B \in \R^{(n-k) \times (n-k)}$,
then the slice $X$ on $W = \R^{k} \times 0^{n-k}$ is the Schur complement of $X$
with respect $B$ (lifted to live in the full space), that is
\[
X^{\cap W} = \begin{pmatrix} A - C B^+ C^\T & 0^{k \times (n-k)}
\\ 0^{(n-k) \times k} & 0^{(n-k) \times (n-k)} \end{pmatrix} \text{ .}  
\]
\end{definition}

\begin{definition}[Projected Determinant] 
\label{def:rest-det}
Define the projected determinant of $X$ on $W$ by
\begin{equation}
\det_W(X) \eqdef \det(O_W^\T X O_W)
\end{equation}
where $O_W$ is any matrix whose columns form an orthonormal basis of $W$. Note
that the projected determinant is invariant to the choice of orthornormal
basis. If $\dim(W) = d$, then we have that 
\[
\vol_d(\pi_W E(X)) = \vol_d(\Ball_2^d) \det_W(X)^{1/2} \text{.}
\]
\end{definition}

\begin{claim}\label{clm:detreverse}
If $B$ is a matrix and $O_V$ is a matrix whose columns form an orthonormal basis of $V = \ran(B)$, then 
\begin{align*}
   \det(B^\T B) = \det(B^\T O_V O_V^\T B) = \det(O_V^\T B B^\T O_V) = \det_V(B B^\T) \; .
\end{align*}
\end{claim}

The following lemma provides some important properties of the matrix projection
and slice.

\begin{lemma} 
\label{lem:slice}
Let $X \in \PSD^n$ and $W \subseteq \ran(X)$. Then
the following holds: 
\begin{enumerate}
\item $(X^{\cap W})^+ = (X^+)^{\da W}$.
\item If $Y \in \R^{n \times n}$ symmetric satisfies $Y \preceq X$ and $\ran(Y) \subseteq W$, then $Y \preceq X^{\cap W}$.
\item For $Y \succeq 0$, $X^{\cap W} + Y^{\cap W} \preceq (X+Y)^{\cap W}$.
\end{enumerate}
\end{lemma}

The following lemma shows that the restricted determinant interacts nicely with
the matrix slice and matrix projection.

\begin{lemma}[Matrix Determinant Lemma]
\label{lem:det} 
Let $X \in \PSD^n$, $W_1 \perp W_2$ satisfying $\ran(X) = W_1 + W_2$. Then
\[
\det_{W_1+W_2}(X) = \det_{W_1}(X^{\cap W_1}) \det_{W_2}(X^{\da W_2}) \text{.}
\]
\end{lemma}

The following lemma gives the basic concavity properties and dual
characterization of the projected determinant:

\begin{lemma}[Properties of Projected Determinants]
\label{lem:proj-det-props}
Let $X \in \PSD^n$ and $W \subseteq \R^n$ a $d$-dimensional subspace.
Then the following holds:
\begin{enumerate}
\item For $t \in \R_{\geq 0}$, $\det_W(tX) = t^d \det(X)$.
\item $\det_W(X)^{1/d} 
= \inf \set{\frac{1}{d} \tr(XZ): Z \in \PSD^n, \ran(Z) = W, \det_W(Z)=1}$. \\
In particular, $\det_W(X)^{1/d}$ is concave in $X$.
\item If $W \cap \ker(X) = \{\vec 0\}$, then for $\Delta \in \R^{n \times n}$, 
\[
\frac{{\rm d}}{{\rm dt}} \det_W(X+t\Delta)^{1/d} \big|_{t=0} 
=  \frac{1}{d} \tr(\det_W(X)^{1/d}(X^{\da W})^+ \Delta) \text{.} 
\]
\item If $W \cap \ker(X) = \{\vec 0\}$, $\det_W(X) \det_W((X^{\da W})^+) = 1$.
\label{itm:detprod}
\end{enumerate}
\end{lemma}

\subsection{Probability}

For two probability measures $\mu_1,\mu_2$ on $\Omega$, we define the
\emph{total variation distance} between $\mu_1$ and $\mu_2$ by
$\Delta(\mu_1,\mu_2) = \sup \set{|\mu_1(S)-\mu_2(S)|: S \subseteq \Omega, S
\text{ measurable }}$. If $\mu_1$ is absolutely continuous with respect to
$\mu_2$ then
\[
\Delta(\mu_1,\mu_2) = (1/2) \int_{\Omega}
|(d\mu_1/d\mu_2(\vec{x})-1)|d\mu_2(\vec{x}) \text{.}
\]

\paragraph{{\bf Gaussian distribution.}}
For a positive definite matrix $A \in \PSD^n$ and $\vec{c} \in \R^n$, we define
the Gaussian distribution $N(\vec{c},A)$ with covariance $A$ and mean $\vec{c}$
to be the distribution with density
\begin{equation}
\label{eq:gauss-dist}
\frac{1}{\det(A)^{1/2} \sqrt{2\pi}^n} 
\exp\parens[\Big]{-\frac{1}{2}(\vec{x}-\vec{c})A^{-1}(\vec{x}-\vec{c})}, 
\forall \vec{x} \in \R^n \text{ .} 
\end{equation} 

For $X \sim N(\vec{c},A)$, by construction we have that $\E[X] = \vec{c}$ and
$\E[(X-\vec{c})(X-\vec{c})^\T] = A$.
We shall say $n$-dimensional standard Gaussian to denote $N(0,I_n)$.

\subsection{Lattices}

A $d$-dimensional Euclidean lattice $\lat \subset \R^n$ is defined as the set of
all integer linear combinations of $d$ linearly independent vectors $B =
(\vec{b}_1,\dots,\vec{b}_d)$ in $\R^n$, where we call $B$ a basis for $\lat$.
If $d=n$ we say that the
lattice is full rank.  

The \emph{determinant} of $\lat$ is defined as $\det(\lat) = \sqrt{\det(B^\T B)}$,
which is invariant to the choice of basis for $\lat$. 
Notice that by
Claim~\ref{clm:detreverse} we can also write 
$\det(\lat) = \sqrt{\det_V(B B^\T)}$ where $V = {\rm span}(\lat)$.
The following elementary 
claim shows how the determinant changes under linear transformation.

\begin{claim}\label{clm:detlattransform}
Let $\lat$ be a lattice and $W = {\rm span}(\lat)$.
Then for a linear transformation $A$,
\[
\det(A \lat)
= \det_W(A^\T A)^{1/2} \det(\lat) \text{ .}
\]
\end{claim}
\begin{proof}
Let $B$ be a basis of $\lat$. Write $B = O_W \bar{B}$ where
the columns of $O_W$ are an orthonormal basis of $W$ and $\bar{B}$ is ${\rm dim}(W) \times {\rm dim}(W)$. Then,
\begin{align*}
\det(A \lat)^2 &= 
\det(B^\T A^\T A B) = 
\det(\bar{B}^\T O_W^\T A^\T A O_W \bar{B}) = 
\det(O_W^\T A^\T A O_W)\det(\bar{B}^\T \bar{B}) \\
&= \det_W(A^\T A)\det(B^\T B) = \det_W(A^\T A) \det(\lat)^2 \text{ ,}
\end{align*}
as claimed.
\end{proof}

The dual lattice of $\lat$ is $\lat^* = \set{\vec{y} \in
{\rm span}(\lat): \pr{\vec{y}}{\vec{x}} \in \Z, \forall \vec{x} \in \lat}$. It
is easy to verify that $B(B^\T B)^{-1}$ yields a basis for $\lat^*$ and that
$\det(\lat^*) = 1/\det(\lat)$.  We say that a subspace $W
\subseteq \R^n$ is a \emph{lattice subspace} of $\lat \subset \R^n$ if $W$
admits a basis of vectors in $\lat$.  We will need a few 
important facts about lattice subspaces and projections. Firstly, the projection
$\pi_W(\lat)$ onto a subspace $W \subseteq \R^n$ is a lattice (i.e., discrete) if
and only if $W$ is a lattice subspace of $\lat^*$. Furthermore, $\pi_W(\lat)^* =
\lat^* \cap W$. Secondly, for two lattice subspaces $V,W$ of $\lat$, both the
intersection $V \cap W$ and the sum $V+W$ is a lattice subspace of $\lat$.  

Let $K \subseteq \R^n$ be a symmetric convex body and $\lat \subseteq \R^n$ as
above, be a $d$-dimensional lattice. 

\begin{definition}[Successive Minima] For $i \in [d]$, we define the $i^{th}$
successive minima of $\lat$ with respect to $K$ by 
\label{def:suc-min}
\[
\lambda_i(K,\lat) = \min \set{r \geq 0: \dim(rK \cap \lat) \geq i}
\text{.}
\]
We write $\lambda_i(\lat)$ to denote $\lambda_i(\Ball_2^n,\lat)$.
\end{definition}

\begin{definition}[Voronoi Cell]
\label{def:voronoi-cell}
For a lattice $\lat \subseteq \R^n$, we define the \emph{Voronoi cell} by
\[
\calV(\lat) = \set[\Big]{\vec{x} \in {\rm span}(\lat): \pr{\vec{x}}{\vec{y}} 
\leq \frac{1}{2}\pr{\vec{y}}{\vec{y}} \quad \forall \vec{y} \in \lat \setminuszero}
\text{ .}
\]
In words, $\calV(\lat)$ corresponds to all the points in ${\rm span}(\lat)$ that
are closer to $\vec{0}$ than to any other point of $\lat$. We note that $\calV$
is a centrally symmetric polytope (and hence convex) which tiles ${\rm
span}(\lat)$ with respect to $\lat$.
\end{definition}

\begin{definition}[Covering Radius] The covering radius of $\lat$ with respect
to $K$, for $K$ a (not necessarily symmetric) convex body, is
\[
\mu(K,\lat) = \inf \set{r \geq 0: {\rm span}(\lat) \subseteq \lat + rK} \text{ .}
\]
Note that replacing $K$ by $K \cap {\rm span}(\lat)$ does
not change the covering radius. We write $\mu(\lat)$ to denote
$\mu(\Ball_2^n,\lat)$. 
\end{definition}

\begin{definition}[Metric on the Torus]
\label{def:toric-metric}
For $\vec{x},\vec{y} \in {\rm span}(\lat) / \lat$, define
$\dist(\vec{x},\vec{y}) = \min_{\vec{z} \in \lat+(\vec{y}-\vec{x})}
\|\vec{z}\|_2$. Note that $\dist$ is in fact a metric. Furthermore, $\mu(\lat) =
\max_{\vec{x} \in {\rm span}(\lat)/\lat} \dist(\vec{x},\lat)$, is the diameter
of the torus under this metric.
\end{definition}

The following is an easy and useful fact about the covering radius, already observed in~\cite{GuruswamiMR05}.

\begin{claim}\label{clm:gmrcovering}
For a lattice $\lat \subset \R^n$ and $\vec x$ a uniform point in ${\rm
span}(\lat) /\lat$, 
\begin{align}
\label{eq:gmrclaim}
\Pr[\dist(\vec{x},\lat) \ge \mu(\lat)/2] \ge 1/2 \text{.} 	
\end{align}
\end{claim}
\begin{proof}
Assume towards contradiction that~\eqref{eq:gmrclaim} does not hold.  Let $\vec
y$ be an arbitrary point in ${\rm span}(\lat)/\lat$ and $\vec x$ a uniform point
as above.  Since $\vec x - \vec y$ and $\vec x$ are both uniformly distributed
in ${\rm span}(\lat)/\lat$, we have that with positive probability, both
$\dist(\vec y, \vec x) < \mu(\lat)/2$ and $\dist(\vec x, \lat) < \mu(\lat)/2$
hold. But this implies by triangle inequality that $\dist(\vec y, \lat) <
\mu(\lat)$. Since $\vec y$ is arbitrary, this contradicts the definition of
covering radius.  
\end{proof}

We recall
the following fundamental theorem of Minkowski. The goal of the present work is
to find partial converses to this theorem.

\begin{theorem}[Minkowski's First Theorem, {see~\cite[Chapter 2]{GruberL87}}] 
\label{thm:mink-first}
For an $n$-dimensional symmetric convex body $K$ and lattice $\lat$ in $\R^n$, 
\[
|K \cap \lat| \geq \ceil{ 2^{-n} \vol_n(K) / \det(\lat) } \text{ .}
\]
\end{theorem}

\begin{theorem}[Minkowski's Second Theorem, {see~\cite[Chapter 2]{GruberL87}}] 
\label{thm:mink-sec}
For an $n$-dimensional symmetric convex body $K$ and lattice $\lat$ in $\R^n$, 
\[
\prod_{i=1}^n \lambda_i(K, \lat) \leq \frac{2^n}{\vol_n(K)} \det(\lat) \leq 
n! \prod_{i=1}^n \lambda_i(K, \lat) \text{ .}
\]
Furthermore, if $K = \Ball_2^n$, the $n!$ on the right-hand side can be replaced
by $2^n \vol_n(\Ball_2^n)^{-1} = ((1+o(1)) \frac{2n}{\pi e})^{n/2}$.
\end{theorem}

We will need the following bound due to Henk, which bounds the number of lattice
points in any scaling of a symmetric convex body in terms of the successive
minima.

\begin{theorem}[Henk's Bound~\cite{Henk02}]
\label{thm:henk-bnd}
For an $n$-dimensional symmetric convex body $K$ and lattice $\lat$ in $\R^n$, and 
$t \geq 0$,
\[
|tK \cap \lat| \leq 2^{n-1} \floor[\Big]{ 1 + \frac{2t}{\lambda_i(K,\lat)}}
\text{ .}
\] 
\end{theorem}

\subsection{The smoothing parameter}

For a positive definite $X \succ 0$ and a countable set $T \subseteq
\R^n$, define
\begin{equation*}
\rho_X(T) = \sum_{\vec{y} \in T} e^{-\pi \vec{y}^\T X^{-1} \vec{y}} \text{.}
\end{equation*}
We extend this to positive semidefinite $X \succeq 0$ by
\begin{equation*}
\rho_X(T) = \sum_{\vec{y} \in T \cap \ran(X)} e^{-\pi \vec{y}^\T X^+ \vec{y}} \text{.}
\end{equation*}
Note that with the above definition, $\rho_X(T)$ is a continuous function
over all of $\PSD^n$.

We remark that the above notation is slightly non-standard, in that we
parametrize the $\rho$ with respect to $X$ and not $X^{1/2}$. This notation will
however be more convenient for us.
For $s>0$, we will often denote $\rho_{s^2 I}$ by $\rho_{s^2}$. 

We define $\eta_{\eps}(\lat)$ the $\eps$-smoothing parameter of $\lat$ as the
unique $s > 0$ satisfying
\[
\rho_{1/s^2}(\lat^*) = \sum_{\vec{y} \in \lat^*} e^{-\pi\|s\vec{y}\|^2} = 1+\eps 
\text{ .}
\]
We will simply say the smoothing parameter of $\lat$ to denote $\eta(\lat)
\eqdef \eta_{1/2}(\lat)$.

We first recall the following fundamental bound on the smoothing parameter,
which is implicit in the work of Banaszczyk~\cite{Bana93} (see for
example~\cite[Lemma 2.17]{cvpp} for a short self-contained proof). 

\begin{theorem}
\label{thm:smoothing-bnd}
For an $n$-dimensional lattice $\lat \subset \R^n$ and $\eps \in (0,1)$,
\[
\frac{\sqrt{\log(2/\epsilon)/\pi}}{\lambda_1(\lat^*)} \leq
\eta_\epsilon(\lat) \leq
\frac{\sqrt{\log((1+\eps)/\eps)/\pi}+\sqrt{n/(2\pi)}}{\lambda_1(\lat^*)}
\text{.}
\]
In particular, for $\eps \geq 2^{-n}$, $\eta_\eps(\lat) \leq \eta_{2^{-n}}(\lat)
\leq \sqrt{n}/\lambda_1(\lat^*)$.
\end{theorem}


We now list some useful properties of the Gaussian function.  The following lemma
implies that we could choose any constant instead of $1/2$ in the definition of
the smoothing parameter while only affecting it by a constant factor. 
\begin{lemma}
\label{lem:mass-decrease}
Let $\lat \subset \R^n$ be a lattice. Then for any $A \in \PSD^n$, $\lat
\subset \ran(A)$ and $t \geq 1$, the following holds:
\begin{enumerate}
\item~\cite{KaiMinDLP13}: $\rho_{A/t^2}(\lat \setminuszero) \leq \rho_A(\lat \setminuszero)^{t^2}$.
\item $\rho_{A/t^2}(\lat \setminuszero) \leq \rho_A(\lat
\setminuszero)/\floor{t}$.
\end{enumerate}
\end{lemma}
\begin{proof} 
For the first part, since $t \geq 1$,
\begin{align*}
\rho_{A/t^2}(\lat \setminuszero) 
&= \sum_{\vec{y} \in \lat \setminuszero} e^{-\pi \vec{y}^\T (A/t^2)^+ \vec{y}} 
= \sum_{\vec{y} \in \lat \setminuszero} (e^{-\pi \vec{y}^\T A^+ \vec{y}})^{t^2} \\
&\leq \left(\sum_{\vec{y} \in \lat \setminuszero} e^{-\pi \vec{y}^\T A^+
\vec{y}}\right)^{t^2} = \rho_A(\lat \setminuszero)^{t^2} \text{, }
\end{align*}
as needed.

We now prove the second part. We first factor $t$ into the lattice,
that is
\[
\rho_{A/t^2}(\lat \setminuszero) = \rho_{A}(t \lat \setminuszero)
\leq \rho_{A}(\floor{t} \lat \setminuszero)
\text{.}
\]
Let $k = \floor{t}$. From here, we decompose the lattice sum into sums
over $1$-dimensional sublattices,
\begin{equation}
\label{eq:dec-line1}
\rho_A(k \lat \setminuszero) = \sum_{\substack{\vec{y} \in \lat \setminuszero \\ \vec{y}
\text{ primitive }}} \sum_{z=1}^\infty e^{-\pi (kz)^2 \vec{y}^\T A^+ \vec{y}}
\text{.}
\end{equation}
Given the above formula, it suffices to show the inequality over each line,
namely showing that for any $r > 0$,
\begin{equation}
\label{eq:dec-line2}
k \sum_{z=1}^\infty e^{-\pi (kz)^2 r} \leq \sum_{z=1}^\infty e^{-\pi z^2 r}
\text{.}
\end{equation}
To see this, note that it would imply
\[
k \sum_{\substack{\vec{y} \in \lat \setminuszero \\ \vec{y}
\text{ primitive }}} \sum_{z=1}^\infty e^{-\pi (kz)^2 \vec{y}^\T A^+ \vec{y}}
\leq \sum_{\substack{\vec{y} \in \lat \setminuszero \\ \vec{y}
\text{ primitive }}} \sum_{z=1}^\infty e^{-\pi z^2 \vec{y}^\T A^+ \vec{y}}
= \rho_A(\lat \setminuszero) \text{, }
\]
which combined with~\eqref{eq:dec-line1} proves the claim.
To prove~\eqref{eq:dec-line2}, a direct computation reveals
\[
\sum_{z=1}^\infty e^{-\pi z^2 r}
= \sum_{z=1}^\infty \sum_{s=0}^{k-1} e^{-\pi (kz-s)^2 r}
\leq \sum_{z=1}^ \infty \sum_{s=0}^{k-1} e^{-\pi (kz)^2 r} 
= k \sum_{z=1}^\infty e^{-\pi (kz)^2 r} \text{, }
\]
as needed.
\end{proof}

We recall that for any ``nice enough'' function $f: \R^n \rightarrow \R$ and any
$n$-dimensional lattice $\lat$, the Poisson summation formula gives
\begin{equation}
\label{eq:poisson-summation}
\sum_{\vec{y} \in \lat} f(\vec{y}) = \frac{1}{\det(\lat)} \sum_{\vec{y} \in \lat^*} \hat{f}(\vec{y}) \; ,
\end{equation}
where $\hat{f}(\vec{y}) = \int_{\R^n} e^{-2\pi i \pr{\vec{x}}{\vec{y}}} f(\vec{x}) d\vec{x}$ is the Fourier
transform of $f$.

The next lemma follows directly from the Poisson summation formula. 
\dnote{5/18: It is not entirely obvious why the matrices we get below are
correct. Maybe it could use some justification.} 

\begin{lemma}[Gaussian Mass of Cosets]
\label{lem:coset-mass}
Let $\lat \subset \R^n$ be a $d$-dimensional lattice, $W = {\rm span}(\lat)$,
$t \in W$, and $A \in \PSD^n$ with $W \subseteq \ran(A)$. Then
\begin{equation}
\rho_A(\lat + \vec{t}) = \frac{\det_W(A^{\cap W})^{1/2}}{\det(\lat)} 
\sum_{\vec{y} \in \lat^*} e^{2\pi i
\pr{\vec{t}}{\vec{y}}} e^{-\pi \vec{y}^\T A^{\cap W} \vec{y}} \text{.}
\end{equation}
In particular, $\rho_A(\lat) = \frac{\det_W(A^{\cap W})^{1/2}}{\det(\lat)}
\rho_{(A^{\cap W})^+}(\lat^*)$. Furthermore, if $\rho_{(A^{\cap W})^+}(\lat^*
\setminuszero) \leq \eps$,
$\eps \in (0,1)$, then 
\[
\rho_A(\lat + \vec{t}) \in [1-\eps,1+\eps] \cdot \frac{\det_W(A^{\cap
W})^{1/2}}{\det(\lat)} \text{ .}
\]
\end{lemma} 

The following claim shows in what sense a Gaussian is ``smooth'' modulo a
lattice. It is a simple extension of~\cite[Lemma 4.1]{MR04} to non-spherical
covariances.

\begin{claim}\label{clm:smoothing}
Let $\lat \subset \R^n$ be a full-rank lattice, and $A \succ 0$ be such that 
$\rho_{A^{-1}}(\lat^* \setminuszero) = \sum_{\vec{y} \in \lat^* \setminuszero} e^{-\pi \vec{y}^\T A \vec{y}} \leq \eps
$
for some $\eps>0$.
Let $X \sim N(0,\frac{1}{2\pi} A)$ be a centered Gaussian random variable in
$\R^n$ with mean $0$ and covariance $\frac{1}{2\pi} A$. 
Then, the total variation distance between $X \bmod \lat$ and the uniform distribution satisfies 
\[
\Delta(X \bmod \lat, U) \leq \eps/2 \; .
\]
\end{claim}
\begin{proof}
Let $\mu_g$ and $\mu$ be the probability measures corresponding to $X \bmod \lat$ and $U$ respectively. 
By Lemma~\ref{lem:coset-mass}, we have that the statistical distance between
the two distributions is
\begin{align*}
 1/2 \int_{\R^n/\lat} |(d\mu_g/d\mu)(\vec{t})-1| d\mu(\vec{t}) 
 &= 1/2 \int_{\R^n/\lat}
\abs[\Big]{\frac{\det(\lat)}{\det(A)^{1/2}} \rho_{A}(\lat+\vec{t})-1} d\mu(\vec{t}) \\
 &= 1/2 \int_{\R^n/\lat} \abs[\Big]{\sum_{\vec{y} \in \lat^* \setminuszero}
e^{2\pi i \pr{\vec{t}}{\vec{y}}} e^{-\pi\vec{y}^\T A\vec{y}}} d\mu(\vec{t}) \\
 &\leq 1/2 \int_{\R^n/\lat} \rho_{A^+}(\lat^* \setminuszero)
d\mu(\vec{t}) \\
&\leq 1/2 \int_{\R^n/\lat} \eps \, d\mu(\vec{t}) = \eps/2
\text{ .} 
\end{align*}
\end{proof}

\begin{lemma}[\cite{Bana93}]
\label{lem:banacosh}
Let $\lat\subset\R^n$ be a lattice of rank $n$. Then, for all $\vec{t} \in \R^n$, 
$
\rho(\lat+\vec{t}) \geq \rho(\vec{t})\rho(\lat)
$.
\end{lemma}
\begin{proof}
\[ \rho(\lat + \vec{t}) =  \rho(\vec{t})\sum_{\vec{y} \in \lat} \cosh(2\pi \inner{\vec{y}, \vec{t}})\rho(\vec{y}) \geq \rho(\vec{t}) \rho(\lat)
\; . \qedhere
\] 
\end{proof}

\subsection{Random lattices}

\onote{give the beast a name?}\dnote{5/18: Siegel's measure on lattices?}
The set of lattices of determinant $1$, which can be identified with the
quotient ${\rm SL}(\R,n)/{\rm SL}(\Z,n)$, has a natural and useful probability measure 
defined on it (see~\cite[Chapter 3]{GruberL87} or~\cite{TerrasBook}). Originally introduced by Siegel~\cite{Siegel45}, it has the following remarkable
properties.

\begin{theorem}[\cite{Siegel45}] 
\label{thm:haar}
Let $\lat$ be an $n$-dimensional lattice distributed as above. Then, the following holds:
\begin{enumerate}
\item For any linear transformation $T$ of determinant $\pm 1$, 
$T \lat $ and $\lat$ are identically distributed.
\item $\lat$ and $\lat^*$ are identically distributed.
\item Let $\R^{n \times d}_{{\rm ind}} = \set{(\vec{x}_1,\dots,\vec{x}_d) \in
(\R^n)^d: \vec{x}_1,\dots,\vec{x}_d \text{ linearly independent}}$. Then, for $d \leq
n-1$ and any measurable subset $A \subseteq \R^{n \times d}_{{\rm ind}}$, we
have that
\[
\E[|\lat^d \cap A|] = \vol_{d n}(A) \text{ .}
\]
\end{enumerate}
\end{theorem}

Here we will only need the case $d=1$ of the last item above,
which says that $\E[|(\lat \setminuszero) \cap A|] = \vol_{n}(A)$. 
By integrating this equality over level sets, we get that for any
Riemann integrable function
$f:\R^n \to \R$, 
\begin{align}\label{eq:siegelforfunctions}
\E\bracks[\Big]{\sum_{x \in \lat \setminuszero}f(x)} = \int f(x) dx	\; .
\end{align}

We note that this distribution can be obtained as the ``limit'' of the
following simple discrete distributions. For $p \in \N$ a large prime, sample
$\vec{a} \in (\Z/p\Z)^n$ uniformly at random, and let $\lat = p^{-1/n}
\set{\vec{x} \in \Z^n: \pr{\vec{a}}{\vec{x}} \equiv 0 \imod p}$. Note that
$\lat$ has determinant $1$ as long as $\vec{a} \neq 0$. It was shown by
\cite{GM03}, that as $p \rightarrow \infty$ this distribution converges in a
strong sense to the distribution of random lattices.

\section{The Main Conjecture}
\label{sec:mainconj}

We now formally state the main conjecture. 

\begin{conjecture}\label{con:maineta}(Main conjecture)
Let $C_\eta(n)>0$ be the smallest number such that for any $\lat \subset \R^n$,
\begin{align}\label{eq:finalgoal}
   \eta(\lat) \le C_\eta(n) \max_{W \neq \set{\vec{0}} \textnormal{ lattice subspace
of } \lat^*} (\det(\lat^* \cap W))^{-1/\dim(W)}.
\end{align}
Then $C_\eta(n) \le \poly \log n$.
\end{conjecture}

Note that by homogeneity, to prove Conjecture~\ref{con:maineta} it suffices to
show that $\eta(\lat) \le \poly \log n$ whenever all sublattices of $\lat^*$
have determinant at least $1$. We now show an equivalence between the smoothing
parameter and a bound on the number of dual lattice points at distance $r$ of
the form $e^{(sr)^2}$. This formalizes the equivalence between the above
conjecture and Conjecture~\ref{con:mainetaintro}.

\begin{lemma}\label{lem:pointcountingeta}
For an $n$-dimensional lattice $\lat$, the following inequality
holds:
\[
\eta(\lat)/\sqrt{3} \leq \max_{r > 0} \frac{\sqrt{\log(|\lat^* \cap
rB_2^n|)/\pi}}{r} \leq \eta(\lat) \text{ .}
\]
\end{lemma}
\begin{proof}
By rearranging, note that $\tilde{s} = \max_{r > 0} \frac{\sqrt{\log(|\lat^* \cap
rB_2^n|)/\pi}}{r}$ is simply the minimum number such that $|\lat^* \cap rB_2^n|
\leq e^{\pi(\tilde{s}r)^2}$ for all $r \geq 0$. 

We first show that $\tilde{s} \leq s$, where $s = \eta(\lat)$. In particular, we
must show that $|\lat^* \cap rB_2^n| \leq e^{\pi(sr)^2}$ for all $r \geq 0$. 
By definition, $\rho_{1/s^2}(\lat^* \setminuszero) = 1/2$, and thus we must have
that
\[
1/2 \geq e^{-\pi (sr)^2} (|\lat^* \cap rB_2^n|-1) \text{.}
\]
Rearranging, and using the fact that the cardinality of a set is integer,
\[
|\lat^* \cap rB_2^n| \le \floor{1 + e^{\pi (sr)^2}/2} \le e^{\pi (sr)^2} \text{,}
\]
as desired. 

We now show that $s \leq \sqrt{3} \tilde{s}$. In particular, we show that
$\rho_{1/(3 \tilde{s}^2)}(\lat^*) \leq 3/2$. The following computation derives
the bound,
\begin{align*}
\rho_{1/(3\tilde{s}^2)}(\lat^*) 
&= \sum_{\vec{y} \in \lat^*} e^{-\pi(3\tilde{s}^2) \|\vec{y}\|_2^2} \\
&= \sum_{\vec{y} \in \lat^*} \int_{0}^\infty I[\|\vec{y}\|_2 \leq r]
(6\pi\tilde{s}^2) \cdot r e^{-\pi(3\tilde{s}^2) r^2} dr \\
&= \int_{0}^\infty |\lat^* \cap rB_2^n| 
(6\pi\tilde{s}^2) \cdot r e^{-\pi(3\tilde{s}^2) r^2} dr \\
&\leq \int_{0}^\infty e^{\pi (\tilde{s} r)^2} 
(6\pi\tilde{s}^2) \cdot r e^{-\pi(3\tilde{s}^2) r^2} dr \\
&= \int_{0}^\infty (6\pi\tilde{s}^2) \cdot r e^{-\pi(2\tilde{s}^2) r^2} dr \\
&= -6/4e^{-\pi(2\tilde{s}^2)r^2} \big|_{r=0}^\infty = 3/2 \text{, }
\end{align*}
as needed.
\end{proof}

\section{The Main Conjecture implies the Kannan-Lov{\'a}sz Conjecture}
\label{sec:klmain}

We start by stating the $\ell_2$ KL conjecture. 

\begin{conjecture}[The $\ell_2$ Kannan-Lov{\'a}sz Conjecture]
\label{conj:kl-l2}
Let $C_{KL}(n)>0$ be the smallest number such that for any $\lat \subset \R^n$,
\begin{align}\label{eq:conjkll2}
\mu(\lat) 
\leq 
C_{KL}(n) 
\max_{\substack{W \textnormal{ lattice subspace
of } \lat^* \\ 1 \leq d =
\dim(W) \leq n}} \sqrt{d} \det(\lat^* \cap W)^{-1/d}
\; .
\end{align}
Then $C_{KL}(n) \le \poly \log n$.
\end{conjecture}

We note that the reverse direction is easy to prove, i.e., that the $\max$ in Eq.~\eqref{eq:conjkll2} is $\lesssim \mu(\lat)$. 
We also note that the best known lower bound on $C_{KL}(n)$
is $O(\sqrt{\log n})$, obtained by the lattice generated by the basis
$B=(\vec{e}_1,\vec{e}_2/\sqrt{2},\dots,\vec{e}_n/\sqrt{n})$.
With the above formulation, we can state the main result of this paper as
follows.

\begin{theorem}
\label{thm:mink-to-kl}
The $\ell_2$ Kannan-Lov{\'a}sz conjecture holds with bound $C_{KL}(n) = O(\log n) C_\eta(n)$.
\end{theorem}

In Section~\ref{sec:klorigin} we provide some optional background on the KL conjecture. 
The rest of this section is dedicated to the proof of Theorem~\ref{thm:mink-to-kl}.

\subsection{Origin of the Kannan-Lov{\'a}sz Conjecture}
\label{sec:klorigin}

The basic question Kannan and Lov{\'a}sz~\cite{KL88} sought to understand is
this: when can we guarantee that a convex body contains a lattice point?  One of
the most elegant ways to do so is to examine the \emph{covering radius} of the
body with respect to the lattice. Given a convex body $K$ and $n$-dimensional
lattice $\lat$ in $\R^n$, we define the covering radius
\[
\mu(K,\lat) = \inf \set{s \geq 0: \lat + sK = \R^n} \text{ ,}
\]
or equivalently, the minimum scaling $s$ of $K$ for which $sK+\vec{t}$ contains
a point of $\lat$ for every translation $\vec{t} \in \R^n$. Note that if
$\mu(K,\lat) \leq 1$, then $K$ contains a lattice point in \emph{every
translation}.

With this definition, we may specialize the main question to: when is the covering
radius of a convex body smaller than $1$? or more generally, is there a good
alternate ``dual'' characterization of the covering radius? 

Good answers to this question have played a crucial role in the development of
faster algorithms for the Integer Programming
(IP)~\cite{Lenstra83,Kan87,HildebrandK10,DPV11,thesis/D12}. The IP
problem, classically defined as deciding whether a linear system of inequalities
$Ax \leq b$ admits an integer solution, can also be phrased more generally, that
is, given a convex body $K$ and lattice $\lat$ in $\R^n$ decide whether $K \cap
\lat \neq \emptyset$. The fastest algorithms for IP require
$n^{O(n)}$-time~\cite{Kan87,thesis/D12}, and a major open problem is whether
there exists a $2^{O(n)}$-time algorithm for IP.

A first satisfactory duality theorem in this context is Khinchine's Flatness
theorem, which states that either $\mu(K,\lat) \leq 1$ or $K$ has small
\emph{lattice width}, i.e.,~$K$ is ``flat.'' 
Letting the width norm of $K$ be ${\rm width}_K(\vec{z}) = \max_{\vec{x} \in K}
\pr{\vec{z}}{\vec{x}} - \min_{\vec{x} \in K} \pr{\vec{z}}{\vec{x}}$, for
$\vec{z} \in \R^n$, we define the lattice width of $K$ w.r.t. to $\lat$ as 
\[
{\rm width}(K, \lat) = \min_{y \in \lat^* \setminuszero} {\rm
width}_K(\vec{y}) \text{.}
\]
Improving the quantitative estimates
on how ``flat'' $K$ must be has been a focus of much
research~\cite{Khinchine48,Bab86,KL88,LLS90, Bana93,Bana96,Bana99}.
The best current estimate on flatness can be stated as follows.

\begin{theorem}[Khinchine's Flatness Theorem] 
\[
1 \leq \mu(K,\lat) {\rm width}(K,\lat) \leq \tilde{O}(n^{4/3}) \text{ .}
\]
\end{theorem}

The bound above is derived by combining the work of Banaszczyk~\cite{Bana96},
who showed that expression in the middle can be bounded by $O(n)$ times the so-called
$M$-$M^*$ estimate for $K$, together with the $\tilde{O}(n^{1/3})$ bound on the
$M$-$M^*$ estimate for general convex bodies due to Rudelson~\cite{Rudelson00}.

From the perspective of IP, if $\mu(K,\lat) \leq 1/2$, one can in fact find a
feasible lattice point using an appropriate rounding algorithm, and if
$\mu(K,\lat) \geq 1/2$, the flatness theorem implies that we can decompose $K$
along $\tilde{O}(n^{4/3})$ lattice hyperplanes (i.e.,~slicing along a flatness
direction). Since the rounding step can in fact be implemented in
$2^{O(n)}$-time~\cite{cvp/D12,thesis/D12}, the main bottleneck is in finding
more efficient decompositions. Note that creating $n^{O(1)}$ subproblems per
dimension yields a branching tree of size $n^{O(n)}$, the dominant term in the
complexity of IP algorithms. 

The flatness estimate can in fact be improved to $O(n)$ for
ellipsoids~\cite{Bana93} and $O(n\log n)$ for centrally symmetric convex
bodies~\cite{Bana96}. However, for any convex body, it is 
known (see, e.g.,~\cite[Remark 3.10]{Bana95}\footnote{As stated, the remark is for
symmetric convex bodies. To deal with an asymmetric body $K$, simply replace $K$ by $K-K$.\dnote{Full derivation: $\mu(K,\lat) \lambda_1((K-K)^*, \lat^*) \geq
\lambda_1((K-K),\lat) \lambda_1((K-K)^*,\lat^*)$. From here by Seigel, for a
random lattice $\lat$ of det $1$, $\lambda_1((K-K),\lat)\lambda_1((K-K)^*,\lat) \geq
(\vol_n(K-K)\vol_n(K-K)^*)^{-1/n} = \Omega(n)$ by Santalo's inequality. Should
we add this somewhere?}})  
that there exists a lattice for which the right-hand
side is $\Omega(n)$, in particular a random lattice achieves this with high
probability, and hence the relationship between lattice width and the covering
radius cannot be made sub-polynomial in general. Thus, it is doubtful
whether one can significantly improve on current IP algorithms by decomposing
along hyperplanes. 

To try and circumvent this problem, Kannan and Lov{\'a}sz~\cite{KL88} defined a
multidimensional generalization of flatness, proving the following bounds:

\begin{theorem}[\cite{KL88}]
\label{thm:klgeneral}
\begin{equation}
\label{kl-bnd}
1 \leq \mu(K,\lat) \min_{\substack{W \textnormal{ lattice subspace
of } \lat^* \\ 1 \leq d =
\dim(W) \leq n}} \vol_d(\pi_W(K))^{1/d} \det(\lat^* \cap W)^{1/d} \leq n
\end{equation}
where $\pi_W$ is the orthogonal projection onto $W$. Furthermore, for $K =
\Ball_2^n$, the right-hand side can be improved to $\sqrt{n}$.
\end{theorem}

We note that the standard flatness theorem corresponds to forcing $d=1$ above.
While not obvious, the above theorem (whose proof is constructive) yields a way
of decomposing an IP along $O(n)^d$ shifts of some $n-d$ dimensional subspace
(note that $d$ is not fixed). This was used in~\cite{thesis/D12} to give the first
$n^{n+o(n)}$ algorithm for IP. 

Currently, there are no known examples for which the right-hand side
of~\eqref{kl-bnd} is larger than $O(\log n)$. This bound is in fact achieved for
$K = {\rm conv}(\vec{0},\vec{e}_1,2\vec{e}_2,\dots,n\vec{e}_n)$ and $\lat = \Z^n$ (where
$\vec{e}_1,\dots,\vec{e}_n$ denote the standard basis). Kannan and Lov{\'a}sz
asked whether this is indeed the worst case, thus we henceforth call an
affirmative answer to this question the Kannan-Lov{\'a}sz conjecture. 
We note that Kannan and Lov{\'a}sz verified this conjecture for $\lat=\Z^n$ and
$K={\rm conv}(\vec{0},a_1\vec{e}_1,\dots,a_n\vec{e}_n)$, $a_i > 0 ~\forall~i
\in [n]$, i.e., with respect to $\Z^n$ and any axis parallel scaling of a
simplex (notice that to prove the conjecture one can restrict to $\lat=\Z^n$,
only allowing $K$ to vary). Encouragingly, the main class of worst-case examples
for Khinchine's flatness theorem, which are derived from the natural ``uniform'' probability 
distribution on lattices are ``easy'' from the perspective of Theorem~\ref{thm:klgeneral}.
In particular, a classical result of Rogers~\cite[Theorem 2]{Rogers1958}, which
gives very precise volumetric estimates on the covering radius of a convex body
with respect to a random lattice, directly shows that for these lattices the
right-hand side can be made $O(1)$ using $W = \R^n$. 

If the conjecture holds and a subspace achieving the conjectured bound can be
computed in at most $O(\log n)^{O(n)}$-time, this would imply an $O(\log
n)^{O(n)}$-time algorithm for IP~\cite{thesis/D12}. Lastly, it was also shown
in~\cite{DM13} that a nearly exact minimizer for~\eqref{kl-bnd} can in fact be
computed. However, the algorithm requires $n^{O(n^2)}$ time.

In this paper we focus on the important special case $K=\Ball_2^n$ of this
conjecture, as in Definition~\ref{conj:kl-l2}. 
This case already seems to
capture many of the main difficulties of the general form of the conjecture.
Moreover, in the study of the geometry of numbers, it has often been the case
that if one can prove the desired estimate for $\ell_2$, then using the
appropriate (often non-trivial) convex geometric tools it can be lifted to
general convex bodies.

We remark that it is entirely possible that our main conjecture in fact implies
the \emph{general} KL conjecture. The main reason our techniques are currently
limited to the $\ell_2$ setting is that the objective we use in the convex
program may become (slightly) non-convex under norms other than $\ell_2$. This
technicality does not however seem like a fundamental problem, and we imagine
that it can be circumvented. As evidence that this should be possible,
we note that the $M$-$M^*$ estimate, from which the bound on Khinchine's Flatness
theorem is derived, is also the solution to a general norm optimization problem over
Gaussian covariances. 

\subsection{Proof outline}
\label{sec:proofoutline}

We now give a high level overview of the proof. 
The first step of the proof, appearing in Section~\ref{sec:smoothmubound}, is to formulate a convex relaxation $\mu_{\textnormal{sm}}$ of the
covering radius $\mu$. This convex relaxation is quite natural, and can be
described as measuring the ``most efficient'' way to smooth a lattice using 
an ellipsoidal Gaussian. We will explore it further in Section~\ref{sec:musmoothtight}.
Next, we use our main conjecture to arrive at a further convex relaxation, 
which we dub $\mu_{\textnormal{det}}$. This is done in Section~\ref{sec:determinantalbound}.
The final and arguably most interesting part of the proof is to round 
the dual formulation of $\mu_{\textnormal{det}}$. This is done in 
Section~\ref{sec:subspacerounding}. 

Since the proof of the theorem is just a combination of theorems appearing
in the following subsections, we already include it here. See below for the 
definitions of the various quantities. 

\begin{proof}[Proof of Theorem~\ref{thm:mink-to-kl}]
Let $\lat \subset \R^n$ be an $n$-dimensional lattice. By
Theorems~\ref{thm:smooth-mu-bnd},~\ref{thm:det-mu-bnd},~\ref{thm:det-mu-dual}
and~\ref{thm:subspace-rounding} below, we have that
\begin{align*}
\mu(\lat) &\leq \sqrt{16/\pi} ~\mu_{{\rm sm}}(\lat) 
           \leq \sqrt{16/\pi} C_\eta(n) ~\mu_{{\rm det}}(\lat)
           = \sqrt{16/\pi} C_\eta(n) ~\mu_{{\rm det}}^{{\rm dual}}(\lat) \\
          &\leq  \sqrt{384/\pi}(\log_2 n + 1)C_\eta(n) 
\max_{\substack{W \textnormal{ lattice subspace of } \lat^* \\ d=\dim(W) \in [n]}} 
\frac{\sqrt{d}}{\det(\lat^* \cap W)^{1/d}} \text{ ,}
\end{align*}
as needed.
\end{proof}

\subsection{Smooth \texorpdfstring{$\mu$}{mu} bound}
\label{sec:smoothmubound}

The first step in the proof is to bound the covering radius by a convex relaxation we
call $\mu_{\rm{sm}}$. The fact that the covering radius is at most $\sqrt{n}$ times 
the smoothing parameter is standard by now (it is already implicit in~\cite{Bana93}), and the theorem can be seen
as a slight extension of this standard fact to non-spherical Gaussians. One 
interesting aspect of the statement, though, is that the resulting minimization problem
is convex. \dnote{mention that this is a different proof, that only goes through
$\ell_1$ smoothing. see my email for an explanation as to why its different.}

\begin{theorem}[Smooth $\mu$ bound]
\label{thm:smooth-mu-bnd}
For an $n$-dimensional lattice $\lat \subset \R^n$, let
\begin{equation}
\label{eq:smooth-mu}
\mu_{\rm{sm}}(\lat)^2 = \min \set[\Big]{\tr(A) : A \succeq 0, \sum_{\vec{y} \in \lat^*
\setminuszero} e^{-\pi \vec{y}^\T A \vec{y}} \leq 1/2} \text{.}
\end{equation}  
Then $\mu(\lat) \leq 4 \pi^{-1/2} \mu_{{\rm sm}}(\lat)$. Furthermore, the
program defining $\mu_{{\rm sm}}(\lat)$ is convex.
\end{theorem}
\begin{proof}
To see that the program defined by~\eqref{eq:smooth-mu} is convex in $A$, note
first that the objective $\tr(A)$ is linear in $A$. Second, the single
constraint $\sum_{\vec{y} \in \lat^* \setminuszero} e^{-\pi \vec{y}^\T A
\vec{y}} \leq 1/2$ is convex since $e^x$ is convex in $x$ and the functions $-\pi
\vec{y}^\T A \vec{y}$, for $\vec{y} \in \lat^*$, are linear in $A$. 

Let $A$ be any valid solution to~\eqref{eq:smooth-mu}, and let 
$X \sim N(0,\frac{1}{2\pi} A)$. 
By Markov's inequality, 
\begin{equation*}
\Pr\bracks[\Big]{\dist(X,\lat)^2 \geq \frac{4}{\pi} \tr{A}} \leq \Pr\bracks[\Big]{\|X\|_2^2 \geq \frac{4}{\pi} \tr{A}} \leq
\frac{\E[\|X\|_2^2]}{\frac{4}{\pi} \tr{A}} = 
\frac{\frac{1}{2\pi} \tr{A}}{\frac{4}{\pi} \tr{A}} = 1/8 \text{ .}
\end{equation*}
Let $\vec z$ be a uniform point in $\R^n / \lat$. 
By Claim~\ref{clm:smoothing} with $\eps=1/2$, with probability at least $7/8 - 1/4 = 5/8 > 1/2$, 
$\vec z$ is within distance $(\frac{4}{\pi} \tr{A})^{1/2}$ of $\lat$.
Using Claim~\ref{clm:gmrcovering}, we obtain 
$\mu(\lat) \le 2(\frac{4}{\pi} \tr{A})^{1/2}$ as desired. 
\end{proof}

\begin{remark}
Observe that the only property of $A$ used in the proof above is that 
$\Delta(X \bmod \lat, U) \leq 1/4$ where $X \sim N(0,\frac{1}{2\pi} A)$.
As a result, for any such $A$ we have that $\mu(\lat) \le 4 \pi^{-1/2} \sqrt{\tr(A)}$.
\end{remark}

\subsection{Determinantal \texorpdfstring{$\mu$}{mu} bound and its dual}
\label{sec:determinantalbound}

\begin{theorem}[Determinantal $\mu$ bound]
\label{thm:det-mu-bnd}
For an $n$-dimensional lattice $\lat \subset \R^n$, let
\begin{equation}
\label{eq:det-mu}
\mu_{\textnormal{det}}(\lat)^2 = 
\min \set[\Big]{\tr(A) : A \succeq 0, \det_W(A) \geq
\frac{1}{\det(\lat^* \cap W)^2}, \forall~ W \neq \set{\vec{0}}
\textnormal{ lattice subspace of } \lat^*}
 \text{.}
\end{equation}  
Then, the program defining $\mu_{{\rm det}}(\lat)$ is convex, and  
\[ 
\mu_{{\rm det}}(\lat)/2 \leq \mu_{{\rm sm}}(\lat) \leq C_\eta(n) \mu_{{\rm det}}(\lat) \; .
\]
\end{theorem}
\begin{proof}
To see that the program is convex, note first the objective $\tr(A)$ is linear
in $A$. Secondly, each constraint $\det_W(A) \geq \det(\lat^* \cap
W)^{-2}$ is convex by concavity of $\det_W(\cdot)^{1/\dim(W)}$ (see
Lemma~\ref{lem:proj-det-props}).

\noindent We now prove the inequalities relating $\mu_{{\rm det}}$ and $\mu_{{\rm sm}}$.

\paragraph{Step 1: $\mu_{{\rm det}}(\lat)/2 \leq \mu_{{\rm sm}}(\lat)$:} Assume that $A$
is a valid solution for $\mu_{{\rm sm}}(\lat)$. We will show that $4A$ is valid
for $\mu_{{\rm det}}(\lat)$, which clearly suffices to prove the claim.  

For any lattice subspace $W$ of $\lat^*$,
\begin{equation}
\label{eq:det-mu-bnd-1}
\rho_{A^+}(\lat^* \cap W) \leq \rho_{A^+}(\lat^*) = 1 + \rho_{A^+}(\lat^*
\setminuszero) \leq 3/2 \leq 2 \text{ .}
\end{equation}
By Lemma~\ref{lem:coset-mass},
\begin{equation}
\label{eq:det-mu-bnd-2}
\begin{split}
\rho_{A^+}(\lat^* \cap W) 
&= \frac{\det_W((A^+)^{\cap W})^{1/2}}{\det(\lat^* \cap W)} 
   \rho_{A^{\da W}}((\lat^* \cap W)^*)  \\
& \geq \frac{\det_W((A^+)^{\cap W})^{1/2}}{\det(\lat^* \cap W)} 
= \frac{\det_W((A^{\da W})^+)^{1/2}}{\det(\lat^* \cap W)} 
= \frac{1}{\det_W(A)^{1/2}\det(\lat^* \cap W)} \; .
\end{split}
\end{equation}
Combining~\eqref{eq:det-mu-bnd-1},\eqref{eq:det-mu-bnd-2} and rearranging, we
get
\[
\det_W(A) \geq \frac{1}{4 \det(\lat^* \cap W)^2} \text{ .}
\]
From here, note that 
\[
\det_W(4A) = 4^{\dim(W)} \det_W(A) \geq 4 \det_W(A) \geq \frac{1}{\det(\lat^*
\cap W)^2} \text{ .}
\]
Since the above holds for all non-zero lattice subspaces $W$ of $\lat^*$, $4A$
is a solution to $\mu_{{\rm det}}(\lat)$ as needed. 

\paragraph{Step 2: $\mu_{{\rm sm}}(\lat) \leq C_{\eta}(n) \mu_{{\rm det}}(\lat)$:}~
Assume that $A$ is a valid solution for $\mu_{{\rm det}}(\lat)$. We will show
that $C_{\eta}(n)^2 A$ is valid for $\mu_{{\rm sm}}(\lat)$, which clearly
suffices to prove the claim. 

To show this, we must prove that 
\[
\rho_{(C_\eta(n)^2 A)^+}(\lat^* \setminuszero) \leq 1/2 \text{ ,}
\]
which is equivalent to asking for $\eta((C_\eta(n) A^{1/2} \lat^*)^*) \leq 1$.
By the definition of $C_\eta(n)$ in~\eqref{eq:finalgoal},
\begin{align*}
\eta((C_\eta(n) A^{1/2} \lat^*)^*) 
&\leq C_\eta(n) \max_{W \neq \set{\vec{0}} \textnormal{ lattice subspace of } A^{1/2} \lat^*} 
                \frac{1}{\det((C_\eta(n) A^{1/2}\lat^*) \cap W)^{1/\dim(W)}} \\
&= \max_{W \neq \set{\vec{0}} \textnormal{ lattice subspace of } A^{1/2} \lat^*} 
                \frac{1}{\det((A^{1/2}\lat^*) \cap W)^{1/\dim(W)}} \\
&= \max_{W \neq \set{\vec{0}} \textnormal{ lattice subspace of } \lat^*} 
                \frac{1}{\det(A^{1/2} (\lat^* \cap W))^{1/\dim(W)}} \\
&= \max_{W \neq \set{\vec{0}} \textnormal{ lattice subspace of } \lat^*} 
                \frac{1}{(\det_W(A)^{1/2}\det(\lat^* \cap W))^{1/\dim(W)}}
    \leq 1 \text{ ,}
\end{align*}
where the last equality follows by Claim~\ref{clm:detlattransform}.
\end{proof}

We now write down the dual of $\mu_{\textnormal{det}}$ and prove that strong
duality holds. The main delicate point is to show that only a finite number of
the constraints of $\mu_{\textnormal{det}}$ are relevant. The rest is a standard
analysis of the KKT conditions. 

\begin{theorem}[Dual of determinantal $\mu$ bound]
\label{thm:det-mu-dual}
For an $n$-dimensional lattice $\lat \subset \R^n$, let
\begin{equation}
\label{eq:det-mu-dual}
\begin{split}
\mu^{{\rm dual}}_{\textnormal{det}}(\lat)^2 
= 
\textit{maximize} ~ &\sum_{i=1}^m \frac{d_i \det_{W_i}(X_i)^{1/d_i}}{\det(\lat^*
\cap W_i)^{2/d_i}} \\ 
\textit{subject to~}
& \sum_{i=1}^m X_i \preceq I \\ 
&W_i \neq \set{\vec{0}} \textnormal{ lattice subspace of } \lat^*, d_i = \dim(W_i), i \in [m] \\
& X_i \succeq 0,~\ran(X_i) = W_i, i \in [m] \\
& m \in \N \text{.}
\end{split}
\end{equation}
Then $\mu_{{\rm det}}(\lat) = \mu^{{\rm dual}}_{{\rm det}}(\lat)$.
\end{theorem}
\begin{proof}~
\paragraph{Step 1: $\mu^{{\rm dual}}_{{\rm det}}(\lat) \leq \mu_{{\rm det}}(\lat)$:}
We first show weak-duality. Let $A$ be any solution to $\mu_{{\rm det}}(\lat)$,
and let $X_1,\dots,X_m \in \PSD^n$ with associated lattice subspaces
$W_1,\dots,W_m$, $d_i = \dim(W_i)$ for $i \in [m]$, be a solution to $\mu_{{\rm
det}}^{{\rm dual}}(\lat)$.

By assumption on $A$, for each $W_i$, $i \in [m]$, we have that
\begin{align*}
\frac{1}{\det(\lat \cap W_i)^{2/d_i}} 
&\leq \det_{W_i}(A)^{1/d_i} 
= \inf \set[\Big]{\frac{1}{d_i} \tr(AZ): Z \in \PSD^n, \ran(Z) = W_i, \det_{W_i}(Z)=1} \\
&\leq \frac{1}{d_i} \tr(A(X_i/\det_{W_i}(X_i)^{1/d_i})) \text{ ,}
\end{align*}
where the first equality follows from Lemma~\ref{lem:proj-det-props}.
Rearranging, we get that 
\[
\tr(A X_i) \geq \frac{d_i \det_{W_i}(X_i)^{1/d_i}}{\det(\lat^* \cap W_i)^{2/d_i}}
\text{.}
\]
Since $\sum_{i=1}^m X_i \preceq I_n$, 
\[
\tr(A) = \tr(A I_n) \geq \sum_{i=1}^m \tr(A X_i)
\geq \sum_{i=1}^m \frac{d_i \det_{W_i}(X_i)^{1/d_i}}{\det(\lat^* \cap
W_i)^{2/d_i}} \text{ .}
\]
Since the above holds for any pair of solutions of $\mu_{{\rm
det}}(\lat)$ and $\mu_{{\rm det}}^{{\rm dual}}(\lat)$, we conclude that $\mu_{{\rm
det}}(\lat) \geq \mu_{{\rm det}}^{{\rm dual}}(\lat)$ as needed.

\paragraph{Step 2: $\mu^{{\rm dual}}_{{\rm det}}(\lat) \geq \mu_{{\rm det}}(\lat)$:}
We now prove strong duality. 

To begin, we show that the solution space of $\mu_{{\rm det}}(\lat)$ can be
restricted to a compact subset of $\PSD^n$. To see this, note that $\eta^2(\lat)
I_n$ is by definition a solution to $\mu_{{\rm sm}}(\lat)$ and
hence by Lemma~\ref{thm:det-mu-bnd}, we have that $4\eta^2(\lat) I_n$ is a
solution to $\mu_{{\rm \det}}(\lat)$. Thus any optimal solution $A^*$ must
satisfy $\tr(A^*) \leq \tr(4\eta^2(\lat) I_n) < 8n \eta^2(\lat)$.\onote{Posterity note: We're taking $2n$ here and not satisfied with $n$
so that this constraint is not tight, as needed later} Since the set $C = \set{A \in
\PSD^n: \tr(A) \leq 8 n \eta^2(\lat)}$ is compact, the claim is proved.

Let $A \in C$ be a solution to $\mu_{{\rm det}}(\lat)$. Let $\lambda_1(A) \geq
\cdots \geq \lambda_n(A)$ denote the eigenvalues of $A$. By assumption on $A$,
\begin{align*}
\frac{1}{\det(\lat^*)^2} \leq \det(A) &= \prod_{i=1}^n \lambda_i(A) \\
& \leq \lambda_n(A) \parens[\big]{\sum_{i=1}^{n-1} \lambda_i(A)/(n-1)}^{n-1}  \\
& \leq \lambda_n(A) \tr(A)^{n-1} \leq \lambda_n(A) (8n \eta^2(\lat))^{n-1} \text{ .}
\end{align*}
Letting $\tau = ((8n \eta^2(\lat))^{n-1} \det(\lat^*)^2)^{-1}$, by the above, we
have that $\lambda_n(A) \geq \tau > 0$. 

Since for any lattice subspace $W \neq \set{\vec{0}}$ of $\lat^*$, we have that
\[
\det_W(A) \geq \lambda_n(A)^{\dim(W)} \geq \tau^{\dim(W)} \; ,
\]
the constraint $\det_W(A) \geq \frac{1}{\det(\lat^* \cap W)^2}$ can only be
tight if $\det(\lat^* \cap W) \leq \tau^{-\dim(W)/2}$.

Given the above, the program 
\begin{equation}
\label{eq:rest-det-prog}
\begin{split}
\textit{minimize~} &\tr(A)  \\
\textit{subject to~} &A \succeq 0 \\
     & \tr(A) \leq 8n \eta^2(\lat) \dnote{5/26: changed 2 to 8} \\
     & \det_W(A) \geq \frac{1}{\det(\lat^* \cap W)^{2}}, 
               ~~ \forall W \neq \set{\vec{0}} \textnormal{ lattice subspace of } \lat^* \\
              & \hspace{8.5em} \text{ satisfying } \det(\lat^* \cap W) \leq
\tau^{-\dim(W)/2} \text{ .}
\end{split}
\end{equation}
is equivalent to $\mu_{{\rm det}}(\lat)^2$.

We now show that the program~\eqref{eq:rest-det-prog} has only a finite number
of constraints. For this purpose, it suffices to show
\[
|\set{M \subseteq \lat: M \text{ sublattice of } \lat, \det(M) \leq
\tau^{-\dim(M)/2}}| < \infty \text{ .}
\] 
Since $\lat$ is discrete, it suffices to show that all such sublattices are
generated by vectors of $\lat$ of bounded norm. Let $M \subseteq \lat$ be a
$d$-dimensional sublattice of $\lat$. By Minkowski's second
Theorem~\ref{thm:mink-sec}, 
\[
\prod_{i=1}^d \lambda_i(M) \leq \frac{2^d}{\vol_d(\Ball_2^d)} \det(M)
\leq \frac{2^d}{\vol_d(\Ball_2^d)} \tau^{-d/2} \text{.}
\]
Rearranging the above,
\[
\lambda_d(M) \leq \frac{2^d}{\vol_d(\Ball_2^d)} \tau^{-d/2} \prod_{i=1}^{d-1}
\lambda_i(M)^{-1}
\leq \frac{2^d}{\vol_d(\Ball_2^d)} \tau^{-d/2} \lambda_1(\lat)^{-d+1} \text{ .}
\]
Let 
\[
\beta = \max_{d\in[n]} \frac{2^d}{\vol_d(\Ball_2^d)} \tau^{-d/2} 
\lambda_1(\lat)^{-d+1} \text{ .}
\]
By the above, any sublattice $M \subseteq \lat$ of determinant at most
$\tau^{-\dim(M)/2}$ has all successive minima bounded by $\beta$.  In
particular, all such sublattices are generated by vectors of length at most
$\sqrt{n} \beta < \infty$ (see for example~\cite[Lemma 7.1]{MG02}), and hence are finite in number.

Let $A$ denote an optimal solution to~\eqref{eq:rest-det-prog}, which exists
since the objective is continuous and the feasible region is non-empty and
compact. Since any optimal solution satisfies $\tr(A) < 8n \eta(\lat^*)^2$ and
$\lambda_n(A) \geq \tau > 0$, the only constraints that can be tight at $A$ are
the subspace determinant constraints. Let $W_1,\dots,W_m$, $d_i = \dim(W_i)$ for
$i \in [m]$, denote the lattice subspaces of $\lat^*$ for which 
\[
\det_{W_i}(A) = \frac{1}{\det(\lat^* \cap W_i)^2} \text{ .}
\]
Note that this list is finite since, as argued previously,
program~\eqref{eq:rest-det-prog} contains only a finite number of determinant
constraints. 

Let $X_i' = (A^{\da W_i})^+ \det_{W_i}(A_i)^{1/d_i}$, for $i \in [m]$. Note that
$X_1',\dots,X_m'$ are well-defined since $A \succ 0$. By construction,
$\ran(X_i') = W_i$, $\det_{W_i}(X_i') = 1$ (by Item~\ref{itm:detprod} of Lemma~\ref{lem:proj-det-props}), and  
\[
\tr(A X_i') = d_i \det_{W_i}(A_i)^{1/d_i} = \frac{d_i}{\det(\lat^* \cap W_i)^{2/d_i}} \text{ .}
\]
Furthermore, by Lemma~\ref{lem:proj-det-props}, $\frac{d}{dt}\det_{W_i}(A +
t\Delta)^{1/d_i}|_{t=0} = \tr(X_i' \Delta)/d_i$ for any $\Delta \in \R^{n \times n}$.

\begin{claim} 
\label{cl:kkt}
$I_n \in {\rm cone}(X_1',\dots,X_m') \eqdef \set{\sum_{i=1}^m a_i X_i': a_i \geq
0, i \in [m]}$.
\end{claim}
\begin{proof}
We first show that ${\rm cone}(X_1',\dots,X_m')$ is pointed, i.e., it does not
contain any 1-dimensional subspaces. If it did, there would exist $a_1,\dots,a_m \geq 0$,
$\sum_{i=1}^m a_i \neq 0$, such that $\sum_{i=1}^m a_i X_i' = 0$. However, no
such combination can exist since $X_1',\dots,X_m' \in \PSD^n \setminuszero$.

Assume now that $I_n \notin {\rm cone}(X_1',\dots,X_m')$. Then, since ${\rm
cone}(X_1',\dots,X_m')$ is a finitely generated pointed cone and $X_i' \neq 0 ~
\forall~i \in [m]$, by the separation theorem\onote{posterity: maybe add a ref}
there exists $\Delta \in \R^{n
\times n}$, $\Delta = \Delta^\T$, such that (a) $\tr(I \Delta) < 0$ and (b)
$\tr(X_i' \Delta) > 0$, $\forall i \in [m]$. But then, for any $t > 0$ small
enough, we claim that $A + t \Delta$ is still feasible
for~\eqref{eq:rest-det-prog}.  Note that this would contradict the optimality of
$A$ since by (a) $\tr(A + t \Delta) < \tr(A)$. To argue feasibility, by
continuity and the fact that~\eqref{eq:rest-det-prog} has a finite number of
constraints, we need only show that $A+t \Delta$ remains feasible on the tight
constraints at $A$.  By condition (b) $\frac{d}{dt} \det_{W_i}(A+t
\Delta)|_{t=0} > 0$, $\forall i \in [m]$. Hence, for all $t > 0$ small enough, 
\[
\det_{W_i}(A+ t \Delta) > \det_{W_i}(A) = \frac{1}{\det(\lat^* \cap W_i)^2},
\forall i \in [m] \text{ ,}
\]
as needed.
\end{proof}

By the above claim, there exists $a_1,\dots,a_m \geq 0$ such that $I_n =
\sum_{i=1}^m a_i X_i'$. Let $X_i = a_i X_i'$, for $i \in [m]$. By construction,
$\det_{W_i}(X_i)^{1/d_i} = a_i$ for all $i \in [m]$. By possibly deleting some
of the $X_i$, we may assume that $X_i \neq 0$, $\forall i \in [m]$.  Note that
$X_1,\dots,X_m$ with associated subspaces $W_1,\dots,W_m$ is a valid
solution to $\mu_{{\rm det}}^{{\rm dual}}(\lat)$. To finish the proof, note that
\[
\mu_{{\rm det}}(\lat)^2 = \tr(A) = \sum_{i=1}^m \tr(AX_i)
= \sum_{i=1}^m \frac{d_i \det_{W_i}(X_i)^{1/d_i}}{\det(\lat^* \cap W_i)^{2/d_i}}
\leq \mu_{{\rm det}}^{{\rm dual}}(\lat)^2 \text{, }
\]
as needed.
\end{proof}

\subsection{Subspace rounding}
\label{sec:subspacerounding}

Here we prove Theorem~\ref{thm:subspace-rounding}, the most involved part of the reduction.
The proof uses a certain ``uncrossing'' inequality, 
Lemma~\ref{lem:uncrossing}, proven in Section~\ref{sec:uncrossing} below. 
We start with what can be seen as a reverse form of the 
inequality of arithmetic and geometric means.

\begin{lemma}[Reverse AM-GM]
\label{lem:rev-am-gm}
Let $a_1,\dots,a_m \geq 0$ and $d_1,\dots,d_m \in \N$. For $S \subseteq [m]$,
define $d_S = \sum_{i \in S} d_i$. Then
\[
\sum_{i=1}^m d_i a_i \leq 
4\ceil{\log_2(2d_{[m]})} \max_{S \subseteq [m]} d_S \parens[\Big]{\prod_{i \in S}
a_i^{d_i}}^{1/d_S} 
\]
\end{lemma}
\begin{proof}
Let $r = \max_{i \in [m]} a_i$ and $k = \ceil{\log_2(2d_{[m]})}$. For $j \geq
1$, define $S_j = \set{i \in [m]: 2^{-j} r < a_i \leq 2^{-j+1} r }$, noting that
the family $\set{S_j: j \geq 1}$ forms a partition of $\set{i \in [m]: a_i \neq 0}$. 
We first show that
the first $k$ sets represent at least half the total sum:
\[
      \sum_{j=k+1}^\infty \sum_{i \in S_j} d_i a_i 
      \leq 
			r 2^{-k} d_{[m]} 
			\leq r/2 
			\leq (1/2) \sum_{i=1}^m d_i a_i \text{.}
\]
Given the above, by construction of the sets $\{S_j: j \geq 1\}$, we get that
\begin{align*}
\sum_{i=1}^m d_i a_i 
 &\leq 2 \sum_{j=1}^k \sum_{i \in S_j} d_i a_i 
  \leq 2k \max_{j \in [k]} \sum_{i \in S_j} d_i a_i \\
 &\leq 4k \max_{j \in [k]} d_{S_j}(2^{-j} r) 
  \leq 4k \max_{j \in [k]} d_{S_j}\parens[\Big]{\prod_{i \in S_j} a_i^{d_i}}^{1/d_{S_j}} \text{ ,}
\end{align*}
as needed.
\end{proof}

\begin{theorem}[Subspace Rounding]
\label{thm:subspace-rounding}
Let $\lat \subset \R^n$ be an $n$-dimensional lattice. Then
\begin{equation}
\label{eq:sub-rnd}
\mu_{{\rm det}}^{\rm dual}(\lat)^2 \leq 24(\log_2 n + 1)^2 \max_{\substack{W \textnormal{ lattice subspace
of } \lat^* \\ d=\dim(W) \in [n]}} \frac{d}{\det(\lat^* \cap
W)^{2/d}} \text{.}
\end{equation}
\end{theorem}
\begin{proof}
Our goal is to show how to find, given any dual solution for \eqref{eq:det-mu-dual},
a subspace of $\lat^*$ having higher value. 
Let $W_1,\dots,W_m$ be lattice subspaces of $\lat^*$ and let $d_i = \dim(W_i) \geq
1$ for $i \in [m]$. Take $X_1,\dots,X_m \in \PSD^n$ such that $\ran(X_i) = W_i$
and $\sum_{i=1}^m X_i \preceq I_n$, corresponding to a valid solution
to~\eqref{eq:det-mu-dual}. Our goal will be to extract a single subspace $W^*$
such that 
\[
\sum_{i=1}^m \frac{d_i \det_{W_i}(X_i)^{1/d_i}}{\det(\lat^* \cap W_i)^{2/d_i}}
\leq 24(\log_2 n + 1)^2 \frac{\dim(W^*)}{\det(\lat^* \cap W^*)^{2/\dim(W^*)}} \text{ .}
\]

\paragraph{Step 1:} Reduction to the case $m \leq n^2$.  
\onote{TODO: Noah suggests to capture these steps in lemmas: 
The statement is something like "For any solution $W_1,\ldots, W_m$, there
exists a solution $W_1',\ldots, W_{m'}'$ with $m' \leq n^2, W_1' \subseteq
\cdots \subseteq W_{n^2}'$,  and [the objective function within a logarithmic
factor]."} \dnote{5/18: agreed. it's worth breaking up}

It suffices to show that for $m > n^2$ we can reduce the support of the solution
by at least $1$ without decreasing the objective value. If $m > n^2$, since the
space of $n \times n$ matrices is $n^2$ dimensional, there must exist a
non-trivial linear dependence $\lambda \in \R^m$, $\lambda \neq 0$, such that
$\sum_{i=1}^m \lambda_i X_i = 0$. 
It follows that there must exist at least one negative $\lambda_i$, since the $X_i$ 
are nonzero positive semidefinite matrices, and otherwise we would not have
$\sum_{i=1}^m \lambda_i X_i = 0$. Similarly there must exist at least one positive $\lambda_j$. 
Given this, we deduce that the set $R = \set{\eps \in \R: \forall i \in [m], 1+\eps \lambda_i \geq
0}$ is a bounded interval and that 
$(1+\eps\lambda_1)X_1,\dots,(1+\eps \lambda_m)X_m$ is a valid dual solution for
any $\eps \in R$. For $\eps \in R$, by homogeneity
\[
\sum_{i=1}^m \frac{d_i \det_{W_i}((1+\eps \lambda_i)X_i)^{1/d_i}}{\det(\lat^*
\cap W_i)^{2/d_i}} = \sum_{i=1}^m (1+\eps \lambda_i) \frac{d_i
\det_{W_i}(X_i)^{1/d_i}}{\det(\lat^* \cap W_i)^{2/d_i}} \text{.}
\]
Hence, the objective is a linear function in $\eps$. In particular, the above
function attains its maximum at one of the end points of $R$, which in turn
corresponds to a dual solution of smaller support.

\paragraph{Step 2:} Reduction to the chain case.

In this step, we show that at the expense of losing a logarithmic factor in the
objective, one can reduce to the case where the subspaces $W_1,\dots,W_m$, $m
\leq n^2$, form a \emph{chain}, that is where $W_1 \subseteq W_2 \subseteq \dots
\subseteq W_m$.  

To begin, we apply Lemma~\ref{lem:rev-am-gm} to the numbers
$\frac{\det_{W_1}(X_1)^{1/d_1}}{\det(\lat^* \cap
W_1)^{2/d_1}},\dots,\frac{\det_{W_m}(X_m)^{1/d_m}}{\det(\lat^* \cap
W_m)^{2/d_m}}$ with multiplicities $d_1,\dots,d_m$. This gives us that
\begin{equation}
\label{eq:det-am-gm}
\begin{split}
\sum_{i=1}^m \frac{d_i \det_{W_i}(X_i)^{1/d_i}}{\det(\lat^* \cap W_i)^{2/d_i}}
&\leq 4\ceil{\log_2(2d_{[m]})} \max_{S \subseteq [m]} d_{S}\left(\prod_{i \in S}
\frac{\det_{W_i}(X_i)}{\det(\lat^* \cap W_i)^2}\right)^{1/d_{S}} \\
&\leq 4\ceil{\log_2(2n^3)}  \max_{S \subseteq [m]} d_{S}\left(\prod_{i \in S}
\frac{\det_{W_i}(X_i)}{\det(\lat^* \cap W_i)^2}\right)^{1/d_{S}} \text{, }
\end{split}
\end{equation}
where the last inequality follows since $d_{[m]} = \sum_{i=1}^m d_i \leq n m
\leq n^3$.

Let $S^* \subseteq [m]$ denote the maximizer on the last line
of~\eqref{eq:det-am-gm}.  Without loss of generality, we may assume that $S^* =
[k]$, for some $k \in [m]$. 

We now show that if $W_1,\dots,W_k$ cannot be rearranged to form a chain, one
can find an updated solution $X_1',\dots,X_k'$ to~\eqref{eq:det-mu-dual} with
associated subspaces $W_1',\dots,W_k'$ such that 
\begin{enumerate}
\item $\prod_{i=1}^k
\frac{\det_{W_i}(X_i)}{\det(\lat^* \cap W_i)^2} \leq \prod_{i=1}^k
\frac{\det_{W'_i}(X'_i)}{\det(\lat^* \cap W'_i)^2}$.
\item $\sum_{i=1}^k \dim(W_i) = \sum_{i=1}^k \dim(W_i')$.
\item $\sum_{i=1}^k \dim(W_i)^2+1 \leq \sum_{i=1}^k \dim(W'_i)^2$. 
\end{enumerate}

Given that the potential in Property 3 is integer valued and ranges from $1$ to
$kn^2$ (since each subspace has dimension at most $n$), the updating process
converges to a chain after at most $kn^2$ iterations.

Assume then that $W_1,\dots,W_k$ cannot be rearranged to form a chain. Then there
exist $i,j \in [k]$, such that $W_i \not\subseteq W_j$ and $W_j \not\subseteq
W_i$. Without loss of generality, we may assume that $i=1,j=2$ and $\dim(W_1)
\leq \dim(W_2)$. We now construct the updated solution as follows: let $W_1' =
W_1 \cap W_2$, $W_2' = W_1+W_2$, $X_1' = (X_1+X_2)^{\cap (W_1 \cap W_2)}/2$,
$X_2' = (X_1+X_2)-(X_1+X_2)^{\cap (W_1 \cap W_2)}/2$, $W_i' = W_i$ and $X_i' =
X_i$ for $i \ge 3$. By construction $X_i' \succeq 0$ and $\sum_{i=1}^k X_i = \sum_{i=1}^k X_i'$, and
thus the updated solution is indeed feasible for~\eqref{eq:det-mu-dual}. Next,
by Lemma~\ref{lem:uncrossing}, we have that
\[
\frac{\det_{W_1}(X_1)}{\det(\lat^* \cap W_1)^2}
\frac{\det_{W_2}(X_2)}{\det(\lat^* \cap W_2)^2}
\leq
\frac{\det_{W_1'}(X_1')}{\det(\lat^* \cap W_1')^2}
\frac{\det_{W_2'}(X_2')}{\det(\lat^* \cap W_2')^2} \text{ , }
\]
and hence Property 1 is satisfied. Since $\dim(W_1') + \dim(W_2') = \dim(W_1 \cap W_2)
+ \dim(W_1 + W_2) = \dim(W_1) + \dim(W_2)$, we also get that Property 2 is satisfied.
Given that $W_1,W_2$ do not form a chain and $\dim(W_1) \leq \dim(W_2)$, we must
have that $\dim(W_1') = \dim(W_1 \cap W_2) < \dim(W_1)$. Therefore $\dim(W_1') =
\dim(W_1 \cap W_2) = \dim(W_1)-a$ and $\dim(W_2') = \dim(W_1+W_2) = \dim(W_2) +
a$ for some $a \geq 1$. We now verify Property 3 via a direct calculation:
\begin{align*}
\dim(W_1')^2 + \dim(W_2')^2 &= (\dim(W_1)-a)^2 + (\dim(W_2)+a)^2 \\
&= \dim(W_1)^2 + \dim(W_2)^2 + 2a(\dim(W_2)-\dim(W_1)) + 2a^2 \\
&\geq \dim(W_1)^2 + \dim(W_2)^2 + 1 \text{, \quad as needed.}
\end{align*}

Let $X_1',\dots,X_k'$ with associated subspaces $W_1' \subseteq W_2' \subseteq
\dots \subseteq W_k'$, $d_i' = \dim(W_i')$ for $i \in [k]$, denote the final
chain solution obtained via the update process. Since $S^* = [k]$, by
equation~\eqref{eq:det-am-gm} and construction of the updated solution, we have
that 
\begin{equation}
\label{eq:red-to-chain}
\begin{split}
\sum_{i=1}^m \frac{d_i \det_{W_i}(X_i)^{1/d_i}}{\det(\lat^* \cap W_i)^{2/d_i}}
&\leq 4\ceil{\log_2(2n^3)} (\sum_{i=1}^k d_i)\left(\prod_{i=1}^k
\frac{\det_{W_i}(X_i)}{\det(\lat^* \cap W_i)^2}\right)^{1/(\sum_{i=1}^k d_i)} \\
&\leq 4\ceil{\log_2(2n^3)} (\sum_{i=1}^k d_i')\left(\prod_{i=1}^k
\frac{\det_{W_i'}(X_i')}{\det(\lat^* \cap W_i')^2}\right)^{1/(\sum_{i=1}^k d_i')} \\
&\leq 4\ceil{\log_2(2n^3)} \sum_{i=1}^k \frac{d_i'
\det_{W_i'}(X_i')^{1/d_i'}}{\det(\lat^* \cap W_i')^{2/d_i'}} \text{, } 
\end{split}
\end{equation} 
where the last inequality follows by the AM-GM inequality. 

\paragraph{Step 3:} The chain case.

Let $X_1',\dots,X_k'$ be the chain solution from the last step with
associated subspaces $W'_1 \subseteq W'_2 \subseteq \dots \subseteq W'_k$, $d'_i =
\dim(W'_i)$ for $i \in [k]$. Let $S_j = \set{i \in [k]: 2^{j-1} \leq d'_i <
2^j}$ for $1 \leq j \leq \floor{\log_2(2n)}$. Noting that $\cup_{1 \leq j \leq
\floor{\log_2(2n)}} S_j = [k]$, we have that
\begin{equation}
\label{eq:det-to-trace}
\begin{split}
\sum_{i=1}^k \frac{d'_i \det_{W_i'}(X'_i)^{1/d'_i}}{\det(\lat^* \cap W'_i)^{2/d_i}}
&\leq \floor{\log_2(2n)} \max_{1 \leq j \leq \floor{\log_2(2n)}} \sum_{i \in
S_j} \frac{d'_i
\det_{W_i'}(X'_i)^{1/d'_i}}{\det(\lat^* \cap W'_i)^{2/d'_i}} \\
&\leq  \floor{\log_2(2n)} \max_{1 \leq j \leq \floor{\log_2(2n)}} \sum_{i \in S_j}
\frac{\tr(X'_i)}{\det(\lat^* \cap W'_i)^{2/d'_i}} \text{, }
\end{split}
\end{equation}
where the last inequality follows by the AM-GM inequality applied to the eigenvalues of $X'_i$. 
For any $j$, $1 \leq j \leq
\floor{\log_2(2n)}$, 
let $\ell_j = \max \set{i: i \in S_j}$,
and notice by the chain property that for all $i \in S_j$ we have
$\ran(X'_i) = W'_i \subseteq W'_{\ell_j}$.
Given this, we have that
\[
\sum_{i \in S_j} X'_i \preceq I \Rightarrow \sum_{i \in S_j} X'_i \preceq
\pi_{W'_{\ell_j}} \text{ .}
\]
Therefore,
\begin{equation}
\label{eq:trace-bnd}
\sum_{i \in S_j} \tr(X'_i) \leq \tr(\pi_{W'_{\ell_j}}) = d'_{\ell_j} \leq 2^j \text{ .}
\end{equation}
Continuing from the last line of~\eqref{eq:det-to-trace}, we have that
\begin{equation}
\label{eq:chain-round}
\begin{split}
\max_{1 \leq j \leq \floor{\log_2(2n)}} \sum_{i \in S_j}
\frac{\tr(X'_i)}{\det(\lat^* \cap W'_i)^{2/d'_i}}
&\leq
\max_{1 \leq j \leq \floor{\log_2(2n)}} \max_{i \in S_j} \left(\frac{1}{\det(\lat^* \cap
W'_i)^{2/d'_i}} \right) \sum_{i \in S_j} \tr(X'_i) \\
&\leq
\max_{1 \leq j \leq \floor{\log_2(2n)}} \max_{i \in S_j} \left(\frac{1}{\det(\lat^* \cap
W'_i)^{2/d'_i}} \right) 2^j \\
&\leq 2
\max_{i \in [k]} \left(\frac{d'_i}{\det(\lat^* \cap
W'_i)^{2/d'_i}} \right) \text{ .}
\end{split}
\end{equation}

Combining~\eqref{eq:red-to-chain},\eqref{eq:det-to-trace},\eqref{eq:chain-round}
we get that
\begin{align*}
\sum_{i=1}^m \frac{d_i \det_{W_i}(X_i)^{1/d_i}}{\det(\lat^* \cap W_i)^{2/d_i}} &\leq (4\ceil{\log_2(2n^3)})(2\floor{\log_2 2n})
\max_{i \in [k]} \frac{d'_i}{\det(\lat^* \cap
W_i')^{2/d_i'}} \\
&\leq 24(\log_2 n + 1)^2 \max_{\substack{W \textnormal{ lattice subspace
of } \lat^* \\ d=\dim(W) \in [n]}} \frac{d}{\det(\lat^* \cap
W)^{2/d}} \text{.}
\end{align*}
The theorem now follows given that the above holds starting from any dual solution.
\end{proof}

\onote{note for posterity: the loss of the $\log$ factor is necessary. 
One can consider a chain of subspaces of dimensions $n, n/2,n/2, n/4,n/4,n/4,n/4, \ldots$ 
already in the middle of~\eqref{eq:red-to-chain} (before the AMGM). By assuming that the lattice
determinants behave like $d^{d/2}$, optimizing the $X_i'$ (each coordinate should be uniform among the subspaces that contain it)
we get an expression like $\Sigma (\prod c_i^{-c_i})^{1/\Sigma} (\prod d_i^{-d_i})^{1/\Sigma}$ (where $c_i$, $i \in [n]$, are the number of subspaces of dimension at least $i$, and $\Sigma$ denotes $\sum d_i = \sum c_i$) which after taking logs 
is precisely the mutual information between the row and the column of a uniform element from the Young diagram; this can be as high as $\log \log n$
in the above Young diagram configuration 
} 

\subsection{The uncrossing inequality}
\label{sec:uncrossing}

In this section we give two proofs of the uncrossing inequality, Lemma~\ref{lem:uncrossing}. We do so
first using an elegant Gaussian inequality,
Proposition~\ref{prop:magicgaussian}, obtaining the desired inequality by a
limit argument. Afterwards, we give a more direct and elementary proof.

For two positive definite matrices $X,Y \succ 0$ define their \emph{parallel sum} $X:Y$ as $(X^{-1}+Y^{-1})^{-1}$. (Twice that quantity 
is known as the harmonic mean; see~\cite[Chapter 4.1]{BhatiaPSDBook}.) This definition is extended to positive semidefinite matrices $X,Y \succeq 0$
as 
\begin{equation}\label{eq:defofparsumofpd}
X:Y := \lim_{\eps \searrow 0} ((X+\eps I)^{-1}+(Y+\eps I)^{-1})^{-1} \; .
\end{equation}
We note that $\ran(X:Y)=\ran(X)\cap \ran(Y)$ (see~\cite[Theorem 5.15]{schurbook}).

\begin{proposition}\label{prop:magicgaussian}
For any lattice $\lat$, and any two positive semidefinite matrices $X,Y \succeq 0$,
\begin{equation}\label{eq:main}
\rho_X(\lat) \rho_Y(\lat)  \le
\rho_{X:Y}(\lat) 
\rho_{X+Y}(\lat) 
\; .
\end{equation}
\end{proposition}

Although not needed in the sequel, we note that this inequality is invariant under taking duals, 
as can be seen from the Poisson summation formula. 
The proof below follows proofs of related inequalities appearing in~\cite{RSD15}.

\begin{proof}
By continuity, it suffices to consider positive definite $X,Y \succ 0$. 
Moreover, it suffices to consider the case $X=I$. I.e., we will prove the inequality
\begin{equation}\label{eq:main2}
\rho_I(\lat) \rho_{Y}(\lat)  \le
\rho_{(I+Y^{-1})^{-1}}(\lat) 
\rho_{I+Y}(\lat) 
\; .
\end{equation}
The general case in \eqref{eq:main} is now obtained from \eqref{eq:main2}
by plugging $X^{-1/2} \lat$ for $\lat$ and $X^{-1/2} Y X^{-1/2}$ for $Y$.

Let $\lat^{\oplus 2} := \lat \oplus \lat$. The left-hand side of \eqref{eq:main2} can be written as
\begin{equation}\label{eq:maindouble}
\rho\left(
\left( \begin{array}{rr}
I & 0 \\
0 & Y
\end{array}\right); \lat^{\oplus 2} \right)
\; ,
\end{equation}
where for convenience here and in the following we write the covariance inside the parenthesis.

Let $T$ be the $2n \times 2n$ matrix
\[
T := \left( \begin{array}{rr}
I & I \\
I & -I
\end{array}\right)
\; .
\]
By applying $T$ to both the lattice and the covariance matrix, we see that
\eqref{eq:maindouble} is equal to
\begin{equation}\label{eq:mainrotated}
\rho\left(
\left( \begin{array}{rr}
I+Y & I-Y \\
I-Y & I+Y
\end{array}\right); 
T\lat^{\oplus 2} \right)
\; .
\end{equation}
The above covariance matrix can be written as 
\begin{equation}\label{eq:aitken}
\left( \begin{array}{rr}
I+Y & I-Y \\
I-Y & I+Y
\end{array}\right) 
=
\left( \begin{array}{cc}
I & 0 \\
(I-Y)(I+Y)^{-1} & I
\end{array}\right) 
\left( \begin{array}{cc}
I+Y & 0 \\
0 & 4(I+Y^{-1})^{-1}
\end{array}\right) 
\left( \begin{array}{cc}
I & (I-Y)(I+Y)^{-1} \\
0 & I
\end{array}\right) 
\; .
\end{equation}
(This can be seen as Aitken's block-diagonalization formula for the Schur complement~\cite{schurbook}.)
Therefore, defining 
\begin{equation*}
A = 
\left( \begin{array}{cc}
I & 0 \\
-(I-Y)(I+Y)^{-1} & I
\end{array}\right) 
\; ,
\end{equation*}
we obtain that \eqref{eq:mainrotated} is equal to 
\begin{equation}\label{eq:maintwisted}
\rho\left(
\left( \begin{array}{cc}
I+Y & 0 \\
0 & 4(I+Y^{-1})^{-1}
\end{array}\right); 
AT\lat^{\oplus 2} \right)
\; .
\end{equation}
We now analyze the lattice $AT\lat^{\oplus 2}$.
For any $(\vec{x}, \vec{y}) \in \lat^{\oplus 2}$, we have
$T (\vec{x}, \vec{y}) = (\vec{z}, \vec{w})$
where $\vec{z} =  \vec{x} + \vec{y}$ and $\vec{w} = \vec{z} - 2 \vec{y}$.
It follows that
\begin{align*}
T \lat^{\oplus 2} &= \{ (\vec{z}, \vec{w}) \in \lat^{2} : \vec{z} \equiv \vec{w} \bmod 2 \lat\} \\
&= 
\bigcup_{\vec{z} \in \lat} \{ \vec{z} \} \times \{2\lat + \vec{z}\} 
\; ,
\end{align*}
where the union is disjoint. Therefore, 
\begin{align*}
A T \lat^{\oplus 2} &= 
\bigcup_{\vec{z} \in \lat} \{ \vec{z} \} \times \{2\lat + \vec{z} - (I-Y)(I+Y)^{-1} \vec{z} \} 
\; .
\end{align*}
As a result, we obtain that~\eqref{eq:maintwisted} is equal to
\begin{align*}
\sum_{\vec{z} \in \lat} 
\rho_{I+Y}(\vec{z}) \rho_{4(I+Y^{-1})^{-1}} ( 2\lat + \vec{z} - (I-Y)(I+Y)^{-1} \vec{z} )
& \le 
\sum_{\vec{z} \in \lat} 
\rho_{I+Y}(\vec{z}) \rho_{4(I+Y^{-1})^{-1}} ( 2\lat ) \\
& = \rho_{I+Y}(\lat) \rho_{(I+Y^{-1})^{-1}} ( \lat ) \; .
\end{align*}
where we used Lemma~\ref{lem:coset-mass} to conclude that the central coset is heaviest. This completes the proof.
\end{proof}

\begin{claim}\label{clm:asymptoticrho}
For any $n$-dimensional lattice $\lat$ and positive definite matrix $X \succ 0$,
\[
\lim_{s \rightarrow \infty} s^{-n/2} \rho_{sX}(\lat) = \frac{\sqrt{\det(X)}}{\det(\lat)} \; .
\]
More generally, if $V$ is a $d$-dimensional lattice subspace of $\lat$ and
$X \succeq 0$ is a positive semidefinite matrix with $\ran(X) = V$,
\[
\lim_{s \rightarrow \infty} s^{-d/2} \rho_{sX}(\lat) = \frac{\sqrt{\det_V(X)}}{\det(\lat \cap V)} \; .
\]
\end{claim}
\begin{proof}
For the first part of the claim, use Lemma~\ref{lem:coset-mass} to get
\[
\rho_{sX}(\lat) = s^{n/2} \sqrt{\det X} (\det \lat)^{-1} \rho_{(sX)^{-1}}(\lat^*)
\]
and notice that $(sX)^{-1}$ converges to $0$.
The second part of the claim follows easily from the first one by restricting to $V$. 
\end{proof}

By applying Proposition~\ref{prop:magicgaussian} with $sX$ and $sY$ with $s$ going to infinity, and
Claim~\ref{clm:asymptoticrho}, we get the following. 

\begin{corollary}\label{cor:intermediateinequality}
For any lattice $\lat$, lattice subspaces $V, W$, and positive semidefinite $X,Y \succeq 0$ with images $V,W$ respectively,
\[
\frac{\det_V(X)}{\det(\lat \cap V)^2} \cdot 
\frac{\det_W(Y)}{\det(\lat \cap W)^2}
\le
\frac{\det_{V \cap W}(X:Y)}{\det(\lat \cap (V \cap W))^2} \cdot 
\frac{\det_{V+W}(X+Y)}{\det(\lat \cap (V+W))^2} \; .
\]
\end{corollary}

We are finally ready to state and prove the uncrossing inequality. Here and below we adopt the 
convention that zero-dimensional determinants (either of a lattice or of a matrix) are $1$. 

\begin{lemma}
\label{lem:uncrossing}
For any lattice $\lat$, lattice subspaces $V, W$, and positive semidefinite $X,Y \succeq 0$ with images $V,W$ respectively,
\[
\frac{\det_V(X)}{\det(\lat \cap V)^2} \cdot 
\frac{\det_W(Y)}{\det(\lat \cap W)^2}
\le
\frac{\det_{V \cap W}\parens[\big]{\parens[\big]{\frac{X+Y}{2}}^{\cap (V \cap W)}}}{\det(\lat \cap (V \cap W))^2} \cdot 
\frac{\det_{V+W}\parens[\big]{X+Y - \parens[\big]{\frac{X+Y}{2}}^{\cap (V \cap W)}}}{\det(\lat \cap (V+W))^2} \; .
\]
\end{lemma}

\begin{proof}
Using Corollary~\ref{cor:intermediateinequality}, it suffices to prove that
\begin{align}\label{eq:corfinal}
\det_{V \cap W}(X:Y) \cdot 
\det_{V+W}(X+Y)
\le
\det_{V \cap W}\parens[\big]{\parens[\big]{\frac{X+Y}{2}}^{\cap (V \cap W)}}  \cdot 
\det_{V+W}\parens[\big]{X+Y - \parens[\big]{\frac{X+Y}{2}}^{\cap (V \cap W)}} \; .
\end{align}
The matrix determinant formula of Lemma~\ref{lem:det} gives
\begin{align*}
& \det_{V+W}\parens[\big]{X+Y - \parens[\big]{\frac{X+Y}{2}}^{\cap (V \cap W)}} =
\det_{V \cap W}\parens[\big]{\parens[\big]{\frac{X+Y}{2}}^{\cap (V \cap W)}} \cdot
\det_{(V+W)\cap (V\cap W)^\perp}\parens[\big]{(X+Y)^{\downarrow (V+W)\cap (V\cap W)^\perp}}  \\
&\quad =
2^{-\dim(V \cap W)} \det_{V \cap W}\parens[\big]{\parens{X+Y}^{\cap (V \cap W)}} \cdot
\det_{(V+W)\cap (V\cap W)^\perp}\parens[\big]{(X+Y)^{\downarrow (V+W)\cap (V\cap W)^\perp}} \\
&\quad  =2^{-\dim(V \cap W)} \det_{V+W}(X+Y) \; .
\end{align*}
Eq.~\eqref{eq:corfinal} is therefore equivalent to
\[
\det_{V \cap W}(2(X:Y)) 
\le
\det_{V \cap W}\parens[\Big]{\parens[\Big]{\frac{X+Y}{2}}^{\cap (V \cap W)}} \; .
\]
Recalling~\eqref{eq:defofparsumofpd}, we can assume that $X$ and $Y$ are positive definite (and $V$ and $W$ are arbitrary). Taking reciprocals on both sides, the above is equivalent to
\[
\det_{V \cap W}\parens[\Big]{\parens[\Big]{\frac{X^{-1}+Y^{-1}}{2}}^{\downarrow (V \cap W)}}
\ge
\det_{V \cap W}\parens[\Big]{\parens[\Big]{\parens[\Big]{\frac{X+Y}{2}}^{-1}}^{\downarrow (V \cap W)}} \; .
\]
This, in turn, follows from 
\[
\frac{X^{-1}+Y^{-1}}{2}
\ge
\parens[\Big]{\frac{X+Y}{2}}^{-1} \; ,
\]
which is the convexity of the mapping $A \mapsto A^{-1}$ on positive definite matrices
(see, e.g.,~\cite[Eq.~(1.33)]{BhatiaPSDBook}). 
\end{proof}

\subsubsection{The elementary proof}

In this section, we prove the uncrossing inequality more directly. 
We first prove it in Lemma~\ref{lem:ind-uncrossing} when the subspaces are linearly
independent and then derive the general case. We start with the following 
linear-algebraic fact.

\begin{claim}
\label{clm:detformula}
Let $V,W \subseteq \R^n$ be linearly independent subspaces, i.e., 
satisfying $V \cap W = \set{\vec{0}}$, and let $X,Y \succeq 0$ have range $V,W$
respectively. 
Let $T = (O_V,O_W)$ be the $n \times (\dim(V)+\dim(W))$ matrix formed by concatenating the
columns of orthonormal bases $O_V,O_W$ for $V,W$ respectively. 
Then,
\[
 \det_V(X) \det_W(Y) = \det_{V+W}((T^+)^{\T} T^+) \det_{V+W}(X+Y) \; .
\]
\end{claim} 
\begin{proof}
Let $d_1=\dim(V)$ and $d_2=\dim(W)$. Clearly, $T$ has
range $V+W$. Furthermore, since $V,W$ are linearly independent, $T$ is
non-singular. Thus, $T^+ T = I_{d_1+d_2}$. In particular, the following
identities follow,
\begin{equation*}
\begin{split}
T^+ \pi_V = T^+ O_V O_V^\T 
    = \begin{pmatrix} O_V^\T \\ 0^{d_2 \times n} \end{pmatrix} \text{ ,} \quad 
T^+ \pi_W = T^+ O_W O_W^\T 
    = \begin{pmatrix} 0^{d_1 \times n} \\ O_W^\T \end{pmatrix} \text{ .} 
\end{split}
\end{equation*}
Given the above, it is clear that $T^+$ is an isometry when restricted to either
$V$ or $W$, mapping $V$ isometrically into the first $d_1$ coordinates and $W$
isometrically in the last $d_2$.  

Now,
\begin{align*}
\det(T^+ (X+Y) (T^+)^{\T}) 
&= \det(T^+ X (T^+)^{\T} + T^+ Y (T^+)^{\T})  \\
&= \det(T^+ \pi_V X (T^+ \pi_V )^{\T} + T^+ \pi_W Y (T^+ \pi_W )^{\T})  \\
&= \det
\begin{pmatrix}
	O_V^\T X O_V & 0 \\
	0 & O_W^\T Y O_W
\end{pmatrix}	  \\
&= \det_V(X) \det_W(Y) \; .
\end{align*}
On the other hand, letting $O_{V+W}$ denote the $n \times (d_1+d_2)$ matrix whose columns are 
an orthonormal basis for $V+W$, 
\begin{align*}
\det(T^+ (X+Y) (T^+)^{\T}) 
&= \det(T^+ O_{V+W} O_{V+W}^\T (X+Y) O_{V+W} O_{V+W}^\T (T^+)^{\T}) \\
&= \det(O_{V+W}^\T (T^+)^{\T} T^+ O_{V+W})  \det(O_{V+W}^\T (X+Y) O_{V+W})  \\
&=
\det_{V+W}((T^+)^{\T} T^+) \det_{V+W}(X+Y) \; .
\end{align*}
The claim follows. 
\end{proof}

\begin{lemma} 
\label{lem:ind-uncrossing}
Let $V,W$ be lattice subspaces of $\lat \subset \R^n$
satisfying $V \cap W = \set{\vec{0}}$, and let $X,Y \succeq 0$ have range $V,W$
respectively. Then
\[
\frac{\det_V(X)}{\det(\lat \cap V)^2} \frac{\det_W(Y)}{\det(\lat \cap W)^2}
= \frac{\det_{V+W}(X+Y)}{\det((\lat \cap V) + (\lat \cap W))^2}
\leq \frac{\det_{V+W}(X+Y)}{\det((\lat \cap (V+W))^2} \text{ .}
\] 
\end{lemma}
\begin{proof}
For the inequality, note that since $(\lat \cap V) + (\lat \cap W)
\subseteq \lat \cap (V+W)$ and the dimensions match, the determinant of the
second lattice is no larger than the first.

We now prove the first identity. 
Let $B_V,B_W$ denote bases of $\lat \cap V$, $\lat \cap W$ respectively.  
Then 
\[
  \det(\lat \cap V)^2 = \det_V(B_V B_V^\T) \text{ and }
	\det(\lat \cap W)^2 = \det_W(B_W B_W^\T) \; .
\]
Moreover, by linear independence, we see that $B_{V,W} = (B_V,B_W)$ is a basis for $(\lat
\cap V) + (\lat \cap W)$, and therefore, 
\[
  \det((\lat \cap V) + (\lat \cap W))^2 = \det_{V+W}(B_{V,W} B_{V,W}^\T) = \det_{V+W}(B_{V} B_V^\T + B_W B_{W}^\T) \; .
\]
The result now follows by applying Claim~\ref{clm:detformula} twice, once to the numerators and once to the denominators. 
\end{proof}

\begin{lemma}
\label{lem:mat-max}
Let $A \succeq 0$ denote an $n \times n$ matrix with range $W$. Then
\[
\max \set{\det_W(X)\det_W(Y) : X,Y \succeq 0, X+Y = A} = \det_W(A/2)^2 \text{.} 
\]
\end{lemma}
\begin{proof}
If $A=0$, the statement is trivial, so we may assume $A \neq 0$. Furthermore,
by a change of basis, we can assume that $W = \R^n$. Now taking $X,Y$ as above,
\begin{align*}
\det(X)\det(Y) &= \left(\det(X)^{1/(2n)}\det(Y)^{1/(2n)}\right)^{2n} \\
  &\leq \left(\frac{1}{2} \det(X)^{1/n} + 
              \frac{1}{2} \det(Y)^{1/n}\right)^{2n} 
        \quad \left(\text{ by AM-GM }\right) \\
  &\leq \left(\det(\frac{1}{2} X + \frac{1}{2} Y)^{1/n}\right)^{2n} 
     \quad \left(\text{ by concavity, see Lemma~\ref{lem:proj-det-props} }\right) \\
  &= \det(A/2)^2 \text{ ,}
\end{align*}
as needed.
\end{proof}

\begin{proof}[Alternate proof of Lemma~\ref{lem:uncrossing}]
We will use the
following convenient notation: for a linear subspace $V \subseteq \R^n$ and
any set $A \subseteq \R^n$, we write $A/V$ to denote $\pi_{V^\perp}(A)$. 
Let $R=V \cap W$ and $S = V+W$. Starting from the left-hand side, applying the 
matrix determinant Lemma~\ref{lem:det}, we get
\[
\frac{\det_V(X)}{\det(\lat \cap V)^2} \cdot 
\frac{\det_W(Y)}{\det(\lat \cap W)^2} =
\frac{\det_R(X^{\cap R})\det_{V/R}(X^{\da R^\perp})}{\det(\lat \cap R)^2 \det((\lat
\cap V)/R)^2} \cdot 
\frac{\det_R(Y^{\cap R})\det_{W/R}(Y^{\da R^\perp})}{\det(\lat \cap R)^2 \det((\lat
\cap W)/R)^2} \text{ .}
\]
Since $V/R$ and $W/R$ are linearly independent and $V/R+W/R = S/R$,
applying Lemma~\ref{lem:ind-uncrossing}, recalling that projections are linear,
the right-hand is less than or equal to 
\[
\frac{\det_R(X^{\cap R})\det_R(Y^{\cap R})}{\det(\lat \cap R)^4} \cdot 
\frac{\det_{S/R}((X+Y)^{\da R^\perp})}{\det((\lat \cap S)/R)^2} =
\frac{\det_R(X^{\cap R})\det_R(Y^{\cap R})}{\det(\lat \cap R)^2} \cdot 
\frac{\det_{S/R}((X+Y)^{\da R^\perp})}{\det(\lat \cap S)^2} \text{ .}
\]
By Lemma~\ref{lem:mat-max}, we have that 
\[
\det_R(X^{\cap R})\det_R(Y^{\cap R}) \leq \det_R((X^{\cap R}+Y^{\cap R})/2)^2
\leq \det_R((X+Y)^{\cap R}/2)^2 \text{, }
\]
where the last inequality follows by Lemma~\ref{lem:slice} part 3. Moreover,
one can readily verify from Definition~\ref{def:schur-compl} that 
\[
(X+Y - (X+Y)^{\cap R}/2)^{\cap R} = (X+Y)^{\cap R}/2 ~~\text{ and }~~
(X+Y - (X+Y)^{\cap R}/2)^{\da R^\perp} = (X+Y)^{\da R^\perp} \text{ .}
\]
Putting everything together,
\begin{align*}
\frac{\det_R(X^{\cap R})\det_R(Y^{\cap R})}{\det(\lat \cap R)^2} \cdot 
\frac{\det_{S/R}((X+Y)^{\da R^\perp})}{\det(\lat \cap S)^2} 
&\leq \frac{\det_R((X+Y)^{\cap R}/2)^2}{\det(\lat \cap R)^2} \cdot 
\frac{\det_{S/R}((X+Y)^{\da R^\perp})}{\det(\lat \cap S)^2} \\
&= \frac{\det_R((X+Y)^{\cap R}/2)}{\det(\lat \cap R)^2} \cdot 
\frac{\det_S(X+Y-(X+Y)^{\cap R}/2)}{\det(\lat \cap S)^2} \text{ ,}
\end{align*}
where the last equality follows by the matrix determinant lemma. The lemma thus
follows.
\end{proof}

\section{Weaker variants of the main conjecture}
\label{sec:variants}

\newcommand{\ellipse}{{\scalebox{1}[0.7]{$\circ$}}}

\newcommand{\etarelax}{\eta}

In an effort to make progress on the main conjecture, it is natural to consider weaker forms of it
and hope that they would be easier to prove. 
In this section we describe four such forms, some 
quite natural in their own right, and describe their relationship. 
Those weaker forms are obtained by relaxing the quantity appearing 
in the right-hand side of Eq.~\eqref{eq:finalgoal} in the main conjecture,
which we call here $\etarelax_{\det}$.
The first two involve the quantities $\etarelax_\rho$ and $\etarelax_\mu$ where instead
of asking the sublattice $\lat^* \cap W$ to have small determinant, we ask it
to have large Gaussian mass (or equivalently, many lattice points in a ball) 
in the case of $\etarelax_\rho$, or 
its dual to have large covering radius in the case of $\etarelax_\mu$.
The remaining two quantities, 
$\etarelax^{\ellipse}_\rho$ and $\etarelax^{\ellipse}_\mu$, are
obtained from the previous two by replacing the lattice subspace 
by an arbitrary ellipsoid. This avoids the discreteness inherent
in the set of lattice subspaces, and might be more amenable to a proof. 

\begin{definition}\label{def:weakerquantities}
For a lattice $\lat$ we define the following quantities, where $W$ always ranges over
all lattice subspaces of $\lat^*$ of positive dimension. 
\begin{align*}
\etarelax_{\det}(\lat) &= \max_{W}(\det(\lat^* \cap W))^{-1/\dim(W)}  \\
\etarelax_\rho(\lat) &= \max_{s>0,W} s \cdot (\log \rho_{1/s^2}(\lat^* \cap W) / \dim W)^{1/2} \\
\etarelax_\mu(\lat) &= \max_{W} \mu(\pi_W(\lat)) / \sqrt{\dim W} \\
\etarelax^{\ellipse}_\rho(\lat) &= \sup_{X \succ 0} (\log \rho_{X}(\lat^*) / \tr X)^{1/2} \\
\etarelax^{\ellipse}_\mu(\lat) &= \sup_{R \text{ non-singular}} \mu(R\lat) / (\tr(R^\T R))^{1/2}  =
\sup_{R \text{ non-singular}, \|R\|_F \le 1} \mu(R\lat) 
\end{align*}
\end{definition}

\begin{remark}
It is easy to check from the definitions that all six ``$\eta$-type'' parameters are
positively homogeneous, that is for $\tilde{\eta} \in \set{\eta_{\det},
\eta_\rho, \eta_\mu, \eta^{\ellipse}_\rho, \eta^{\ellipse}_\mu, \eta}$ as
above, we have $\tilde{\eta}(\lambda \lat) = \lambda \tilde{\eta}(\lat)$ for
$\lambda > 0$. We will show a stronger property in 
Lemma~\ref{lem:eta-continuity} below. 
Also, in the definitions of $\etarelax^{\ellipse}_\rho$ and $\etarelax^{\ellipse}_\mu$
we can equivalently take the supremum over all (nonzero) matrices $X,R$,
including singular ones. This holds due to standard limit arguments.
\onote{Daniel's proof: We first show that we can assume that $R\lat$ is
discrete, i.e.~a lattice, without decreasing $\mu(R\lat)$. 
Assume not, then let $W = \cap_{\eps > 0} {\rm span}(R\lat \cap \eps B_2^n)$,
noting that $\dim(W) \geq 1$ by assumption. Now let $\pi$ denote the $n \times
n$ projection matrix onto $W^\perp$. From here, it is easy to check that $\|\pi
R\|_F \leq \|R\|_F$, $\mu(\pi R \lat) = \mu(R \lat)$ (note we can move any
target along $W$ at $0$ cost), and that $\pi R \lat$ is discrete (since
otherwise $W$ would be bigger).}
\end{remark}

\newcommand{\vertle}{\rotatebox{270}{$\le$}}

\begin{theorem}
\label{thm:conj-sandwich}
For any lattice $\lat$ we have the inequalities 
\begin{align*}
\frac{1}{C_\eta(n)} \eta(\lat) \le \etarelax_{\det}(\lat) \lesssim &\etarelax_\rho(\lat) \lesssim \etarelax_\mu(\lat) \\[-1.5em]
& ~~~\vertle  \quad \qquad \vertle \\
&\etarelax^{\ellipse}_\rho(\lat) \lesssim \etarelax^{\ellipse}_\mu(\lat) \lesssim  \eta(\lat) \; .
\end{align*}
\end{theorem}

\begin{proof}
The seven inequalities are proved as follows.
\begin{itemize}
\item
$\eta(\lat) \le C_\eta(n) \etarelax_{\det}(\lat)$:
This is precisely the definition of $C_\eta$. 
\item
$\etarelax_{\det}(\lat) \lesssim \etarelax_\rho(\lat)$:
By Minkowski's first theorem, Theorem~\ref{thm:mink-first}, for any lattice subspace $W$, $\lat^*$ has $\exp(\dim W)$ many points
in a ball of radius $c \sqrt{\dim W} \cdot \det(\lat^* \cap W)^{1/\dim W}$, where $c>0$ is
a universal constant. Therefore, taking $s = c' \cdot \det(\lat^* \cap W)^{-1/\dim W}$ in the definition of $\etarelax_\rho$ for a 
small enough constant $c'>0$ proves the inequality. 
\item
$\etarelax_\rho(\lat) \lesssim \etarelax_\mu(\lat)$:
Since $(\pi_W(\lat))^* = \lat^* \cap W$, this follows from Lemma~\ref{lem:lowerboundonmu} below.
\item
$\etarelax_\rho(\lat) \le \etarelax^{\ellipse}_\rho(\lat)$:
Let $W$ be any lattice subspace of $\lat^*$ and consider $X = s^{-2} \pi_W + \eps \pi_{W^\perp}$. Then as $\eps$ goes to zero, 
$\tr X$ converges to $s^{-2} \dim W$ and 
$\log \rho_{X}(\lat^*)$ converges to $\log \rho_{1/s^2}(\lat^* \cap W)$.
\item 
$\etarelax_\mu(\lat) \le \etarelax^\ellipse_\mu(\lat)$: This is similar to the previous inequality. 
Let $W$ be any lattice subspace of $\lat^*$ and consider $R = \pi_W + \eps \pi_{W^\perp}$.
Then as $\eps$ goes to zero, 
$\tr(R^\T R)$ converges to $\dim W$ and 
$\mu(R\lat)$ converges to $\mu(\pi_W(\lat))$. 
\item
$\etarelax^{\ellipse}_\rho(\lat) \lesssim \etarelax^{\ellipse}_\mu(\lat)$:
By Lemma~\ref{lem:lowerboundonmu} with $s=1$, for any positive semidefinite $X$, 
\[
\log \rho_X(\lat^*) = 
\log \rho(X^{-1/2} \lat^*) = \log \rho((X^{1/2} \lat)^*) \lesssim \mu(X^{1/2} \lat)^2 \;, 
\]
and so the inequality follows by taking
$R=X^{1/2}$ in the definition of $\etarelax^{\ellipse}_\mu$. 
\item
$\etarelax^{\ellipse}_\mu(\lat) \lesssim  \eta(\lat)$:
Observe that the definition of $\mu_{\rm{sm}}$ in~\eqref{eq:smooth-mu} can be equivalently written as 
\[
\mu_{\rm{sm}}(\lat) = \min_{R} \sqrt{\tr(R^\T R)} \cdot \eta(R^{-1} \lat) \text{.}
\]
Then, using Theorem~\ref{thm:smooth-mu-bnd}, we obtain
\[
 \sup_{R} \frac{\mu(R \lat)}{\sqrt{\tr(R^\T R)}} \lesssim 
 \sup_{R} \frac{\mu_{sm}(R \lat)}{\sqrt{\tr(R^\T R)}} 
\le 
\sup_{R} \frac{\sqrt{\tr(R^\T R)} \cdot \eta(R^{-1} R \lat) }{\sqrt{\tr(R^\T R)}}
= 
\eta(\lat)
\;.
\]

\end{itemize}
\end{proof}

\begin{lemma}\label{lem:lowerboundonmu}
For any lattice $\lat$ and $s>0$,
\begin{align}\label{eq:lowerboundonmu}
\mu(\lat) \ge \max_{s>0} s \cdot \sqrt{\frac{\log \rho(s\lat^*)}{\pi}}
\; .	
\end{align}
\end{lemma}
\begin{proof}
Below we will prove that
\begin{align}\label{eq:mulowerbound}
\mu(\lat) \ge \sqrt{\frac{\log \rho(\lat^*)}{\pi}} \text{ .}
\end{align}
Noting that $\mu(\lat/s) = \mu(\lat)/s$ and $(\lat/s)^* = s \lat^*$, the lemma
follows by considering all scalings. Using Lemma~\ref{lem:banacosh},
\begin{align*}
1 = \int_{\R^n} \rho(\vec x) d \vec x &= 
\int_{\R^n/\lat} \rho(\lat + \vec x) d \vec x \\
&\ge
\rho(\lat) \int_{\R^n/\lat} \rho(\dist(\vec x,\lat)) d \vec x \\
& \ge \rho(\lat) \det(\lat) \rho(\mu(\lat)) =
\rho(\lat^*) e^{-\pi \mu(\lat)^2 } \; ,
\end{align*}
which implies~\eqref{eq:mulowerbound} by rearranging.
\end{proof}

One might wonder if the inequality in Lemma~\ref{lem:lowerboundonmu}
also holds in reverse, say up to polylogarithmic factors. It turns out that this
is an easy consequence of the $\ell_2$ KL conjecture, and we may therefore
refer to it as \emph{the weak-KL conjecture}. Intuitively, it shows that the covering radius
can be characterized up to polylogarithmic factors by the norm distribution of points
in the dual lattice. To make this more precise, we show in Lemma~\ref{lem:mu-lb-pc} below 
that the maximum in~\eqref{eq:lowerboundonmu} has a simple
equivalent point counting formulation. In particular, we show that it is in
essence the minimum $s > 0$ for which the number of points at any radius $r$ is
bounded by a function of the form $e^{sr}$. We note the interesting similarity
with the point counting formulation of the smoothing parameter in Lemma~\ref{lem:pointcountingeta}, which gives a bound of the form $e^{(sr)^2}$.

\paragraph{Weaker forms of the main conjecture.}
By replacing the quantity $\etarelax_{\det}$ appearing in the main
conjecture with one of other four quantities appearing in Definition~\ref{def:weakerquantities},
we get weaker forms of the main conjecture.
Some of those are quite natural. 
For instance, by using $\etarelax_\mu$, we obtain a conjecture 
that can be interpreted as saying that if $\lat$ is not smooth, 
then there is a certificate for that in the form of a projection where 
the Gaussian ``does not reach'' the covering radius of the projected lattice. 
Also, the conjecture obtained from $\etarelax_\rho^\ellipse$ is quite appealing
as both sides of the inequality only involve the Gaussian mass.
Finally, by using $\etarelax^{\ellipse}_\mu(\lat)$ we obtain the weakest
form of the main conjecture. It turns out that this weakest form is
sufficient for the applications in Section~\ref{sec:mixingtime} and Section~\ref{sec:gapspp},
and we therefore define it explicitly. 

\begin{conjecture}\label{con:covweaketa}(Weak conjecture)
Let $C^{(\mu,\ellipse)}_\eta(n)>0$ be the smallest number such that for 
any $n$-dimensional lattice $\lat \subset \R^n$,
\begin{align}\label{eq:covweaketagoal}
   \eta(\lat) \le 
	 C^{(\mu,\ellipse)}_\eta(n) \etarelax^{\ellipse}_\mu(\lat) \text{ .}
\end{align}
Then $C^{(\mu,\ellipse)}_\eta(n) \le \poly \log n$.
\end{conjecture}

\noindent
In Section~\ref{sec:kl-sandwich} we will show that the
weak conjecture combined with the $\ell_2$-KL conjecture imply 
the main conjecture. \dnote{6/21: maybe more specific about the equivalence?} 


\paragraph{Point counting formulation.}
Here we prove the claim made after Lemma~\ref{lem:lowerboundonmu}.

\begin{lemma}
\label{lem:mu-lb-pc}
Let $\lat \subset \R^n$ be an $n$-dimensional lattice. Then
\[
2 \sqrt{\pi} e^{3/2} \max_{r > 0}
 \frac{\log(|\lat \cap r\Ball_2^n|)}{2\pi r} \geq \max_{s > 0} s \sqrt{\frac{\log \rho(s\lat)}{\pi}} \geq \max_{r > 0}
 \frac{\log(|\lat \cap r\Ball_2^n|)}{2\pi r} \text{.}
\]
\end{lemma}
\begin{proof}
We first prove the second inequality. Fixing $r > 0$, we have
that
\begin{align*}
\max_{s > 0} \frac{1}{\pi} s^2 \log \rho(s\lat) 
&\geq \max_{s > 0} \frac{1}{\pi} s^2\log \rho(s(\lat \cap r\Ball_2^n)) \\
&\geq \max_{s > 0} \frac{1}{\pi} s^2(\log(|\lat \cap r\Ball_2^n|) - \pi s^2 r^2) \\
&= \max_{t > 0} \frac{1}{\pi} (t\log(|\lat \cap r\Ball_2^n|) - \pi t^2 r^2) \\
 &= \frac{\log(|\lat \cap r\Ball_2^n|)^2}{4\pi^2 r^2} \text{, } 
\end{align*}
where the last equality follows by setting $t=\log(|\lat \cap
r\Ball_2^n|)/(2\pi r^2)$. 

We now prove the first inequality. First note that the expression 
\[
\max_{r > 0} \frac{\log(|\lat \cap r\Ball_2^n|)}{2\pi r} \text{ ,}
\]
is simply the minimum number $\alpha > 0$ such that $|\lat \cap r\Ball_2^n|
\leq e^{2\pi \alpha r} ~~ \forall r \geq 0$.

Now fix $s > 0$. From here,
\begin{align*}
\rho(s\lat) &= \sum_{\vec{y} \in \lat} e^{-\pi\|s\vec{y}\|^2_2} \\
    &= \int_0^\infty 
     |\lat \cap r\Ball_2^n| \cdot 2\pi s^2 r \cdot e^{-\pi s^2 r^2} dr \\
    &\leq \int_0^\infty 
      e^{2\pi\alpha r} \cdot 2\pi s^2 r \cdot e^{-\pi s^2 r^2} dr \\
    &= -e^{2\pi \alpha r} e^{-\pi s^2 r^2}\big|_{r=0}^\infty + 
          \int_0^\infty 2\pi\alpha e^{2\pi \alpha r} e^{-\pi s^2 r^2} dr \\
    &= 1 + e^{\pi(\alpha/s)^2} \int_0^\infty 
            2\pi\alpha e^{-\pi (s r - \alpha/s)^2} dr \\
    &\leq 1 + e^{\pi(\alpha/s)^2} \int_{-\infty}^\infty 
              2\pi\alpha e^{-\pi (s r)^2} dr \\
    &= 1 + (2\pi \alpha) e^{\pi(\alpha/s)^2}/s \text{ .}
\end{align*}

We will distinguish two cases for $s$.

\paragraph{Case 1: $s \geq 2\pi e^2 \alpha $.} Given the above, if $s \geq 2\pi e^2 \alpha$,
we clearly have that $\rho(s \lat) \leq 1+1/e$, and hence $\rho(s \lat
\setminuszero) \leq 1/e$. Parametrizing $s = 2\pi e^2 \alpha t$ 
for $t \geq 1$, by Lemma~\ref{lem:mass-decrease}, we have that
\begin{align*}
s \sqrt{ \frac{\log \rho(s \lat)}{\pi} } 
&= 2\pi e^2 \alpha  \cdot t \sqrt{\frac{\log(1 + \rho_{1/t^2}(2\pi e^2 \alpha \cdot \lat
\setminuszero))}{\pi}} \\
&\leq 2\sqrt{\pi} e^2 \alpha  \cdot t e^{-t^2/2} \leq 2\sqrt{\pi}e^{3/2}\alpha \text{,}
\end{align*}
noting that the above is maximized at $t=1$.

\paragraph{Case 2: $s \leq 2 \pi e^2 \alpha $.} In this case, parametrizing $s =
2\pi e^2 \alpha t$ for $t \in [0,1]$, we have
\begin{align*}
s \sqrt{ \frac{\log \rho(s\lat)}{\pi} }
&\leq s \sqrt{\frac{\log(1 + (2\pi \alpha) e^{\pi (\alpha/s)^2}/s)}{\pi}} \\
&= 2 \sqrt{\pi} e^2 \alpha \cdot t \cdot \sqrt{\log(1 + e^{\pi (2\pi e^2 t)^{-2}}/(te^2))} \\
&\leq 2 \sqrt{\pi} e^2 \alpha \cdot t \cdot 
         \sqrt{\log((e^2+1)e^{\pi (2\pi e^2 t)^{-2}}/(te^2))} \\
&\leq 2 \sqrt{\pi} e^2 \alpha 
         \sqrt{t^2 \log(1+1/e^2) + \pi (2\pi e^2)^{-2} - (t^2/2) \log t^2} \\
&\leq 2 \sqrt{\pi} e^2 \alpha 
         \sqrt{\log(1+1/e^2) + \pi (2\pi e^2)^{-2} + 1/(2e)} \\
&\leq 2 \sqrt{\pi} e^{3/2} \alpha \text{ .} 
\end{align*}
\end{proof}

\paragraph{Continuity of the $\eta$ parameters.}
We end this section by showing basic continuity properties of the $\eta$-type parameters.
\begin{lemma}
\label{lem:eta-continuity}
Let $\tilde{\eta} \in \set{\eta_{\det}, \eta_\rho, \eta_\mu,
\eta^{\ellipse}_\rho, \eta^{\ellipse}_\mu, \eta}$. For $\lat \subset \R^n$ an
$n$-dimensional lattice, for an invertible transformation $T \in \R^{n \times n}$,
\[
\|T^{-1}\|^{-1}\tilde{\eta}(\lat) \leq \tilde{\eta}(T\lat) \leq \|T\|
\tilde{\eta}(\lat) \text{ .}
\] 
In particular, $\tilde{\eta}(\lat) = \tilde{\eta}(T\lat)$ if $T$ is an
orthogonal transformation. Furthermore, the map $T \mapsto
\tilde{\eta}(T\lat)$, with domain equal to the space of $n \times n$ invertible
matrices, is continuous.
\end{lemma}
\begin{proof}
We first note that it suffices to prove $\tilde{\eta}(T\lat) \leq \|T\|
\tilde{\eta}(\lat)$. To recover the lower bound, simply rearrange the inequality
$\tilde{\eta}(\lat) = \tilde{\eta}(T^{-1}T\lat) \leq \|T^{-1}\|
\tilde{\eta}(T\lat)$. In what follows, $W$ will always range over lattice
subspaces of $\lat^*$.

\begin{enumerate}
\item[$\eta_{\det}$:]~ 
$\eta_{\det}(T\lat) = \max_{W}
\det(T^{-\T}(\lat^* \cap W))^{-1/\dim(W)}$.

Let $B_W$ denote a basis of $\lat^* \cap W$. Then $\det(T^{-\T}(\lat^* \cap
W))^2 = \det(B_W^\T (T^\T T)^{-1} B_W)$. Since 
\[
\|T\|^{-2} I \preceq (T^\T T)^{-1} \text{ ,}
\]
we get that
\[
\|T\|^{-2\dim(W)} \det(B_W^\T B_W) \leq \det(B_W^\T (T^\T T)^{-1} B_W) \text{ .}
\]
In particular,
\[
\det(T^{-\T}(\lat^* \cap W))^{-1/\dim(W)}
\leq  \|T\| \det(\lat^* \cap W)^{-1/\dim(W)} \text{, }
\]
and hence $\eta_{\det}(T\lat) \leq \|T\| \eta_{\det}(\lat)$, as needed.

\item[$\eta_\rho$:] Since $\|T\|^{-1}\|\vec{x}\| \leq \|T^{-\T}\vec{x}\|$,
we have that
\begin{align*}
\eta_\rho(T\lat) &= \max_{W,s \geq 0}
s\sqrt{\frac{\log \rho_{1/s^2}(T^{-\T}(\lat^* \cap
W))}{\dim(W)}} 
\leq \max_{W,s \geq 0}
s\sqrt{\frac{\log \rho_{1/s^2}(\|T\|^{-1}(\lat^* \cap
W))}{\dim(W)}} \\
&= \max_{W,s \geq 0}
\|T\| s\sqrt{\frac{\log \rho_{1/s^2}(\lat^* \cap
W)}{\dim(W)}} = \|T\| \eta_\rho(\lat) \text{, }
\end{align*}
as needed.

\item[$\eta_\mu$:] 
Let $W$ be a lattice subspace of $\lat^*$ and $T^{-\T}W$ be the corresponding
lattice subspace of $(T\lat)^* = T^{-\T} \lat^*$. It suffices to show that
$\mu(\pi_{T^{-\T} W} (T \lat)) \le \|T\| \mu(\pi_W(\lat))$. 
Let $V$ be the orthogonal complement of $W$, and notice that $TV$ is then
the orthogonal complement of $T^{-\T}W$. Take any $\vec{x} \in \R^n$. 
By definition, there exist $\vec{y} \in \lat$ and $\vec{z} \in V$
such that 
\[
\| T^{-1} \vec{x} - (\vec{y}+\vec{z})\|_2 \le \mu(\pi_W(\lat)) \; . 
\]
Therefore, 
\[
\| \vec{x} - (T\vec{y}+T\vec{z})\|_2 \le \|T\| \mu(\pi_W(\lat)) \; . 
\]
But $T \vec{y} \in T \lat$ and $T \vec{z} \in TV$, the orthogonal complement 
of $T^{-\T}W$, as needed. 

\item[$\eta^\ellipse_\rho$:] 
\begin{align*}
\eta^{\ellipse}_{\rho}(T\lat)^2 
&= \max_{X \succ 0} \frac{\rho_X(T^{-\T} \lat^*)}{\tr(X)} 
= \max_{X \succ 0} \frac{\rho_{T^\T X T}(\lat^*)}{\tr(X)} 
= \max_{X \succ 0} \frac{\rho_{X}(\lat^*)}{\tr(X(T^\T T)^{-1})} \\
&\leq \|T\|^2 \max_{X \succ 0} \frac{\rho_{X}(\lat^*)}{\tr(X)} = \|T\|^2
\eta^\ellipse_\rho(\lat)^2 \text{ .}
\end{align*}

\item[$\eta^\ellipse_\mu$:]
\begin{align*}
\eta^{\ellipse}_{\mu}(T\lat)^2
&= \max_{R \text{ non-singular }} \frac{\mu(R T\lat)^2}{\tr(R^\T R)} 
= \max_{R \text{ non-singular }} \frac{\mu(R\lat)^2}{\tr(R^\T R (T^\T
T)^{-1})} \\
&\leq \|T\|^2 \max_{R \text{ non-singular }} \frac{\mu(R\lat)^2}{\tr(R^\T R)} =
\|T\|^2 \eta^{\ellipse}_{\mu}(\lat)^2 \text{. }
\end{align*} 
\item[$\eta$:] 
\begin{align*}
\eta(T\lat) &= \min \set{s \geq 0: \rho_{1/s^2}(T^{-\T} \lat^*) \leq 1} \\
&\leq \min \set{s \geq 0: \rho_{1/s^2}(\|T\|^{-1} \lat^*) \leq 1} \\
&= \|T\| \min \set{s \geq 0: \rho_{1/s^2}(\lat^*) \leq 1} 
=\|T\| \eta(\lat) \text{ .}
\end{align*}
\end{enumerate}
For the ``in particular'' part, note that for an orthogonal transformation $T$ we have
$\|T^{-1}\| = \|T\| = 1$, as needed. For the ``furthermore'' part, let $R \in \R^{n
\times n}$ be an invertible matrix. We must show that the map $T \mapsto
\tilde{\eta}(T\lat)$ is continuous at $R$.  Define $U_\delta = \set{(I+\Delta)
R: \Delta \in \R^{n \times n}, \|\Delta\| < \delta}$ for $\delta < 1$.  Here, it
is easy to see that $(U_{1/2^n})_{n=1}^\infty$ forms a convergent sequence of
open neighborhoods around $R$ in the space of invertible matrices.  Continuity
now follows from the first part since
\begin{align*}
(1-\delta) \tilde{\eta}(R\lat) 
&\leq \|(I+\Delta)^{-1}\|^{-1} \tilde{\eta}(R\lat) \\
&\leq \tilde{\eta}((I+\Delta)R\lat) \\
&\leq \|I+\Delta\| \tilde{\eta}(R\lat) \leq (1+\delta) \tilde{\eta}(R\lat) \text{ , }
\end{align*}
whenever $\|\Delta\| < \delta < 1$.
\end{proof}


\subsection{The Kannan-Lov{\'a}sz conjecture and the weak conjecture}
\label{sec:kl-sandwich}

The goal of this section is to show that the $\ell_2$ Kannan-Lov{\'a}sz conjecture is
equivalent to a polylogarithmic bound on the worst case ratio between
$\eta^{\ellipse}_\mu$ and $\eta_{\det}$. 

\begin{definition} 
Let $C^{\det}_{(\mu,\ellipse)}(n)$ denote the smallest number such that
for any $n$-dimensional lattice $\lat \subset \R^n$,
\[
\eta^\ellipse_\mu(\lat) \leq C^{\det}_{(\mu,\ellipse)}(n) \cdot \eta_{\det}(\lat)
\text{ .}
\]
\end{definition}

The main equivalence is given below.

\begin{theorem}
\label{thm:kl-sandwich}
For $n \ge 1$, $C^{\det}_{(\mu,\ellipse)}(n) \leq C_{KL}(n) \lesssim \log n
\cdot C^{\det}_{(\mu,\ellipse)}(n)$.
\end{theorem}

 \dnote{6/21: when you're done with section 6, maybe just find a place to
remind people that KL can now equivalently be thought of as a containment
factor. Just a sentence would be enough I think. }
The first inequality immediately implies that the weak conjecture 
together with the $\ell_2$-KL conjecture imply the main conjecture.
The second inequality is a sharpened version of our main result, 
Theorem~\ref{thm:mink-to-kl}, since 
$C^{\det}_{(\mu,\ellipse)}(n) \lesssim C_\eta(n)$ by Theorem~\ref{thm:conj-sandwich}.

\begin{proof}
We prove $C^{\det}_{(\mu,\ellipse)}(n) \leq C_{KL}(n)$. Let $\lat \subset \R^n$
be an $n$-dimensional lattice. Let $R \in \R^{n \times n}$ be a non-singular
matrix. Noting that $(R\lat)^* = R^{-\T} \lat^*$, by the definition of
$C_{KL}(n)$, there exists a lattice subspace $W$ of $\lat^*$ such that
\begin{align*}
C_{KL}(n)^{-1} \mu(R\lat) 
&\leq \frac{\sqrt{\dim(W)}}{\det(R^{-\T}(\lat^* \cap W))^{1/\dim(W)}} \\
&=
\det_W(R^\T R)^{1/(2\dim(W))} \cdot \frac{\sqrt{\dim(W)}}{\det(\lat^* \cap W)^{1/\dim(W)}} \\
&\leq
 \frac{\sqrt{\tr(R^\T R \pi_W)}}{\det(\lat^* \cap W)^{1/\dim(W)}} 
 \leq \frac{\sqrt{\tr(R^\T R)}}{\det(\lat^* \cap W)^{1/\dim(W)}} \text{, }
\end{align*}
where the first equality is by Claim~\ref{clm:detlattransform} and the second
inequality follows from the AM-GM inequality. By rearranging, we get
$\mu(R\lat)/\|R\|_F \leq C_{KL}(n)\cdot 1/\det(\lat^* \cap W)^{1/\dim(W)}$.
Thus, by maximizing over $R$ we get $\eta^{\ellipse}_\mu(\lat) \leq
C_{KL}(n) \cdot \eta_{\det}(\lat)$ as needed.

The proof of the second inequality $C_{KL}(n) \lesssim \log n \cdot C^{\det}_{(\mu,\ellipse)}(n)$
is very similar to the proof of Theorem~\ref{thm:mink-to-kl} in Section~\ref{sec:proofoutline}. 
The only change is to replace $\mu_{\rm sm}$ with a tighter bound on $\mu$. Namely, 
for any $n$-dimensional lattice $\lat \subset \R^n$ we have,
\begin{align*}
\mu(\lat)^2 
&\leq 
\min \set{\tr(X): X \succeq 0, ~  \forall R \text{ non-singular }, \tr(X R^\T R) \ge \mu(R \lat)^2} \\
& = 
\min \set{\tr(X): X \succeq 0, ~  \forall R \text{ non-singular }, \tr(R^\T R) \ge \mu(R X^{-1/2} \lat)^2} \\
& = 
\min \set{\tr(X): X \succeq 0, ~ \eta^\ellipse_\mu(X^{-1/2} \lat) \le 1 }  \\
& \le 
\min \set{\tr(X): X \succeq 0, ~ \eta_{\rm det}(X^{-1/2} \lat) \le 1/C^{\det}_{(\mu,\ellipse)}(n) } \\
& = 
C^{\det}_{(\mu,\ellipse)}(n)^2 \min \set{\tr(X): X \succeq 0, ~ \eta_{\rm
det}(X^{-1/2} \lat) \le 1 } \quad (\text{by homogeneity of $\eta_{\det}$}) 
\\
& = 
C^{\det}_{(\mu,\ellipse)}(n)^2 \mu_{\rm det}(\lat)^2 \; ,
\end{align*}
where the first inequality is immediate (take $R=I$),\footnote{%
In fact, as we shall see in Section~\ref{sec:sm-cov-body}, this inequality is an equality up to a universal constant.
} 
and the last equality is by the
definition of $\mu_{\textnormal{det}}(\lat)$ in~\eqref{eq:det-mu} and
Claim~\ref{clm:detlattransform}. We now
complete the proof using Theorems~\ref{thm:det-mu-dual}
and~\ref{thm:subspace-rounding}, exactly as in the proof of
Theorem~\ref{thm:mink-to-kl}.
\end{proof}

\section{The smooth covariance bodies and the covering radius}
\label{sec:sm-cov-body}

In this section, we explain how the main and weak conjectures can be
equivalently reformulated in terms of \emph{containment factors} of a single
family of convex bodies. We note that this viewpoint was already implicitly used
in Theorem~\ref{thm:mink-to-kl} (the reduction of the $\ell_2$-KL conjecture to
the main conjecture), 
and thus the goal here is to make this
viewpoint more explicit. 
Going further, we shall also develop the tools needed to give an equivalent
characterization of the weak conjecture in terms of the tightness of the 
smooth-$\mu$ bound (see Section~\ref{sec:musmoothtight}).

\begin{definition}[Smooth Covariance Bodies]
\label{def:sm-cov}
Let $\tilde{\eta} \in \set{\eta_{\det}, \eta_\rho, \eta_\mu,
\eta^{\ellipse}_\rho, \eta^{\ellipse}_\mu, \eta}$ as in
Definition~\ref{def:weakerquantities}. For an $n$-dimensional lattice $\lat
\subset \R^n$, define its $\tilde{\eta}$-smooth covariance body
\[
{\rm sm}_{\tilde{\eta}}(\lat) = \set{X \in \PSD^n: X \succ 0, \tilde{\eta}(X^{-1/2} \lat)
\leq 1} \text{ .}
\]
For the case, $\tilde{\eta} = \eta$, we shall simply write ${\rm sm}(\lat)$ for
${\rm sm}_\eta(\lat)$.
\end{definition}

As mentioned above, our first goal is to give an alternate characterization of
the many lattice conjectures in terms of containment factors of the smooth
covariance bodies, which we define presently: 

\begin{definition}[Containment Factor] 
\label{def:cont-factor}
For $A \subseteq \R^n$, we say that $A$ is closed under scaling up if $\lambda
\cdot A \subseteq A$ for all $\lambda \geq 1$. For closed sets $A,B \subseteq
\R^n \setminuszero$ that are closed under scaling up, we define the \emph{containment
factor} of $A$ with respect to $B$ by 
\[
c(A,B) = \inf \set{s \geq 0: s \cdot A \subseteq B} \text{ .}
\]
Since $B$ is closed and does not contain the origin, note that $c(A,B) > 0$.
Furthermore, since $B$ is closed under scaling up, $\lambda \cdot A \subseteq B$
for all $\lambda \geq c(A,B)$. The analogous statements holds for $c(B,A)$. If
both $c(A,B)$ and $c(B,A) < \infty$, we see that
\[
c(A,B) A \subseteq B \subseteq \frac{1}{c(B,A)} A \text{ .}
\]

It will also be useful to examine the containment factor with respect to a point
$\vec{x} \in \R^n$ (which is not closed under scaling up) and set $B$ as above,
in which case we write $c(\vec{x},B)$ for $c(\set{\vec{x}},B)$. 
\end{definition}

The next few lemmas will establish the tight relationship between the
containment factors of the smooth covariance bodies with the
conjectures.  

\begin{lemma} 
\label{lem:sandwich-basics}
Let $\tilde{\eta} \in \set{\eta_{\det}, \eta_\rho, \eta_\mu,
\eta^{\ellipse}_\rho, \eta^{\ellipse}_\mu, \eta}$. For $\lat \subset
\R^n$ an $n$-dimensional lattice: 
\begin{enumerate}
\item For $X \in \PSD^n$, $X \succ 0$,
$
c(X, {\rm sm}_{\tilde{\eta}}(\lat)) = 
\tilde{\eta}^2(X^{-1/2}\lat) \text{ .}
$
\item For $T \in \R^{n \times n}$ invertible,
$
{\rm sm}_{\tilde{\eta}}(T\lat) = \set{T X T^\T: X \in {\rm
sm}_{\tilde{\eta}}(\lat)} \text{ .}
$
\item If $\lat = B \Z^n$, then $\tilde{\eta}(\lat)^2 = c((B^\T B)^{-1}, {\rm
sm}_{\tilde{\eta}}(\Z^n))$. 
\end{enumerate}
\end{lemma}
\begin{proof}
For the first part, we have that for $s \geq 0$
\begin{align*}
sX \in {\rm sm}_{\tilde{\eta}}(\lat) 
&\Leftrightarrow \tilde{\eta}((sX)^{-1/2}\lat) \leq 1 
\Leftrightarrow \tilde{\eta}(X^{-1/2}\lat)/\sqrt{s} \leq 1 \quad \left(\text{ by
homogeneity of $\tilde{\eta}$ } \right) \\
&\Leftrightarrow \tilde{\eta}(X^{-1/2}\lat)^2 \leq s \text{ ,} 
\end{align*}
as needed.

We prove the second part. Note that for $X \succ 0$, the matrices $X^{-1/2}
T^{-1}$ and $(T X T^{\T})^{-1/2}$ are related by an orthogonal transformation
since they both satisfy the relation $B^\T B = (TXT^{\T})^{-1}$. Thus by
Lemma~\ref{lem:eta-continuity}, 
\[
\tilde{\eta}(X^{-1/2}\lat) = \tilde{\eta}(X^{-1/2}T^{-1} T\lat)  
= \tilde{\eta}((T X T^{\T})^{-1/2} T\lat) \text{ .}
\]
In particular, $X \in {\rm sm}_{\tilde{\eta}}(\lat) \Leftrightarrow TXT^\T \in {\rm
sm}_{\tilde{\eta}}(T\lat)$. The statement thus follows.

For the last part, for $\alpha \geq 0$, we have that
\begin{align*}
\tilde{\eta}(\lat) \leq \alpha 
&\Leftrightarrow \alpha^2 I_n \in {\rm sm}_{\tilde{\eta}}(\lat) 
                 \quad \left( \text{ by part 1 } \right) \\
&\Leftrightarrow \alpha^2 I_n \in \set{B X B^T: X \in {\rm sm}_{\tilde{\eta}}(\Z^n)} 
                 \quad \left( \text{ by part 2 } \right) \\
&\Leftrightarrow \alpha^2 (B^\T B)^{-1} \in {\rm sm}_{\tilde{\eta}}(\Z^n) \text{.}
\end{align*}
The result now follows by noting the minimum valid choice for $\alpha$ is
$\tilde{\eta}(\lat)$.
\end{proof}

We now show that the smooth covariance bodies satisfy the basic properties
needed for the containment factors to be meaningfully applied to
them.

\begin{lemma}
\label{lem:smooth-body-basics}
Let $\tilde{\eta} \in \set{\eta_{\det}, \eta_\rho, \eta_\mu,
\eta^{\ellipse}_\rho, \eta^{\ellipse}_\mu, \eta}$. Then for any $n$-dimensional
lattice $\lat \subset \R^n$:
\begin{enumerate}
\item The body ${\rm sm}_{\tilde{\eta}}(\lat)$ is closed, full-dimensional,
closed under scaling up and does not contain the origin.
\item For $X \in {\rm sm}_{\tilde{\eta}}(\lat)$ and $\Delta \in \PSD^n$, 
$X + \Delta \in {\rm sm}_{\tilde{\eta}}(\lat)$.
\item For $X \in \PSD^n$ and $X \succ 0$, 
$c(X,{\rm sm}_{\tilde{\eta}}(\lat)) \lesssim \eta^2(\lat) / \lambda_n(X)$.
\end{enumerate} 
\end{lemma}
\begin{proof}
We prove part $1$. To begin, note that by definition ${\rm
sm}_{\tilde{\eta}}(\lat)$ does not contain the zero matrix, since it only contains
positive definite matrices. We now show that ${\rm sm}_{\tilde{\eta}}(\lat)$ is
closed under scaling up. Take $X \in {\rm sm}_{\tilde{\eta}}(\lat)$. Then for
$\lambda \geq 1$, 
\[
\tilde{\eta}((\lambda X)^{-1/2}(\lat)) = \frac{1}{\sqrt{\lambda}}
\tilde{\eta}(X^{-1/2} \lat) \leq \tilde{\eta}(X^{-1/2}\lat) \leq 1 \text{ ,}
\]
where the first equality follows by positive homogeneity of $\tilde{\eta}$.
Hence $\lambda X \in {\rm sm}_{\tilde{\eta}}(\lat)$ as needed. 

We now show that ${\rm sm}_{\tilde{\eta}}(\lat)$ is full-dimensional. First note
that by Lemma~\ref{lem:sandwich-basics}, $\eta^2(\lat) I_n \in {\rm sm}(\lat)$,
which we note is well defined since $\eta(\lat) \leq \sqrt{n}/\lambda_1(\lat^*)
< \infty$ by Theorem~\ref{thm:smoothing-bnd}. By
Theorem~\ref{thm:conj-sandwich}, note that $\tilde{\eta}((\eta(\lat)^2
I_n)^{-1/2}\lat) = \tilde{\eta}(\lat)/\eta(\lat) \lesssim 1$. Thus, there exists
$C > 0$ such that $Y = C \eta^2(\lat) I_n$ satisfies $\tilde{\eta}(Y^{-1/2}\lat)
< 1/2$. From here, by Lemma~\ref{lem:eta-continuity}, the map $X \rightarrow
\tilde{\eta}(X^{-1/2}\lat)$ is continuous over the positive definite cone. Thus,
there exists an open neighborhood $U$ around $Y$ in the space of symmetric
matrices, such that $\tilde{\eta}(X^{-1/2}\lat) \leq 1$ for all $X \in U$. Thus
$U \subseteq {\rm sm}_{\tilde{\eta}}(\lat)$, and hence ${\rm
sm}_{\tilde{\eta}}(\lat)$ is full-dimensional.

We now show that ${\rm sm}_{\tilde{\eta}}(\lat)$ is closed. For this purpose, it
suffices to show that $S_D = \set{X \in {\rm sm}_{\tilde{\eta}}(\lat): \tr(X) \leq D}$ is closed for any $D > 0$. As argued above, the map $X
\rightarrow \tilde{\eta}(X^{-1/2}\lat)$ is continuous over the set positive
definite matrices. Thus ${\rm sm}_{\tilde{\eta}}(\lat) = \set{X \in \PSD^n: X
\succ 0, \tilde{\eta}(X^{-1/2}\lat) \leq 1}$ is closed under the subspace
topology on $\set{X \in \PSD^n: X \succ 0}$ (as a subspace of the set of
symmetric matrices). We now claim that if $X \in S_D$ then $X \succeq \tau I_n$
for some $\tau > 0$. Assuming this, $S_D$ is also closed under the subspace
topology on $\set{X \in \PSD^n: X \succeq \tau I_n, \tr(X) \leq D}$. Since the
latter set is a closed subset of symmetric matrices, we have that $S_D$ is also
a closed subset of symmetric matrices, as needed. It thus suffices to prove the
claim. Now take $X \in S_D$. By Theorem~\ref{thm:conj-sandwich}, we have
that $\eta_{\det}(X^{-1/2}\lat) \leq C \tilde{\eta}(X^{-1/2}\lat) \leq C$ for
some absolute constant $C > 0$. Let $\lambda_1(X) \geq \dots \geq \lambda_n(X)
\geq 0$ denote the eigenvalues of $X$. Then, in particular, we have that
\begin{align*}
C^{2n}  &\geq \frac{1}{\det((X^{-1/2}\lat)^*)^2} 
     = \det(\lat^*)^{-2} \det(X)^{-1} 
     = \det(\lat^*)^{-2} \prod_{i=1}^n \lambda_i(X)^{-1} \\
     &\geq \det(\lat^*)^{-2} \lambda_n(X)^{-1} (\sum_{i=1}^{n-1}
\lambda_i(X)/(n-1))^{-(n-1)} 
      \geq \det(\lat^*)^{-2} \lambda_n(X)^{-1} D^{-(n-1)} \text{ .}
\end{align*} 
By rearranging, setting $\tau = \frac{1}{C^{2n} \det(\lat^*)^2 D^{n-1}}$, we
have that $\lambda_n(X) \geq \tau$ as needed.

For part 2, let $Y = X + \Delta$ and note that $Y \succeq X$. By definition of
${\rm sm}_{\tilde{\eta}}(\lat)$ it suffices to show that
$\tilde{\eta}(Y^{-1/2}\lat) \leq \tilde{\eta}(X^{-1/2}\lat)$. Letting $T =
Y^{-1/2}X^{1/2}$, by Lemma~\ref{lem:eta-continuity} we have that
\[
\tilde{\eta}(Y^{-1/2}\lat) = \tilde{\eta}(T X^{-1/2}\lat) \leq \|T\|
\tilde{\eta}(X^{-1/2}\lat) \text{ .}
\]  
Thus it suffices to show that $\|T\| \leq 1$. From here,
\[
\|T\| \leq 1 \Leftrightarrow T^\T T \preceq I \Leftrightarrow X^{1/2} Y^{-1}
X^{1/2} \preceq I \Leftrightarrow Y^{-1} \preceq X^{-1} \Leftrightarrow Y
\succeq X \text{ ,}
\]
as needed.

For part 3, for $X \succ 0$ note that $(C\eta(\lat)^2/\lambda_n(X)) X \succeq
C\eta^2(\lat) I_n$ ($X \succ 0$ assures $\lambda_n(X) > 0$). By the proof of
part 1, $C \eta^2(\lat) I_n \in {\rm sm}_{\tilde{\eta}}(\lat)$ for $C \geq 1$ large
enough, and thus by part 2, $(C\eta(\lat)^2/\lambda_n(X)) X \in {\rm
sm}_{\tilde{\eta}}(\lat)$. Hence $c(X,{\rm sm}_{\tilde{\eta}}(\lat)) \lesssim 
\eta(\lat)^2/\lambda_n(X)$, as needed.
\end{proof}

The following lemma shows that the relationship between the $\eta$-type
parameters can be equivalently characterized in terms of the containment factors
between the associated smooth covariance bodies.

\begin{lemma} 
\label{lem:sandwich-equiv}
Let $\eta_1,\eta_2 \in \set{\eta_{\det}, \eta_\rho, \eta_\mu,
\eta^{\ellipse}_\rho, \eta^{\ellipse}_\mu, \eta}$ and $n \ge 1$. Then
$\eta_1(\lat) \leq C(n) \eta_2(\lat)$ for every $n$-dimensional lattice $\lat
\subset \R^n$ if and only if $c({\rm sm}_{\eta_2}(\Z^n),{\rm
sm}_{\eta_1}(\Z^n)) \leq C(n)^2$. 
\end{lemma}
\begin{proof}
By part 3 of Lemma~\ref{lem:sandwich-basics}, 
saying that for every $n$-dimensional lattice $\lat
\subset \R^n$, $\eta_1(\lat) \leq C(n) \eta_2(\lat)$ is equivalent to
saying that for any non-singular $B \in \R^{n \times n}$, 
\[
c((B^\T B)^{-1}, {\rm sm}_{{\eta_1}}(\Z^n)) \le
C(n)^2 
c((B^\T B)^{-1}, {\rm sm}_{{\eta_2}}(\Z^n))
\; ,
\]
which in turn is equivalent to saying that for all $X \succ 0$,
\[
c(X, {\rm sm}_{{\eta_1}}(\Z^n)) \le
C(n)^2 
c(X, {\rm sm}_{{\eta_2}}(\Z^n))
\; .
\]
But the latter is obviously equivalent to 
$c({\rm sm}_{\eta_2}(\Z^n),{\rm
sm}_{\eta_1}(\Z^n)) \leq C(n)^2$.
\end{proof}
\begin{remark}
Observe that for any $n$-dimensional lattice $\lat$, the containment factor
$c({\rm sm}_{\eta_2}(\lat), {\rm sm}_{\eta_1}(\lat))$ is equal to $c({\rm
sm}_{\eta_2}(\Z^n),{\rm sm}_{\eta_1}(\Z^n))$ since the corresponding bodies are
joint linear transformations of each other.
\end{remark}

From the above lemma, we derive the following direct corollary which
re-expresses the main and weak conjecture in terms of containment factors.

\begin{corollary}
\label{cor:equiv-sandwich}
$C_\eta(n)^2 = c({\rm sm}_{\eta_{\det}}(\Z^n),{\rm sm}(\Z^n))$
and $C^{(\mu,\ellipse)}_\eta(n)^2 = c({\rm sm}_{\eta^\ellipse_\mu}(\Z^n), {\rm
sm}(\Z^n))$.
\end{corollary}

We now define the following useful notation: we write $A \sqsubseteq B$, for
sets $A,B \subseteq \R^n$ if there exists an absolute constant $C > 0$ such that
$c(A,B) \leq C$. Using this notation, we derive the following immediate
corollary by combining Theorem~\ref{thm:conj-sandwich} and
Lemma~\ref{lem:sandwich-equiv}.

\begin{corollary}
\label{cor:known-sandwich}
\begin{align*}
{\rm sm}(\Z^n) \sqsubseteq & ~{\rm sm}_{\eta_\mu^{\ellipse}}(\Z^n) 
               \sqsubseteq {\rm sm}_{\eta_{\rho}^{\ellipse}}(\Z^n) \\
               &~~\text{\rotatebox[origin=c]{270}{$\subseteq$}}
      \quad\quad\quad\quad~~~~ \text{\rotatebox[origin=c]{270}{$\subseteq$}} \\
               & ~{\rm sm}_{\eta_\mu}(\Z^n) 
               ~\sqsubseteq
               {\rm sm}_{\eta_\rho}(\Z^n) 
               \sqsubseteq {\rm sm}_{\eta_{\det}}(\Z^n) 
               \subseteq {\rm sm}(\Z^n)/C_\eta(n)^2 \text{ .}
\end{align*}
\end{corollary}

We now give useful inequality descriptions for the smooth covariance bodies and
show that these bodies -- with the exception of ${\rm sm}_{\eta_\mu}$ -- are in
fact convex.  

\begin{lemma}
\label{lem:sm-body-desc}
Let $\lat \subset \R^n$ be an $n$-dimensional lattice. Then
\begin{enumerate}
\item ${\rm sm}_{\eta_{\det}}(\lat) = \set{X \in \PSD^n: \det_W(X) \geq
\frac{1}{\det(\lat^* \cap W)^2}, \forall W \text{ lattice subspace of } \lat^*}$.
\item ${\rm sm}_{\eta_\rho}(\lat) = \set{X \in \PSD^n: \sum_{\vec{y} \in \lat^* \cap
W} e^{-\pi s \vec{y}^\T X \vec{y}} \leq e^{\dim(W)/s}, \forall  W \text{
lattice subspace of } \lat^* \text{ and } s > 0}$.  
\item ${\rm sm}_{\eta_\rho^\ellipse}(\lat) = \set{X \in \PSD^n: \tr(X Y) \geq
\log
\rho_Y(\lat^*), Y \succ 0}$.
\item ${\rm sm}_{\eta_\mu^\ellipse}(\lat) = \set{X \in \PSD^n: \tr(X R^\T R) \geq
\mu(R \lat)^2, ~\forall R \text{ non-singular }}$.
\item ${\rm sm}(\lat) = \set{X \in \PSD^n: \sum_{\vec{y} \in \lat^*
\setminuszero} e^{-\pi \vec{y}^\T X \vec{y}} \leq 1/2}$.
\end{enumerate}
In particular, the above bodies are all convex.
\end{lemma}
\begin{proof}
From the stated representations, the smooth covariance bodies above are either
the intersection of closed halfspaces or the intersection of sublevel /
superlevel sets of continuous convex / concave functions. Thus, all the smooth
covariance bodies listed above are convex. 

We now derive the required representations. Part 1 
follows from the definition of $\mu_{\textnormal{det}}(\lat)$ in~\eqref{eq:det-mu} and
Claim~\ref{clm:detlattransform} (and was already implicitly used 
in the proof of Theorem~\ref{thm:kl-sandwich}).
Part 5 follows immediately
from the definition of the smoothing parameter $\eta$.
It remains to prove parts 2,~3, and~4.

In the following, $X \in \PSD^n$ will denote a positive definite matrix and $W$
will denote a lattice subspace of $\lat^*$.

\noindent For 2, we have that
\begin{align*}
X \in {\rm sm}_{\eta_\rho}(\lat) 
   &\Leftrightarrow \eta_\rho(X^{-1/2} \lat) \leq 1 \\
   &\Leftrightarrow s^2 \log(\rho_{1/s^2}(X^{1/2}(\lat^* \cap W)))/\dim(W) \leq
1,~\forall W \text{ and } s > 0 \\
   &\Leftrightarrow \sum_{\vec{y} \in \lat^* \cap W} e^{-\pi s^2 \vec{y}^\T X
\vec{y}} \leq e^{\dim(W)/s^2},~\forall W \text{ and } s > 0 \text{ ,}
\end{align*}
as needed.

\noindent For 3, we have that
\begin{align*}
X \in {\rm sm}_{\eta_\rho^\ellipse}(\lat) 
   &\Leftrightarrow \eta_\rho^\ellipse(X^{-1/2} \lat) \leq 1 \\
   &\Leftrightarrow \log(\rho_Y(X^{1/2} \lat^*)) \leq \tr(Y),~\forall Y \succ 0 \\
   &\Leftrightarrow \log(\rho_Y(\lat^*)) \leq \tr(Y X),~\forall Y \succ 0 \text{ ,}
\end{align*}
as needed.

\noindent For 4, we have that
\begin{align*}
X \in {\rm sm}_{\eta_\mu^\ellipse}(\lat) 
   &\Leftrightarrow \eta_\mu^\ellipse(X^{-1/2} \lat) \leq 1 \\
   &\Leftrightarrow \mu(RX^{-1/2}\lat)^2 \leq \tr(R^\T R),~\forall R \text{ non-singular } \\
   &\Leftrightarrow  \mu(R\lat)^2 \leq \tr(X R^\T R),~\forall R \text{
non-singular } \text{ .}
\end{align*}
\end{proof}

Using the convexity of the above smooth covariance bodies, we can now give a
useful dual formulation of the containment factor.

\begin{lemma}
\label{lem:sandwich-dual} 
Let $\eta_1,\eta_2 \in \set{\eta_{\det}, \eta_\rho, \eta^{\ellipse}_\rho,
\eta^{\ellipse}_\mu, \eta}$. Then for $C > 0$ and $\lat \subset \R^n$ an
$n$-dimensional lattice,
\[
c({\rm sm}_{\eta_2}(\lat),{\rm sm}_{\eta_1}(\lat)) \leq C \Leftrightarrow
\min_{X \in {\rm sm}_{\eta_2}(\lat)} C \cdot \tr(X Z) \geq \min_{Y \in {\rm
sm}_{\eta_1}(\lat)} \tr(Y Z), ~ \forall Z \in \PSD^n, Z \succ 0 \text{ .}
\]
\end{lemma}
\begin{proof}
The $\Rightarrow$ direction follows trivially by homogeneity and the containment $C \cdot {\rm
sm}_{\eta_2}(\lat) \subseteq {\rm sm}_{\eta_1}(\lat)$. To prove
the $\Leftarrow$ direction, 
let $X \in {\rm sm}_{\eta_2}(\lat)$ be such that $C X \notin {\rm
sm}_{\eta_1}(\lat)$, and we will prove that the right-hand side does not
hold.
Since ${\rm sm}_{\eta_1}(\lat)$ is a closed convex
subset of symmetric matrices, by the separation theorem there exists a
non-zero symmetric matrix $Z \in \R^{n \times n}$ such that
\begin{equation}
\label{eq:sand-dual}
\tr(C X Z) = C \tr(X Z) < \min_{Y \in {\rm sm}_{\eta_1}(\lat)} \tr(Y
Z) \text{ .}
\end{equation}
We shall now show that $Z \in \PSD^n$ and later modify $Z$ to be positive
definite to complete the proof. If $Z \notin \PSD^n$, then there exists $\Delta
\in \PSD^n$ such that $\tr(Z \Delta) < 0$ (e.g.,~$\Delta = \vec{v} \vec{v}^\T$
where $\vec{v}$ is an eigenvector of $Z$ with negative eigenvalue).  Take
an arbitrary $Y \in
{\rm sm}_{\eta_1}(\lat)$. Since $Y + M \Delta \in {\rm sm}_{\eta_1}(\lat)$ for
all $M \geq 0$ (Lemma~\ref{lem:smooth-body-basics} part 2), we have that the minimum 
in~\eqref{eq:sand-dual} is at most $\inf_{M \geq 0} \tr(Z (Y + M \Delta)) =
-\infty$, a clear contradiction. Thus $Z \in \PSD^n$ as needed. 

To make $Z$ positive definite, we simply replace $Z$ by $Z' = Z + \eps I_n$ for
$\eps > 0$ small enough. Note that~\eqref{eq:sand-dual} has a strict
inequality, hence the gap between both sides is at least some $\eps' > 0$.  Since
${\rm sm}_{\eta_1}(\lat)$ contains only positive definite matrices, the right-hand 
side of~\ref{eq:sand-dual} can only increase when replacing $Z$ by $Z'$.
Furthermore, $\tr(C X Z') = \tr(C X) + C \eps \cdot \tr(X) < \tr(CX) + \eps'$,
for $\eps < \eps'/(C\tr(X))$. Hence, for $\eps > 0$ small 
enough,~\eqref{eq:sand-dual} holds with $Z$ replaced by $Z'$ as needed. 
\end{proof}

\subsection{Smooth \texorpdfstring{$\mu$}{mu} tightness and the weak conjecture}
\label{sec:musmoothtight}

Recall from Theorem~\ref{thm:smooth-mu-bnd} the convex program $\mu_{\rm{sm}}$ that provides an upper bound on $\mu$.

\begin{definition}\label{def:smoothmuconstant}(Smooth $\mu$ tightness)
Let $C^{(s\mu)}(n)>0$ be the smallest number such that for any $n$-dimensional
lattice $\lat \subset \R^n$,
\begin{align*}
\mu_{\rm{sm}}(\lat) \le C^{(s\mu)}(n) \mu(\lat)
\text{ .}
\end{align*}
\end{definition}

Intuitively, $C^{(s\mu)}(n)$ being small means that
it is always possible to smooth a lattice with a (not necessarily spherical) 
Gaussian whose expected norm is nearly as small as possible, namely, 
not much more than the covering radius. Obviously, one cannot smooth a 
lattice with a Gaussian of expected norm less than the covering radius --
this is precisely the content of Theorem~\ref{thm:smooth-mu-bnd}. 

Our goal in this section is to show that $C^{(s\mu)}(n) \approx
C^{(\mu,\ellipse)}_\eta(n)$. This implies that the
weak conjecture (Conjecture~\ref{con:covweaketa}) is equivalent to the statement
that $C^{(s\mu)}(n) \le \poly \log n$. We mention in passing that the proof of
Theorem~\ref{thm:mink-to-kl} combined with the easy reverse direction of the KL
conjecture (first inequality in Theorem~\ref{thm:klgeneral}) already implies
that for any lattice $\lat$, $\mu_{\rm{sm}}(\lat) \lesssim C_\eta(n) \mu(\lat)$,
i.e., that $C^{(s\mu)}(n) \lesssim C_\eta(n)$. 

At a high level, the plan is to show that $C^{(s\mu)}(n)$ is in essence the
``dual'' form of $C^{(\mu,\ellipse)}_\eta(n)$. 
We begin by proving the easier direction.

\begin{lemma} 
\label{lem:sm-mu-easy}
For any $n \in \N$, $C^{(\mu,\ellipse)}_\eta(n) \leq C^{(s\mu)}(n)$.
\end{lemma}
\begin{proof}
By Corollary~\ref{cor:equiv-sandwich}, we recall that
$C^{(\mu,\ellipse)}_\eta(n)^2$ is equal to the containment factor 
$c({\rm sm}_{\eta^\ellipse_\mu}(\Z^n), {\rm sm}(\Z^n))$.
By Lemma~\ref{lem:sandwich-dual}, 
it suffices to show that $\forall Z
\succ 0$
\begin{equation}
\label{eq:sm-mu-easy}
C^{(s\mu)}(n)^2 \min \set{\tr(X Z): X \in {\rm sm}_{\eta^\ellipse_\mu}(\Z^n)}
\geq \min \set{\tr(Y Z): Y \in {\rm sm}(\Z^n)}  \text{ .}
\end{equation}
Given $Z$ as above, since $Z$ is positive definite we may write $Z = R^\T R$
for $R$ non-singular. Using the inequality representation for ${\rm
sm}_{\eta^\ellipse_\mu}(\Z^n)$ in Lemma~\ref{lem:sm-body-desc}, for any $X \in
{\rm sm}_{\eta^\ellipse_\mu}(\Z^n)$ we have that
\[
\mu(R\Z^n)^2 \leq \tr(X R^\T R) = \tr(X Z) \text{ ,}
\]
and hence the left hand side of~\eqref{eq:sm-mu-easy} is at least $C^{(s\mu)}(n)^2
\mu(R\Z^n)^2$. Thus, it suffices to prove that
\begin{align*}
C^{(s\mu)}(n)^2 \mu(R\Z^n)^2 &\geq \min \set{\tr(R Y R^\T): Y \in {\rm sm}(\Z^n)} \\
              &= \min \set{\tr(Y): Y \in {\rm sm}(R\Z^n)} \quad \left(\text{ by
Lemma~\ref{lem:sandwich-basics} part 2}\right) \\
              &= \mu_{{\rm sm}}(R\Z^n)^2 \text{ .}
\end{align*}
The desired inequality now follows directly from the definition of $C^{(s\mu)}(n)$.
\end{proof}

Note that the only reason the above proof does not immediately give the reverse
inequality $C^{(\mu,\ellipse)}_\eta(n) \gtrsim C^{(s\mu)}(n)$ is because we
currently only have a \emph{lower bound} of $\mu(R\Z^n)^2$ on the value of the
program $\min \set{\tr(X R^\T R): X \in {\rm sm}_{\eta^\ellipse_\mu}(\Z^n)}$.
Indeed, the main content of this section is to show that $O(\mu(R\Z^n)^2)$ is an
upper bound on the value of this program. Given the inequality description 
\[
{\rm sm}_{\eta^\ellipse_\mu}(\Z^n) = \set{X \in \PSD^n: \tr(X R^\T R) \geq
\mu(R\Z^n)^2, \forall R \text{ non-singular }} \text{ ,}
\] 
note that we must essentially show that all the above inequalities are
irredundant (up to scaling). 

To achieve the desired upper bounds, we will use another smooth approximation of
the covering radius as defined below.

\begin{definition}[Average $\mu$]
For an $n$-dimensional lattice $\lat \subset \R^n$, we define
\[
\bar{\mu}^2(\lat) = \E_{\vec{x} \leftarrow \R^n / \lat}[\dist(\vec{x},\lat)^2]
\]
where $\vec{x}$ is distributed uniformly on $\R^n / \lat$.  Letting $\mathcal{V}
= \mathcal{V}(\lat)$ denote the Voronoi cell of $\lat$, note that
\[
\bar{\mu}^2(\lat) = \E_{\vec{x} \leftarrow \mathcal{V}}[\|\vec{x}\|^2]
\]
where $\vec{x}$ is distributed uniformly on $\mathcal{V}$.
We define the $\eta$-type parameter 
\[
\eta_{\bar{\mu}}^\ellipse(\lat) = \sup_{R \text{ non-singular } }
\frac{\bar{\mu}(R\lat)}{\|R\|_F} \text{ ,}
\]
and the corresponding smooth covariance body
\[
{\rm sm}_{\eta^\ellipse_{\bar{\mu}}}(\lat) = \set{X \in \PSD^n: X \succ 0,
\eta^\ellipse_{\bar{\mu}}(X^{-1/2}\lat) \leq 1} \text{ .}
\]
\end{definition}

The following proposition gives the tight relationship between $\mu$ and $\bar{\mu}$. 
\begin{proposition} 
\label{prop:bar-mu}
For an $n$-dimensional lattice $\lat \subset \R^n$:
\begin{enumerate}
\item $\bar{\mu}(\lat) \leq \mu(\lat) \leq \sqrt{8} \bar{\mu}(\lat)$. In
particular, $\eta^\ellipse_{\bar{\mu}}(\lat) \leq \eta^\ellipse_\mu(\lat) \leq
\sqrt{8} \eta^\ellipse_{\bar{\mu}}(\lat)$.\footnote{We remark
that the constant $\sqrt{8}$ can be improved to $2$; see~\cite{HavivLR09}.}
\item ${\rm sm}_{\eta^\ellipse_{\bar{\mu}}}(\lat) = \set{X \in \PSD^n: \tr(X
R^\T R) \geq \bar{\mu}(R\lat)^2, ~ \forall R \text{ non-singular }}$.
\end{enumerate}
\end{proposition}
\begin{proof}
We prove part 1. Firstly, we see that
\[
\bar{\mu}(\lat)^2 = \E_{\vec{x} \leftarrow \R^n / \lat}[\dist(\vec{x},\lat)^2]
\leq \max_{\vec{x} \in \R^n / \lat} \dist(\vec{x},\lat)^2 = \mu(\lat)^2 \text{ .}
\]
For the reverse inequality, by Claim~\ref{clm:gmrcovering}, we have that
\[
\E_{\vec{x} \leftarrow \R^n / \lat}[\dist(\vec{x},\lat)^2]
\geq (\mu(\lat)^2/4) \Pr_{\vec{x} \leftarrow \R^n / \lat}[\dist(\vec{x},\lat) \geq
\mu(\lat)/2] \geq \mu(\lat)^2/8 \text{ ,}
\]
as needed. For the in particular, it follows immediately from the above
inequalities and the definitions of $\eta^\ellipse_\mu$ and
$\eta^\ellipse_{\bar{\mu}}$. 

For part 2, we simply repeat the derivation of the inequality representation of ${\rm
sm}_{\eta^\ellipse_\mu}(\lat)$ in Lemma~\ref{lem:sm-body-desc}, replacing $\mu$
by $\bar{\mu}$.
\end{proof}

We now present the main technical lemma of this section which gives exact
optimal solutions \dnote{can we show that the optimal solution is unique?} and
values for the relevant minimization problems over ${\rm
sm}_{\eta^\ellipse_{\bar{\mu}}}(\Z^n)$.

\begin{lemma} 
\label{lem:bar-mu-opt}
Let $\lat \subset \R^n$ be an $n$-dimensional lattice. Then for $R \in \R^{n
\times n}$ non-singular, 
\[
\min \set{\tr(X R^\T R): X \in {\rm sm}_{\eta^\ellipse_{\bar{\mu}}}(\lat)}
= \bar{\mu}(R\lat)^2 \text{, }
\]
and an optimal solution to the above program is $R^{-1}\E_{\vec{x} \leftarrow
\mathcal{V}(R\lat)}[\vec{x} \vec{x}^\T] R^{-\T}$.   
\end{lemma}
\begin{proof}
Noting that by Lemma~\ref{lem:sandwich-basics} part 2 (extended in the obvious way), 
\[
\min \set{\tr(X R^\T R): X \in {\rm sm}_{\eta^\ellipse_{\bar{\mu}}}(\lat)}
= \min \set{\tr(X): X \in {\rm sm}_{\eta^\ellipse_{\bar{\mu}}}(R\lat)} \text{, }
\]
it suffices to prove the lemma when $R = I_n$.

From here, by the inequality description from Proposition~\ref{prop:bar-mu},
\[
{\rm sm}_{\eta^\circ_{\bar{\mu}}}(\lat)
= \set{X \in \PSD^n: \tr(X R^\T R) \geq \bar{\mu}(R\lat)^2,~\forall R \text{
non-singular}} \text{ ,}
\]
it is clear that the minimum value above for $R = I_n$ is at least
$\bar{\mu}(\lat)^2$. Now for $X = \E_{\vec{x} \leftarrow
\mathcal{V}(\lat)}[\vec{x}\vec{x}^\T]$, we have that
\[
\tr(X) = \E_{\vec{x} \leftarrow \mathcal{V}(\lat)}[\|\vec{x}\|^2]
= \E_{\vec{x} \leftarrow \R^n / \lat}[\dist(\vec{x},\lat)^2] = \bar{\mu}(\lat)^2.
\]
Thus, it suffices to show that $X$ is feasible. Taking $R \in \R^{n \times n}$
non-singular, we have that
\[
\tr(X R^\T R) = \E_{\vec{x} \leftarrow \mathcal{V}(\lat)}[\|R\vec{x}\|^2] =
\E_{\vec{x} \leftarrow R(\mathcal{V}(\lat))}[\|\vec{x}\|^2] \geq \E_{\vec{x}
\leftarrow \R^n / R\lat}[\dist(\vec{x},R\lat)^2] = \bar{\mu}(R\lat)^2 \text{ ,}
\]
where the last inequality follows since $R\mathcal{V}(\lat)$ is a fundamental
domain for $R\lat$. Thus $X \in {\rm sm}_{\eta^\circ_{\bar{\mu}}}(\lat)$, as
needed.
\end{proof}

Using the above lemma, we can now readily prove the reverse inequality
$C^{(s\mu)}(n) \lesssim C^{(\mu,\ellipse)}_\eta(n)$.

\begin{lemma}
\label{lem:sm-mu-hard}
For any $n \in \N$, $C^{(s\mu)}(n) \leq \sqrt{8} C^{(\mu,\ellipse)}_\eta(n)$.
\end{lemma}
\begin{proof}
For an $n$-dimensional lattice $\lat$, we must show that
\[
\mu_{{\rm sm}}(\lat)^2 \leq 8 C^{(\mu,\ellipse)}_\eta(n)^2 \cdot \mu(\lat)^2 \text{ .}
\] 
By Corollary~\ref{cor:equiv-sandwich}, we recall that $c({\rm
sm}_{\eta^\ellipse_\mu}(\lat),{\rm sm}(\lat)) = C^{(\mu,\ellipse)}_\eta(n)^2$.
Thus, by Lemma~\ref{lem:sandwich-dual},
\begin{equation}
\label{lem:sm-mu-hard1}
\mu_{{\rm sm}}(\lat)^2 = \min \set{\tr(X): X \in {\rm sm}(\lat)} \leq
C^{(\mu,\ellipse)}_\eta(n)^2 \min \set{\tr(X): X \in {\rm
sm}_{\eta^{\ellipse}_\mu}(\lat)} \text{ .}
\end{equation}
Since 
by Proposition~\ref{prop:bar-mu} for any $n$-dimensional lattice $\lat \subset \R^n$, 
$\eta^\ellipse_\mu(\lat) \leq
\sqrt{8} \eta^\ellipse_{\bar{\mu}}(\lat)$,
we have by Lemma~\ref{lem:sandwich-equiv} that
\begin{equation}
\label{lem:sm-mu-hard2}
\min \set{\tr(X): X \in {\rm sm}_{\eta^{\ellipse}_\mu}(\lat)}
\leq 8 \cdot \min \set{\tr(X): X \in {\rm
sm}_{\eta^{\ellipse}_{\bar{\mu}}}(\lat)}
\text{ .}
\end{equation}
Lastly, by Lemma~\ref{lem:bar-mu-opt} and Proposition~\ref{prop:bar-mu}, we get
\begin{equation}
\label{lem:sm-mu-hard3}
\min \set{\tr(X): X \in {\rm sm}_{\eta^{\ellipse}_{\bar{\mu}}}(\lat)}
= \bar{\mu}(\lat)^2 \leq  \mu(\lat)^2 \text{ .}
\end{equation}
The result now follows by
combining~\eqref{lem:sm-mu-hard1},\eqref{lem:sm-mu-hard2},\eqref{lem:sm-mu-hard3}.
\end{proof}


\section{Mixing Time of Brownian Motion on the Torus}
\label{sec:mixingtime}

Given a full-rank lattice $\lat \subset \R^n$, consider the Brownian motion on $\R^n/\lat$ starting from the origin. The probability density function at time $t > 0$ is given by
$f_t : \R^n/\lat \to \R^+$,
\[
f_t(\vec x) = t^{-n/2} \rho_{t}(\lat + \vec x) \; .
\]
As $t$ goes to infinity, the Brownian motion converges to the uniform distribution. 
For $1 \le p < \infty$, the \emph{$L_p$ mixing time} is defined as
\[
\tau_p(\lat) = \inf \set[\Big]{t>0 \,:\, \parens[\Big]{\det(\lat)^{-1} \int_{\R^n/\lat} |\det(\lat) f_t(\vec x) - 1|^p d \vec x }^{1/p} < 1/4 } \; ,
\]
and extended to $p=\infty$ by
\[
\tau_\infty(\lat) = \inf \set[\big]{t>0 \,:\, \forall \vec x, ~ |\det(\lat) f_t(\vec x) - 1| < 1/4 } \; .
\]
We clearly have that 
$\tau_p(\lat) \le \tau_q(\lat)$ for any $1 \le p \le q \le \infty$. 
Also, the constant $1/4$ is arbitrary (see, e.g.,~\cite[Section 4.5]{LevinPW09}). 

Using the Poisson summation formula of Lemma~\ref{lem:coset-mass}, one sees that 
\[ 
  \sup_{\vec x} |\det(\lat) f_t(\vec x) - 1| = \rho_{t^{-1}}(\lat^*
\setminuszero ) \; ,
\]
and hence $\tau_\infty(\lat) = \eta_{1/4}(\lat)^2$. 
Moreover, using Parseval's identity, we get 
\begin{align*}
   \det(\lat)^{-1} \int_{\R^n/\lat} (\det(\lat) f_t(\vec x) - 1)^2 d \vec x  &= 
 \det(\lat) \int_{\R^n/\lat} f_t(\vec x)^2 d \vec x  - 1  \\
 &= \rho_{(2t)^{-1}}(\lat^* \setminuszero ) \; .
\end{align*}
It follows that $\tau_\infty(\lat) \le 2 \tau_2(\lat)$ (and we can even take $1/2$ instead of $1/4$ in the definition of $\tau_2$). 
This property is not unique to Brownian motions on the torus -- the $L_\infty$ mixing time of any reversible Markov chain is always
at most twice the $L_2$ mixing time (see, e.g., the appendix of~\cite{MontenegroTetali06}). 

What about the $L_1$ mixing time? How much smaller can it be than the $L_2$ (or $L_\infty$) mixing time? 
Analyzing the $L_1$ mixing time of Markov chains is generally quite hard. We note that there are examples of 
random walks on finite transitive graphs where the $L_1$ and $L_2$ mixing times differ greatly~\cite{PeresRevelle04}. Still, one can hope that $L_1$ and $L_2$ mixing times 
are close when considering Brownian motion on manifolds such as the torus. This
and related questions were considered by Saloff-Coste (see, e.g.,~\cite{SaloffCoste94,SaloffCoste04,BendikovSaloffCoste03}) who also asks it explicitly in a recent survey~\cite[Problem 11]{SaloffCoste10}.

We next prove that under the weak conjecture, $L_1$ mixing time is approximately the same as the $L_\infty$ mixing time. 
Intuitively, the proof proceeds as follows. If we are below the $L_\infty$ mixing time, 
then by the weak conjecture, there is a certificate for that in the 
form of a Euclidean structure (or equivalently, a linear transformation) under
which the Brownian motion does not reach the covering radius. 
But if this is the case, then surely it cannot even mix in the $L_1$ sense. 

\begin{theorem}
For any $n$-dimensional lattice $\lat$, 
$\tau_\infty(\lat) \lesssim C^{(\mu,\ellipse)}_\eta(n)^2 \,\tau_1(\lat)$.
\end{theorem}
\begin{proof}
Let $\lat$ be an $n$-dimensional lattice. 
Let $X \sim N(0,\frac{1}{2\pi}A)$ be a Gaussian random variable in
$\R^n$ where $A = \tau_\infty(\lat) / (16 C^{(\mu,\ellipse)}_\eta(n)^2) I$. Our goal is to show that 
$\Delta(X \bmod \lat, U) > 1/4$, i.e.,
that it is not mixed in the $L_1$ sense. Assume towards contradiction that 
$\Delta(X \bmod \lat, U) \le 1/4$.
Recalling that $\eta(\lat) = \tau_\infty(\lat)^{1/2}$, 
we get from the definition of $C^{(\mu,\ellipse)}_\eta$ in~\eqref{eq:covweaketagoal} that there exists
 a linear transformation $R$ for which
\begin{align}\label{eq:ellonemixing}
  \mu(R \lat) \ge \tau_\infty(\lat)^{1/2} \sqrt{\tr (R^\T R)} / C^{(\mu,\ellipse)}_\eta(n) \; .
\end{align}
Obviously,
\[
   \Delta(RX \bmod R\lat, U) = \Delta(X \bmod \lat, U) \le 1/4 \; ,
\]
where $U$ denotes the uniform distribution over the appropriate torus (namely, $\R^n/(R\lat)$ and $\R^n/\lat$, respectively). 
By the remark following Theorem~\ref{thm:smooth-mu-bnd}, 
\begin{align*}
    \mu(R\lat) &\le 4 \pi^{-1/2} \sqrt{\tr(R^\T R A)} \\
		&= 4 \pi^{-1/2} \sqrt{\tr(R^\T R)}  \tau_\infty(\lat)^{1/2} / (4 C^{(\mu,\ellipse)}_\eta(n))
		\; ,
\end{align*}
contradicting~\eqref{eq:ellonemixing}.
\end{proof}

\section{Computational complexity and cryptography}
\label{sec:complexity}

In this section, we present the complexity implications of the weak conjecture 
(Conjecture~\ref{con:covweaketa}) for approximating the smoothing parameter (Section~\ref{sec:gapspp}), and
of the $\ell_2$ KL conjecture (Conjecture~\ref{conj:kl-l2}) for approximating the covering radius
(Section~\ref{sec:gapcrp}).

\subsection{The Weak Conjecture and \texorpdfstring{$\GapSPP$}{GapSPP}}
\label{sec:gapspp}

We consider the following two computational problems. 
The first, $\GapSPP$, is the problem of approximating the smoothing parameter $\eta(\lat)$ 
of an input lattice $\lat$. 
The second, Discrete Gaussian Sampling (DGS), is the task of generating samples distributed
according to $D_{\lat+\vec{t},s}$ for parameters $s \geq \eta(\lat)$.
Here, the \emph{discrete Gaussian distribution} $D_{\lat+\vec{t},s}$ is the 
discrete distribution with support $\lat+\vec{t}$ whose probability mass function is proportional
to the restriction of the Gaussian density of standard deviation $s$ to $\lat+\vec{t}$.

The smoothing parameter and the discrete Gaussian distribution 
are closely related. For instance, one of the main important properties 
of the smoothing parameter is that for $s = \tilde{\Omega}(\eta(\lat))$, 
the discrete Gaussian distribution $D_{\lat+\vec{t},s}$ ``behaves like'' a
continuous Gaussian of standard deviation $s$ in terms of its global statistics
such as moments.
Also, both play a fundamental role in lattice-based cryptography,
in particular in the best known worst-case to average-case reductions for lattice problems
(e.g.,~\cite{MR04,R09,GPV08,KaiMinDLP13,MP13}). 
Finally, they recently featured in the fastest known provable algorithms for lattice
problems~\cite{AggarwalDRS15}.


Given the tight relationship between the smoothing parameter and the discrete Gaussian distribution, a
natural question is whether one can compute a good approximation of the 
smoothing parameter (that is, solve $\GapSPP$) using only oracle access to a discrete Gaussian sampler. 
The best known reduction is from $\tilde{O}(\sqrt{n})$-$\GapSPP$ to DGS sampling,
and is implicit in~\cite{MR04}.
The main goal of this section is 
to show that conditioned on the weak conjecture,
one obtains an exponential improvement in the approximation factor,
namely a reduction from $\poly \log n$-$\GapSPP$ to DGS.
The formal statement appears in Theorem~\ref{thm:gapspp-to-sgs}.






\paragraph{Connection to lattice-based cryptography.}


Together with prior work~\cite{MR04,
GPV08,MP13},
this reduction directly implies a worst case to average case reduction from
$\tilde{O}(\sqrt{n})$-$\GapSPP$ to the Shortest Integer Solution problem (SIS),  one of the base hard problems in lattice-based cryptography (see Definition~\ref{def:sis}).
The best unconditional approximation factor is $O(n)$. We will
describe this in more detail in Section~\ref{sec:lbc}.

For the
Learning with Errors (LWE) problem (see~\cite{R09}), the other and perhaps most
versatile base problem in lattice-based cryptography, it was shown in~\cite{KaiMinDLP13} that
$\tilde{O}(\sqrt{n}/\alpha)$-$\GapSPP$ reduces to LWE, where $\alpha$ is the LWE
error parameter. Interestingly, they also show that the reduction of
$\tilde{O}(n/\alpha)$-{\rm GapSVP} to LWE in~\cite{R09} can be recovered by running the
$\GapSPP$ reduction and using the known relations between the smoothing parameter and 
the shortest vector in the dual lattice (i.e., one can factor the reduction through
$\GapSPP$). 

Given the above results, it seems that $\GapSPP$ might be a good alternative
to the standard worst-case problems such as ${\rm GapSVP}$ or ${\rm SIVP}$ for
worst case to average case reductions. Indeed, the obtained approximation factor
is an $\tilde{O}(\sqrt{n})$ factor better when reducing to SIS (conditionally)
and LWE (unconditionally), though one may argue that it is perhaps somewhat
dubious to compare approximation factors with respect to different lattice
problems.  

On a concluding note, we remark that other than the results presented here and
those in~\cite{KaiMinDLP13}, very little work has been done to understand the
fine-grained complexity of $\GapSPP$. We hope here to have helped motivate
its further study.

\paragraph{Overview of the reduction.}
We first explain the known reduction from $\tilde{O}(\sqrt{n})$-$\GapSPP$ to DGS sampling
implicit in~\cite{MR04}.
Given as input a lattice $\lat$ and $s>0$, the reduction simply calls the DGS oracle with $\lat$, $s$, and $\vec{t}=0$,
and then runs a certain statistical test to check if the output ``looks like''
a discrete Gaussian distribution. Specifically, the
test checks that (1) all vectors are in $\lat$ and are of length $O(s \sqrt{n})$ and (2) they span $\R^n$.
It is obvious that this reduction can be implemented in polynomial time. 
Correctness follows by showing that (1) for any parameter $s =
\Omega(\eta(\lat))$, $\tilde{O}(n)$ DGS samples on an $n$-dimensional lattice
$\lat$ are likely to be of length $O(s \sqrt{n})$ and span $\R^n$,
and that (2) for $s = \tilde{O}(\eta(\lat)/\sqrt{n})$, any set of
lattice points of length $O(s \sqrt{n})$ will be contained in a proper subspace
of $\lat$. We note that in the second case, it is crucial that the test is 
guaranteed to fail regardless of the distribution of samples, since
the oracle can behave arbitrarily for $s$ below the smoothing parameter. 

Our improved reduction is from $\poly \log n$-$\GapSPP$ to DGS sampling,
and is conditioned on the weak conjecture. 
Let us assume that we need to distinguish between $\eta(\lat) \leq
1$ and $\eta(\lat) \geq \poly \log n$. We first pick a coset $\vec{t}$ of $\lat$
uniformly at random. We then ask the oracle to produce $O(n)$ DGS samples at
parameter $O(1)$ over $\lat+\vec{t}$. We then compute the empirical second
moment matrix $C$ over these samples, and accept if the largest eigenvalue of
$C$ is $O(1)$ and reject otherwise. The fact that this test succeeds in a yes
instance follows from standard concentration arguments for subgaussian random
variables.  However, for no instances, we require the weak conjecture to prove
soundness (against any distribution, not just DGS). At a technical level, we
will use the weak conjecture to deduce that with constant probability over
$\vec{t}$, for some linear map $R$, $\|R\|_F \leq 1$, the set $R(\lat+\vec{t})$
consists only vectors of length $\Omega(1)$. The mere existence of this matrix
$R$ will turn out to be enough to force the covariance matrix of \emph{any
distribution} on $\lat+\vec{t}$ to have largest eigenvalue of size $\Omega(1)$.
This concludes the reduction. 

As can be seen from the above description, our reduction (as well as that in~\cite{MR04}) actually do not require
true discrete Gaussian samples (i.e., samples from $D_{\lat+\vec{t},s}$) -- 
any subgaussian distribution on $\lat+\vec{t}$ would be equally good. We make this formal
in Definition~\ref{def:dsgs} below, where we define the computational problem 
of discrete subgaussian sampling (DSGS). This mild strengthening of
the reduction turns out to be quite useful for the connection to SIS, since
some of the known reductions to SIS only produce subgaussian samples but not
true discrete Gaussian samples.

\subsubsection{Reduction of GapSPP to Discrete Subgaussian Sampling}

We begin with the necessary preliminaries on subgaussian random variables.

\begin{definition}[Subgaussian Random Variable]
\label{def:subg}
We say that a random variable $X \in \R$ is $s$-subgaussian or subgaussian with
parameter $s$, for $s > 0$, if for all $t \geq 0$,
\[
\Pr[|X| \geq t] \leq 2 e^{-(t/s)^2/2} \text{ .}
\]
We note that the canonical example of a $1$-subgaussian distribution is $N(0,1)$
itself. For a vector-valued random variable $X \in \R^n$, we say that $X$ is
$s$-subgaussian if all its one-dimensional marginals are, i.e., if $\forall \theta \in S^{n-1}$, the random variable
$\pr{X}{\theta}$ is $s$-subgaussian.
\end{definition}

%

\begin{lemma}
\label{lem:subg-sec-mom-bnd}
Let $X \in \R^n$ be an $s$-subgaussian random vector. Then,
\[
\E[XX^\T] \preceq 4 s^2 I_n \text{.} 
\]
\end{lemma}
\begin{proof}
For any $\theta \in S^{n-1}$, we have that
\[
\E[\pr{X}{\theta}^2] = \int_0^\infty \Pr[\pr{X}{\theta}^2 \geq t] dt
\leq \int_0^\infty 2e^{-t/(2s^2)} dt = 4 s^2 \text{, } 
\]
as needed.
\end{proof}

For a random vector $X \in \R^n$, we define $\E[XX^\T]$ to be its second moment
matrix. If $\E[X]=0$, the second moment matrix is also called the covariance matrix
of $X$.

\begin{lemma}[{\cite[Corollary 5.50]{Vershynin12}}]
\label{lem:second-moment-estimate}
Let $X_1, \ldots, X_N \in \R^n$ be i.i.d. $s$-subgaussian random vectors with
second moment matrix $\Sigma$. Then if $N \geq c n (t/\eps)^2$, $\eps \in
(0,1)$, $t \geq 1$, we have that  
\[
\Pr\bracks[\Big]{\length[\Big]{\parens[\Big]{\frac{1}{N}\sum_{i=1}^N X_i X_i^\T} - \Sigma } \geq \eps s^2 } \leq
2 e^{-t^2 n} \text{ ,} 
\]
where $c > 0$ is an absolute constant.
\end{lemma}

We now define the main lattice problems of
interest in this section.

\begin{definition}[Gap Smoothing Parameter Problem]
\label{def:gapspp}
For $\alpha = \alpha(n) \geq 1$, $\eps = \eps(n) \geq 0$,
$\alpha$-$\GapSPP_\eps$ (the Smoothing Parameter Problem) is defined as follows:
given a basis $B$ for a lattice $\lat \subset \R^n$ and $s > 0$, decide
whether $\eta_\eps(\lat) < s$ (YES instance) or $\eta_\eps(\lat) \geq \alpha s$
(NO instance). 
\end{definition}

\begin{definition}[Discrete Subgaussian Sampling Problem]
\label{def:dsgs}
For $\sigma$ a function that maps lattices to non-negative real numbers, and $m
= m(n) \in \N$, $\DSGS{m}{\sigma}$ (the Discrete Subgaussian Sampling Problem) is defined
as follows: given a basis $\vec{B}$ for a lattice $\lat \subset \R^n$, a shift
$\vec{t} \in \R^n$, and a parameter $s \geq \sigma(\lat)$, output a sequence of
$m$ independent and identically distributed $s$-subgaussian random vectors
supported on $\lat+\vec{t}$. We note that the function $\sigma$ need not be
efficiently computable, and that the output of the sampler is only guaranteed
when $s \geq \sigma(\lat)$.
\end{definition}

The main result of this section is the following.

\begin{theorem}
\label{thm:gapspp-to-sgs}
For any $\alpha = \alpha(n) \geq 1$, there exists a polynomial time reduction
from $O(\alpha C^{(\mu,\ellipse)}_\eta(n))$-$\GapSPP_{1/2}$ to
$\DSGS{O(n)}{\alpha \eta}$. 
\end{theorem} 

\begin{proof}
Let $\gamma = 8 \cdot \alpha \cdot C_{\eta}^{(\mu,\ellipse)}(n)$. Let $\lat
\subset \R^n$ and $s > 0$ specify the $\gamma$-$\GapSPP_{1/2}$ instance. 
The reduction proceeds as follows. Sample a
uniformly random coset $\vec{t}$ in $\R^n/\lat$. Use the $\DSGS{N}{\alpha
\eta(\lat)}$ oracle to sample $X_1,\dots,X_{N}$ i.i.d.\ subgaussian random
vectors over $\lat+\vec{t}$ with subgaussian parameter $\alpha \cdot s$, where
$N = c n$ for $c$ to be chosen later. If the oracle fails to return $N$ 
samples in $\lat+\vec{t}$, reject.  Otherwise, if the largest eigenvalue of 
$\frac{1}{N} \sum_i X_i X_i^\T$ is at most $5 \alpha^2 s^2$ we
accept, and otherwise reject.

It is clear that the reduction runs in polynomial time, so we need only verify
its correctness. We will show that for YES instances, the
reduction correctly accepts with overwhelming probability, and that for NO
instances it rejects with probability at least $1/2$.  

Assume that $\eta(\lat) < s$. First, note that the sampling oracle succeeds
since we query it on parameter $\alpha \cdot s > \alpha \eta(\lat)$. 
We can choose $c$ large enough so that by Lemma~\ref{lem:second-moment-estimate} 
(with $\eps=1/2$, $t=1$), with probability at least $1-2e^{-n}$,
\begin{equation}
\label{eq:gapspp1}
\length[\Big]{\frac{1}{N}\parens[\Big]{\sum_{i=1}^N X_i X_i^\T}-\Sigma} \leq 
\alpha^2 s^2 \text{, }
\end{equation}
where $\Sigma = \E[X_1 X_1^\T]$ (recalling that the samples are i.i.d.).
By Lemma~\ref{lem:subg-sec-mom-bnd}, $\Sigma \preceq 4\alpha^2 s^2 I_n$, and
hence conditioning on~\eqref{eq:gapspp1}, we have that 
\[
\frac{1}{N} \sum_{i=1}^N X_i X_i^\T \preceq \Sigma + \alpha^2 s^2 I_n 
                       \preceq 5 \cdot \alpha^2 s^2 \cdot I_n \text{,}
\]
as needed.

Now assume that $\eta(\lat) \geq \gamma \cdot s = 8 \cdot \alpha
\cdot C_\eta^{(\mu,\ellipse)}(n) \cdot s$. Note that if the sample oracle fails,
we correctly reject, and so we may assume that the oracle returns samples
$X_1,\dots,X_N \in \lat + \vec{t}$.

By definition of $C_\eta^{(\mu,\ellipse)}(n)$, there exists a non-singular $R \in \R^{n \times
n}$, $\|R\|_F \leq 1$, such that $\eta(\lat) \leq C_\eta^{(\mu,\ellipse)}(n)
\mu(R\lat)$. 
Since $\vec{t}$ is uniform over $\R^n/\lat$, $R \vec{t}$ is uniform in
$\R^n/(R\lat)$. Thus, with probability at least $1/2$, the distance
between $R\vec{t}$ and $R\lat$ is at least $\mu(R\lat)/2$ (see Claim~\ref{clm:gmrcovering}).
Conditioning on this event, every point in $R(\lat+\vec{t})$ has $\ell_2$ norm
at least $\mu(R\lat)/2$, and in particular, $\|RX_i\|_2 \geq \mu(R\lat)/2$,
$\forall i \in [N]$. Letting $\lambda_1$ denote the largest eigenvalue of
$\frac{1}{N} \sum_{i=1}^N X_i X_i^\T$, we recall the characterization
\begin{align*}
\lambda_1 
&= \max_{Z \in \PSD^n, \tr(Z) \leq 1} 
  \tr\left(Z \cdot \frac{1}{N} \sum_{i=1}^N X_i X_i^\T\right) \\ 
&= \max_{Z \in \PSD^n, \tr(Z) \leq 1} \frac{1}{N} \sum_{i=1}^N X_i^\T Z X_i \text{.}
\end{align*}
Since $R^\T R \succeq 0$ and $\tr(R^\T R) = \|R\|_F^2 \leq 1$, we may plug in $Z
= R^\T R$ above, which yields
\[
\lambda_1 \ge \frac{1}{N} \sum_{i=1}^N \|R X_i\|^2_2 \geq \mu(R\lat)^2/4 \geq
\eta(\lat)^2/(4C_\eta^{(\mu,\ellipse)}(n)^2) \geq 16 \cdot \alpha^2 \cdot s^2 \text{ .} 
\]
Given the above, we see that with probability at least $1/2$ the largest eigenvalue 
is bigger $5 \alpha^2 s^2$. Hence, we reject with probability at least $1/2$,
as claimed.
\end{proof}


\subsubsection{The implications to lattice-based cryptography}
\label{sec:lbc}

Reductions from worst-case lattice problems to the average-case problem SIS 
are often stated as reductions from standard lattice problems such as ${\rm SIVP}$ or ${\rm GapSVP}$. 
Here we observe that they are in fact implicitly reductions from ${\rm DSGS}$. 
As such, they can be combined with the reduction from Theorem~\ref{thm:gapspp-to-sgs} to obtain a 
reduction from the worst case problem $\GapSPP$ to ${\rm SIS}$. 
For concreteness, we will follow here the reduction from~\cite{GPV08}, which
is a simplification of the reduction in~\cite{MR04}. 
A similar conclusion should apply to more recent and refined reductions, such 
as the one in~\cite{MP13}. 

We first define the average-case SIS problem. 
A possible setting of parameters is $q=\poly(n)$, $m= \lceil n \log q \rceil$, and $\beta = \sqrt{m}$. We note
that with this setting solutions are guaranteed to exist, and so the problem is not vacuous. 

\begin{definition}[Small Integer Solution Problem]
\label{def:sis}
For integers $q$, $m$, and $n$,
and a real number $\beta \ge 1$, the ${\rm SIS}_{q,m,\beta}$ problem asks to
find with non-negligible probability, given a matrix $\vec{A}$ chosen uniformly from $\Z_q^{n\times m}$, a nonzero $\vec{e} \in  \Z^m$ such that $\vec{A}\vec{e} = 0 \bmod q$ and
$\|\vec{e}\|_2 \le \beta$.
\end{definition}

We also need to define the following standard lattice problem. 

\begin{definition}[Shortest Independent Vectors Problem]
An input to ${\rm SIVP}_\gamma$ is a basis of an
$n$-dimensional lattice $\lat$. The goal is to output a set of 
$n$ linearly independent lattice vectors $\vec{S} = \{\vec{s}_1,\ldots,\vec{s}_n\} \subset \lat$
such that $\|\vec{S}\| \le \gamma(n) \lambda_n(\lat)$ where $\|\vec{S}\| := \max_i \|\vec{s}_i\|$.
\end{definition}

The following is the main worst-case to average-case reduction shown in~\cite{GPV08}.

\begin{theorem}[{\cite[Proposition 5.7]{GPV08}}] \label{thm:gpvsivp}
For any $m, \beta = \poly(n)$ and for any prime $q \ge \beta \cdot \omega(\sqrt{n \log n})$, 
there is a reduction from ${\rm SIVP}_\gamma$, where $\gamma = \beta \cdot \tilde{O}(\sqrt{n})$,
to ${\rm SIS}_{q,m,\beta}$. 
\end{theorem}

The following is a very slight modification of~\cite[Definition 9.1]{GPV08}
(where it was assumed that $s_{\max} = \|\vec{S}\|$; the modification will be convenient
for us). 

\begin{definition}[Incremental Independent Vectors Decoding]
 An input to ${\rm IncIVD}^{\phi}_{\gamma,g}$
is a tuple $(\vec{B},\vec{S},\vec{t},s_{\max})$ ,
where $\vec{B}$ is a basis for a full-rank lattice $\lat \subset \R^n$, $\vec{S} \subset \lat$ is a full-rank set of lattice vectors such that
$\|\vec{S}\| \le s_{\max}$, $\vec{t} \in \R^n$ is a target point, and $s_{\max} \ge \gamma(n) \phi(\lat)$. 
The goal is to output a lattice vector $\vec{v} \in \lat$ such that $\|\vec{t} - \vec{v}\| \le s_{\max} / g$. 
\end{definition}

Finally, we can state the main technical reduction in~\cite{GPV08}.
\begin{theorem}[{\cite[Theorem 9.2]{GPV08}}]\label{thm:gpvmainred}
For any $g(n) > 1$, $m(n),\beta(n) = \poly(n)$, $\gamma(n) = g(n) \cdot \beta(n) \cdot \sqrt{n}$, $q(n) \ge \gamma(n) \cdot \omega(\sqrt{\log n})$, and negligible $\epsilon(n)$, 
there is a probabilistic poly-time reduction from 
${\rm IncIVD}^{\eta_\eps}_{\gamma,g}$ 
to ${\rm SIS}_{q,m,\beta}$. 
\end{theorem}

Our main observation is that the proof of the theorem above in~\cite{GPV08} actually gives
more: the output $\vec{v}$ of the ${\rm IncIVD}$ solution 
is such that $\vec{t} - \vec{v}$ is subgaussian with parameter at most $2 s_{\max} / (g \sqrt{n})$. 
This follows from (and is implicitly used in) the proof of \cite[Claim 9.5]{GPV08}. Indeed, 
in that proof, $\vec{t} - \vec{v}$ is a linear combination of independent samples from 
discrete Gaussian distributions (on various cosets of the lattice) of parameter $s$. Since $s$ is above smoothing, it follows from~\cite[Lemma 2.8]{MP12} 
that the discrete Gaussian distribution is subgaussian.  As explained in~\cite[Section 2.4]{MP12},
the sum of independent subgaussian variables is again subgaussian, and we obtain that 
$\vec{t} - \vec{v}$ is subgaussian with parameter $s_{\max} / (g \sqrt{n})$. 
Finally, the definition of subgaussianity in~\cite{MP12} is slightly different from ours, and the extra factor of $2$ 
is enough to compensate for this. 

Equipped with this observation, we can prove the main theorem of this section. 

\begin{theorem}
For any $m(n),\beta(n) = \poly(n)$, and prime $q(n) = n \cdot \beta(n) \cdot \omega(\sqrt{\log n})$, 
there is a poly-time reduction from ${\rm DSGS}^{\sigma}$ for $\sigma(\lat) = \tilde{O}(\beta(n) \eta(\lat))$ (with any desired polynomial number of samples) to 
solving ${\rm SIS}_{q,m,\beta}$ on the average with non-negligible probability. 
\end{theorem}
\begin{proof}
Let the input to the ${\rm DSGS}$ problem be a basis $\vec{B}$ for a lattice $\lat \subset \R^n$, a shift
$\vec{t} \in \R^n$, and a parameter $s \geq \sigma(\lat)$.
Let $g = \sqrt{n}$ and $s_{\max} = s\cdot g\cdot \sqrt{n} / 2$.

We start by applying Theorem~\ref{thm:gpvsivp} to obtain a set $\vec{S} \subset \lat$ 
of full-rank vectors of length 
\[
\|\vec{S}\| \le \beta \cdot \tilde{O}(\sqrt{n}) \lambda_n(\lat)
\le \beta \cdot \tilde{O}(n) \eta(\lat)
\le s_{\max}
\; ,
\]
where the second inequality follows from, e.g.,~\cite[Corollary 3.16]{R09}.

Next, choose a negligible $\eps$, say $\eps = n^{-\log_2 n}$. 
By Lemma~\ref{lem:mass-decrease}, 
for any lattice $\lat$, we have that 
$\eta(\lat) \le \eta_{\eps}(\lat) \leq (\log_2 n) \eta(\lat)$.
Therefore, we have $s_{\max} \ge \eta_\eps(\lat)\cdot g \cdot \beta \cdot \sqrt{n}$.
Applying Theorem~\ref{thm:gpvmainred}, and using the observation above, 
we obtain a subgaussian sample from $\lat + \vec{t}$
with parameter $2 s_{\max} / (g \sqrt{n}) = s$, as desired. 
\end{proof}

Combined with Theorem~\ref{thm:gapspp-to-sgs}, this yields a reduction from
$\tilde{O}(\beta(n) C^{(\mu,\ellipse)}_\eta(n))$-$\GapSPP_{1/2}$ to
${\rm SIS}_{q,m,\beta}$. 
In particular, with the setting of parameters mentioned above Definition~\ref{def:sis}, 
and assuming the weak conjecture, we obtain a reduction from  
$\tilde{O}(\sqrt{n})$-$\GapSPP_{1/2}$ to 
${\rm SIS}_{q,m,\beta}$, an improvement upon the best known
reduction by an $\tilde{O}(\sqrt{n})$ factor.


\subsection{KL and the Covering Radius Problem}
\label{sec:gapcrp}

The $\ell_2$ KL conjecture has an easy application to the computational 
complexity of lattice problems. Although this application is not directly 
related to the main topic of this paper, we record it here for future 
reference and for further motivation. 

Consider the decision version of the $\gamma$-approximate covering radius problem, denoted $\GapCRP_\gamma$,
where $\gamma=\gamma(n)>0$ is an approximation factor.
Here, we are given a lattice $\lat$ and a number $r>0$ and the goal is to decide if
the covering radius $\mu(\lat)$ of $\lat$ is at most $r$ (YES instances) or 
more than $\gamma \cdot r$ (NO instances). 
This problem was considered by Guruswami et al.~\cite{GuruswamiMR05} who showed that 
$\GapCRP_{\sqrt{n}} \in \NP$ and that $\GapCRP_2 \in \AM$. They also showed
that $\GapCRP_{\sqrt{n}} \in \coNP$ and that $\GapCRP_{\sqrt{n/\log n}} \in \coAM$.
Here we observe that the $\ell_2$ KL conjecture implies an exponential improvement
on the two latter results. Namely, we have that 
$\GapCRP_{C_{KL}(n)} \in \coNP$. Indeed, this is easy to prove. 
Construct a verifier that is given as proof a lattice subspace $W$ of $\lat^*$ and verifies
that $\det(\lat^* \cap W) \le d^{d/2}$. If $\mu(\lat) \ge C_{KL}(n)$
then such a subspace exists by definition. If, on the other hand, 
$\mu(\lat) \le 1$, then the easy reverse direction of Eq.~\eqref{eq:conjkll2}
shows that no such subspace can exist.

\section{Limits for Strong Reverse Minkowski Inequalities}
\label{sec:morepreciserevminko}

In this somewhat more speculative section, we discuss a possible stronger form
of the main conjecture that, if true, would truly deserve the name ``reverse
Minkowski.'' We recall that the main conjecture bounds from above the number of lattice
points at any radius $r > 0$ by a function of the form $e^{(\poly\log n) r^2}$
when all sublattice determinants are at least $1$. It is a priori unclear why
assuming such a uniform upper bound on all radii is the ``natural'' thing to ask for
(hopefully, we have at least demonstrated its usefulness), and one may be tempted
to ask: given a radius $r > 0$, what is the tightest bound on
the number of lattice points at this radius in terms of sublattice determinants? 

To arrive at a plausible candidate for such a bound, let us recall Minkowski's
first theorem, Theorem~\ref{thm:mink-first}, 
which provides volumetric \emph{lower} bounds on lattice point counts. The lower
bound is stated in terms of the determinant of the full lattice, but obviously,
having a sublattice of small determinant would also suffice, as we can invoke
Minkowski's theorem inside the subspace spanned by that sublattice. More
precisely, it follows from Minkowski's theorem that for any lattice $\lat$, 
\begin{align*}
| \lat \cap r \Ball_2^n| \geq M((r/2),\lat) \text{\; ,}
\end{align*}
where 
\begin{equation}
\label{def:mink-bnd}
M(r,\lat) = \max_{\substack{W \textnormal{ lattice subspace
of } \lat \\  0 \leq d = \dim(W) \leq n}} \vol_d(r\Ball_2^d)/\det(\lat \cap W) \text{\; .}
\end{equation}
By convention, the
quotient is $1$ for $d=0$, and hence we note that $M(r,\lat) \geq 1$ always. 

It now becomes natural to ask whether the volumetric bound
$M(r,\lat)$ is far from being tight. 
To this end, we introduce the following definition. 

\begin{definition}
Let $C_M(n)$ denote the smallest number such that for any $r>0$ and any $n$-dimensional lattice
$\lat$, 
\[
|r\Ball_2^n \cap \lat| \leq M(C_M(n)r,\lat) \; .
\]
\end{definition}

We may speculate that $C_M(n) \le \poly \log n$. 
We note that while this would be amazing if true, a
counterexample would possibly be just as (or more) instructive and
yield useful insights into the structure of lattice points.
We also note that this would imply the main conjecture, Conjecture~\ref{con:maineta}, as
shown in the following proposition. 

\begin{proposition}\label{prop:strongrevminkimpliesmain}
For any $n$, $C_\eta(n) = O(C_M(n))$.
\end{proposition}
\begin{proof}
Let $\lat \subseteq \R^n$ be any lattice for which
all sublattices have determinant at least $1$. 
By scaling, it suffices to show that $\eta(\lat^*) \le O(C_M(n))$. 
For any $r > 0$,
\begin{align*}
|r\Ball_2^n \cap \lat| 
  &\leq M(C_M(n)r,\lat) \leq \max_{0 \leq d \leq n} \vol_d(B_2^d) (C_M(n)r)^d \\
  &\leq \max_{0 \leq d \leq n} \left(\frac{c C_M(n) r}{\sqrt{d}}\right)^d 
  \leq \max_{x \geq 0} \left(\frac{c C_M(n) r}{\sqrt{x}}\right)^x \\
  &= \max_{x \geq 0} e^{x \ln(c C_M(n) r) - (1/2) x \ln x} 
  = e^{(c C_M(n) r)^2/(2e)} \text{ ,} 
\end{align*}
where $c > 0$ is an absolute constant. Here, the last equality follows from the
fact that the function $x \ln a - 1/2 x \ln x$ is concave and is maximized
(derivative is $0$) at $x = a^2/e$. 
By Lemma~\ref{lem:pointcountingeta}, $\eta(\lat^*) \le O(C_M(n))$, as desired.
\end{proof}

The best upper bound we can prove is $C_M(n) = O(\sqrt{n})$, as
shown in Theorem~\ref{thm:weak-mink} below. We note that by 
Proposition~\ref{prop:strongrevminkimpliesmain}, this recovers the best known bound $O(\sqrt{n})$ on $C_\eta(n)$
(a direct proof of which is shown in Theorem~\ref{thm:best-conj-bnds}).

\paragraph{General convex bodies.}
One can take the above discussion a step further and wonder why we should restrict
ourselves to scalings of the Euclidean ball, since Minkowski's theorem holds
more generally for any symmetric convex body. Indeed, slightly overloading
notation, for any symmetric convex body $K \subseteq \R^n$, Minkowski's theorem
implies that $|\lat \cap K| \geq M((K/2),\lat)$, where $M(K,\lat)$ is defined as
in~\eqref{def:mink-bnd} with $rB_2^d$ replaced by $K \cap W$. 

\begin{definition}
Let $C'_M(n)$ denote the smallest number such that for any $n$-dimensional lattice
$\lat$, and any symmetric convex body $K \subseteq \R^n$,
\[
|\lat \cap K| \leq M(C'_M(n)K,\lat) \; .
\]
\end{definition}

Clearly $C'_M(n) \ge C_M(n)$.
Theorem~\ref{thm:weak-mink} below gives the best upper bound 
we can prove, $C'_M(n) \le O(n)$. 
Given the foregoing discussion, one might be tempted to ask whether 
the upper bound on $C'_M(n)$ can be improved, say to $\poly \log n$. 
Unfortunately, we can show that this is false: in Theorem~\ref{thm:rev-mink-lb}
we prove that $C'_M(n) \ge \Omega(n^{1/4-\eps})$.
The proof crucially uses a convex body $K$ that is \emph{very far}
from a Euclidean ball, hence we may still hope that a better inequality is
possible for the Euclidean ball, namely, that $C_M(n) \le \poly \log n$. 

\subsection{Best known upper bounds}

\begin{theorem}[Weak Reverse Minkowski] 
\label{thm:weak-mink}
For a symmetric convex body $K$ and $n$-dimensional lattice $\lat$ in $\R^n$,
$|K \cap \lat| \leq M(3nK, \lat)$. Furthermore, for any radius $r >
0$, $|r\Ball_2^n \cap \lat| \leq M(6\sqrt{n}r, \lat)$. 
\end{theorem}

The above inequality is in fact relatively simple to prove and so we find it
somewhat surprising that it was not discovered earlier. We now give a short proof
based on (the easy direction of) Minkowski's second theorem and a bound of
Henk~\cite{Henk02} on the number of lattice points in a symmetric convex body in
terms of the successive minima.

\begin{proof}
We first prove the bound for general $K$ and specialize to $\Ball_2^n$
afterwards. 
If $\lambda_1(K,\lat)> 1$ then $K \cap \lat =
\set{\vec{0}}$, and hence the desired inequality becomes trivial since
$M(K,\lat) \geq 1$ always.  Thus, we may assume that $\lambda_1(K,\lat)\le 1$.
Let $d \in [n]$ denote largest index such that $\lambda_d(K,\lat)
\leq 1$.

Let $W = {\rm span}(\lat \cap K)$. By construction, $\lat \cap K = (\lat \cap W)
\cap (K \cap W)$, $\dim(\lat \cap W) = d$, and $\lambda_i(K \cap W,\lat \cap W)
= \lambda_i(K,\lat) \leq 1$ $~\forall~i \in [d]$. Now by Henk's bound (see
Lemma~\ref{thm:henk-bnd}), 
\begin{align*}
|K \cap \lat| 
 &\leq 2^{d-1} \prod_{i=1}^d \floor[\Big]{ 1 + \frac{2}{\lambda_i(K \cap W,\lat \cap W)} } \\
 &\leq 2^{d-1} \prod_{i=1}^d \frac{3}{\lambda_i(K \cap W,\lat \cap W)} \quad
\left(\text{ the successive minima are at most $1$ }\right) \\ 
 &\leq 6^d \prod_{i=1}^d \frac{1}{\lambda_i(K \cap W,\lat \cap W)} \\
 &\leq (3^d d!) \frac{\vol_d(K \cap W)}{\det(\lat \cap W)} \quad \left(\text{ by
Minkowski's Second Theorem, Theorem~\ref{thm:mink-sec} }\right)  \\
 &\leq \frac{\vol_d(3d (K \cap W) )}{\det(\lat \cap W)} \leq M(3d K,
\lat) \text{ .} \quad \left(\text{ by homogeneity of volume }\right)
\end{align*}
To obtain the improvement for $K=\Ball_2^n$ (by scaling $\lat$ it suffices to
prove the ``furthermore'' part for $r=1$), we note that the $d!$ term above obtained from
Minkowski's second theorem can be improved to $2^d \vol_d(\Ball_2^d)^{-1} \leq
2^d d^{d/2}$. 
\end{proof}

\subsection{Lower bound for general convex bodies}

\begin{theorem}
\label{thm:rev-mink-lb}
For any $\eps > 0$ there exists a constant $C > 0$ such that for any $n$ large
enough, there exists an $n$-dimensional symmetric convex body $K$ and lattice
$\lat$ such that $M(C n^{1/4-\eps} K, \lat) \leq |\lat \cap K|$. 
\end{theorem}

For our construction, we will require the following lower bound on the
determinants of sublattices of a random lattice. The formulation below is due
to~\cite{ShapiraW14}, which in turn is based on the estimates
of~\cite{Thunder98}.  

\begin{theorem}[{\cite[Proposition 3]{ShapiraW14}}]
\label{thm:det-lb-haar}
Let $\lat$ be a random $n$-dimensional lattice of determinant $1$. 
Then there exists an absolute constant $\alpha >
0$, such that with probability tending to $1$ as $n \rightarrow \infty$, the
following holds: 
$\forall d \in [n]$, $\forall W$ lattice subspace of $\lat$,
$\dim(W) = d$, $\det(\lat \cap W) \geq (\alpha n^{(n-d)/(2n)})^d$.  
\end{theorem}

We will also need Urysohn's mean width inequality. 

\begin{theorem}[\cite{Urysohn24}] Let $K \subseteq \R^n$ be a symmetric 
convex body. For $X \sim N(0,I_n)$, the following holds:
\[
\E[\max_{\vec{y} \in K} |\pr{X}{\vec{y}}|] \geq \E[\|X\|_2]
\left(\frac{\vol_n(K)}{\vol_n(\Ball_2^n)}\right)^{1/n} \gtrsim n \vol_n(K)^{1/n}  \text{.}
\]
\end{theorem}

\begin{proof}[Proof of Theorem~\ref{thm:rev-mink-lb}] 
Let $\lat$ denote a random lattice of determinant $1$ in $\R^n$ satisfying
the estimates of Theorem~\ref{thm:det-lb-haar}. Let $K = {\rm conv}(\pm
\vec{v}_1,\dots, \pm \vec{v}_N)$ be the symmetric convex hull of the $N =
\ceil{2^{n^{1/2+\eps}}}$ shortest points in
$\lat$, where $\eps$ is in $(0,1/4)$. 
We will show that for $n$ large enough, $K$ and $\lat$ satisfy the
conditions of theorem.  

First note that by construction $|K \cap \lat| \geq N$. Moreover, since
$\det(\lat) = 1$, by Minkowski's first theorem there are at least $2^n$ points
at radius $4\vol_n(\Ball_2^n)^{-1/n} = O(\sqrt{n})$. Therefore, all the points
in $K$ also have length $O(\sqrt{n})$. Now for a subspace $W$ of dimension $d$, we have that
\begin{align*}
\vol_d(K \cap W)^{1/d} 
  &\leq \vol(\pi_W(K))^{1/d} \\
  &\lesssim \frac{1}{d} 
        \E_{X \sim N(0,I_n)}[\max_{i \in [N]} |\pr{X}{\pi_W(\vec{v}_i)}|]
\quad \left(\text{ by Urysohn's inequality }\right) \\
  &\lesssim \frac{1}{d} \sqrt{\log N} \max_{i \in [N]} \|\vec{v}_i\|_2 
  \quad \left(\text{ by the union bound }\right) \\
  &\lesssim \frac{n^{3/4+\eps/2}}{d}  \text{ .}
\end{align*}
In particular, there exists a universal constant $\gamma > 0$, such that
\begin{equation}
\label{eq:vol-lb-bad-mink}
\vol_d(K \cap W) \leq (\gamma n^{3/4+\eps/2}/d)^d \text{ .}
\end{equation}

\dnote{Note for poterity: It seems that the weakness of the current argument is that we don't say
anything about how the vectors must shrink as the dimension of $W$ decrease}


Putting~\eqref{eq:vol-lb-bad-mink} together with the estimates from
Theorem~\ref{thm:det-lb-haar}, we get that for a $d$-dimensional lattice subspace $W$ of $\lat$ and $t \geq 0$, 
\begin{equation}
\label{eq:vol-bnds}
\frac{\vol_d(t K \cap W)}{\det(\lat \cap W)} 
\leq \parens[\Big]{t \cdot \frac{\gamma}{\alpha}
\cdot n^{1/4+(\eps + d/n)/2}/d}^d \text{.} 
\end{equation}
Take $\eps' \in (0,1/4)$, and set $C = \alpha \eps'/\gamma$
and $t = C n^{1/4-(\eps+\eps')/2}$. 
Then the right-hand side of~\eqref{eq:vol-bnds} is at most
\begin{equation}
\label{eq:vol-bnds-two}
\parens[\Big]{\eps' 
\cdot n^{1/2+d/(2n)-\eps'/2}/d}^d \text{.} 
\end{equation}
In the rest of the proof we will show that for large enough $n$, 
\eqref{eq:vol-bnds-two} is bounded from above by $N$ for all $1\le d \le n$. This 
implies that $M(tK,\lat) \leq N$ and completes the proof. 

We consider three cases. First, if $\eps'n \leq d \leq n$, then~\eqref{eq:vol-bnds-two}
is at most 
\[
\parens[\big]{\eps' 
\cdot n^{1-\eps'/2}/(\eps' n)}^d \le 1 \le N\text{.} 
\]
If $\sqrt{n} \leq d \leq \eps' n$, \eqref{eq:vol-bnds-two}
is at most 
\[
\parens[\Big]{\eps' 
\cdot n^{1/2+\eps'/2-\eps'/2}/\sqrt{n}}^d \le 1 \le N\text{.} 
\]
Finally, if $1 \le d \le \sqrt{n}$, \eqref{eq:vol-bnds-two}
is at most 
\[
\parens[\Big]{n^{1/2+\sqrt{n}/(2n)-\eps'/2}}^{\sqrt{n}}
= 
2^{(1/2+\sqrt{n}/(2n)-\eps'/2)\sqrt{n} \log_2 n} \text{,} 
\]
which for large enough $n$ is less than $N$ since the exponent 
is less than $n^{1/2+\eps}$.
\end{proof}

\section{A continuous relaxation}
\label{sec:pisier}

As evidence towards our main conjecture, we will show that a reasonable
``continuous relaxation'' of it holds. To state the continuous relaxation, we
will need the following alternate characterization of the smoothing parameter
in terms of the Voronoi cell of the dual lattice (see Definition~\ref{def:voronoi-cell}).

\begin{definition}[$K$-norm] 
\label{def:k-norm}
Let $K \subseteq \R^n$ be a centrally symmetric
convex body. We define $\|\vec{x}\|_K = \min \set{s \geq 0: \vec{x} \in sK}$,
$\forall \vec{x} \in \R^n$, to be the norm induced by $K$. 
\end{definition}

\begin{theorem}[\cite{thesis/D12}]
\label{thm:eta-gaus-process}
Let $\lat \subseteq \R^n$ be an $n$-dimensional lattice, and let $\calV^* =
\calV(\lat^*)$. Then for $X \sim N(0,I_n)$, we have that
\[
\E[\|X\|_{\calV^*}] = \Theta(1) \eta(\lat) \text{.}
\]
\end{theorem}

The above theorem characterizes the smoothing parameter as a Gaussian norm
expectation, which allows us to make useful connections with the study of
Gaussian processes and convex geometry~\cite{Pisier89}. 
Recalling that the reverse direction of the main conjecture is easy, 
and taking duals, we have that the main conjecture
is equivalent to saying that for any 
lattice $\lat \subset \R^n$ with Voronoi cell $\calV$, 
\begin{align}\label{eq:finalgoalequiv}
   1 \lesssim
	\E[\|X\|_{\calV}]
	\min_{\substack{W \textnormal{ lattice subspace of } \lat \\ d = \dim(W) \in
[n]}} (\det(\lat \cap W))^{1/d}
    \lesssim \poly \log n \;.
\end{align}

We now make the easy observation that for any lattice $\lat$ with Voronoi cell $\calV$,
and any lattice subspace $W$ of $\lat$,
\begin{align}\label{eq:dettovol}
\det(\lat \cap W) \geq \vol_d(\calV \cap W) \; .	
\end{align}
To see this, notice that since the Voronoi cell $\calV$ tiles space with respect to $\lat$, 
$\calV \cap W$ yields a \emph{packing} with respect to $\lat \cap W$. 
That is, the translates of $\calV \cap W$ induced by $\lat \cap W$ are all interior disjoint,
implying~\eqref{eq:dettovol}. 
Therefore,
\begin{align*}
\min_{\substack{W \textnormal{ lattice subspace of } \lat \\ d=\dim(W) \in
[n]}} \det(\lat \cap W)^{1/d} 
&\geq 
\min_{\substack{W \textnormal{ lattice subspace of } \lat \\ d=\dim(W) \in
[n]}} \vol_d(\calV \cap W)^{1/d} \\
&\geq 
\min_{\substack{W \subseteq \R^n \\ d=\dim(W) \in
[n]}} \vol_d(\calV \cap W)^{1/d} 
\text{ .}
\end{align*}
Using this, we can now relax the upper bound in~\eqref{eq:finalgoalequiv} and
ask whether for any symmetric convex body $K$ which is the Voronoi cell of some
lattice \dnote{why care about whether its a Voronoi cell?},
\begin{align}\label{eq:finalgoalequivcontinuous}
   1 \lesssim
	\E[\|X\|_{K}]
	\min_{\substack{W \subseteq \R^n \\ d=\dim(W) \in
[n]}} \vol_d(K \cap W)^{1/d} 
    \lesssim \poly \log n \;.
\end{align}
We view this as a natural continuous relaxation of our main conjecture. 

Amazingly, Eq.~\eqref{eq:finalgoalequivcontinuous} is \emph{known to be true},
and is in fact true for any 
symmetric convex body $K$. As always, the lower bound is standard, and we
include a proof in Section~\ref{sec:vollowerbound} below for completeness.  The
surprising thing is that the upper bound is true.  This is a direct consequence
of a theorem of Milman and Pisier~\cite{MP87} known as the \emph{volume number
theorem} (see also~\cite[Chapter 9]{Pisier89}). Precisely, the volume number
theorem implies\footnote{We remark that one of the logarithmic factors can be
replaced by the so-called K-convexity constant of the Banach space
$(\R^n,\|\cdot\|_K)$, however this makes no essential difference in the current
application.}
\begin{equation}
\label{thm:pisier-milman}
\E[\|X\|_K] 
\min_{\substack{W \subseteq \R^n \\ d=\dim(W) \in [n]}}
\vol_d(K \cap W)^{1/d}
\lesssim \log(n+1)^2  \text{.}
\end{equation}

The Milman-Pisier theorem thus gives us another point of evidence that the main
conjecture might be correct. It is the most non-trivial implication of the main 
conjecture that we know is true. 
Furthermore, one may even attempt to emulate the proof
of the volume number theorem in a way that respects the discrete structure of
the lattice (indeed, one needs to find a \emph{lattice subspace}), though we
have had limited success on this front. 

Moreover, there is unfortunately 
still a convex geometric ``obstruction'' to this approach. 
To see it, we start by observing that~\eqref{eq:dettovol} can be 
strengthened, and in fact for any $d$-dimensional lattice subspace $W$,
\begin{align*}
\det(\lat \cap W) &= \vol_d(\calV(\lat \cap W)) \\
                    &= \vol_d\parens[\Big]{\set[\Big]{\vec{x} \in W: \pr{\vec{x}}{\vec{y}} \leq
\frac{1}{2} \|\vec{y}\|_2^2,~\forall \vec{y} \in \lat \cap W \setminuszero}} \\ 
                    &\geq \vol_d(\pi_W(\calV))
\text{,}
\end{align*}
where the inequality follows from the fact that the constraints of
$\calV(\lat \cap W)$ (except for the constraint that $\calV(\lat \cap W)
\subseteq W$) form a subset of the constraints of $\calV$. 
This is stronger than~\eqref{eq:dettovol} since obviously $\vol_d(\pi_W(\calV)) \geq \vol_d(\calV \cap W)$.
Thus, one might argue that instead of~\eqref{eq:finalgoalequivcontinuous},
the ``correct'' continuous relaxation of the main conjecture should say that
for any symmetric convex body $K$ which is the Voronoi cell of some lattice,
\begin{align*}
   1 \lesssim
	\E[\|X\|_{K}]
	\min_{\substack{W \subseteq \R^n \\ d=\dim(W) \in
[n]}} \vol_d(\pi_W(K))^{1/d} 
    \lesssim \poly \log n \;.
\end{align*}
Being weaker than that in~\eqref{eq:finalgoalequivcontinuous},
the lower bound clearly still holds for all symmetric convex bodies. 
Unfortunately, we do not know whether the upper bound holds, even
though it seems plausible that it does hold for all symmetric convex bodies $K$.
Interestingly, 
if that were the case, it would have non-trivial implications in convex geometry. In
particular, it was communicated to us by E. Milman~\cite{Milman15} (see also~\cite{GiannopoulosM14,Milman15IMRN}), that such a
result would yield a new proof of the best known bound for the slicing
conjecture for convex bodies (up to polylogarithmic factors). 

\subsection{The volumetric lower bound}
\label{sec:vollowerbound}
For completeness, we include the following standard
lower bound (see for example~\cite{Pisier89}).

\begin{lemma}
\label{lem:gauss-vol-lb}
Let $K \subseteq \R^n$ be a symmetric convex body and let $\|\cdot\|_K$ the
induced norm. Then for $X \sim N(0,I_n)$, the following holds:
\[
\E[\|X\|_K] \gtrsim \max_{\substack{W \subseteq \R^n \\ d=\dim(W)  \in [n]}}
\frac{1}{\vol_d(K \cap W)^{1/d}} \text{.}
\]
\end{lemma}
\begin{proof}
We first prove the case $W = \R^n$. 
We integrating in polar coordinates and use Jensen's inequality,
\begin{align*}
\E[\|X\|_K] 
          &= \E[\|X\|_2] \int_{S^{n-1}} \|\theta\|_K {\rm d}\theta \\   
          &\geq \E[\|X\|_2] \left(\int_{S^{n-1}} \|\theta\|_K^{-n} 
                     {\rm d}\theta\right)^{-1/n} 
                    \quad \left(\text{ by Jensen }\right) \\ 
          &= \E[\|X\|_2] \left(\frac{\vol_n(K)}{\vol_n(\Ball_2^n)}\right)^{-1/n} 
           \gtrsim \frac{1}{\vol_n(K)^{1/n}} \text{.}
\end{align*}
To prove it for all $W$,
note that
\[
\E[\|X\|_K] = \E[\|\pi_W(X) + \pi_{W^\perp}(X)\|_K] 
            \geq \E[\|\pi_W(X)\|_{K \cap W}] \text{, }
\]
by Jensen, since $\pi_W(X)$ and $\pi_{W^\perp}(X)$ are independent and
$\E[\pi_{W^\perp}(X)] = 0$. We recover the desired lower bound applying the full
dimensional inequality to $\E[\|\pi_W(X)\|_{K \cap W}]$.
\end{proof}

\section{Sanity checks}
\label{sec:sanity}

We begin by giving the best known bounds on both the main conjecture and the
$\ell_2$ Kannan-Lov{\'a}sz conjecture, which are by now classical. We give
a proof here for completeness.

\begin{theorem} 
\label{thm:best-conj-bnds}
For $n \in \N$, $C_{\eta}(n) = O(\sqrt{n})$ and $C_{KL}(n) = O(\sqrt{n})$. 
\end{theorem}
\begin{proof}
Let $\lat$ be an $n$-dimensional lattice. The bound of $C_{\eta}(n) =
O(\sqrt{n})$, follows directly from the classical inequality (see
Theorem~\ref{thm:smoothing-bnd}) 
\[
\eta(\lat) \leq \frac{\sqrt{n}}{\lambda_1(\lat^*)} \text{, }
\]
noting that the right-hand side corresponds to lattice subspaces of $\lat^*$ of
dimension $1$.

We now bound $C_{KL}(n)$. For this purpose, let $B=(\vec{b}_1,\dots,\vec{b}_n)$
denote an Hermite-Korkin-Zolotarev (HKZ) basis of $\lat$.  That is, a basis for which $\|\gs{\vec{b}_i}\|_2 =
\lambda_1(\pi_i(\lat))$, where $\pi_i = \pi_{{\rm
span}(\vec{b}_1,\dots,\vec{b}_{i-1})^\perp}$. It now follows from Babai's
nearest plane algorithm~\cite{Bab86} that
\[
\mu(\lat)^2 \leq \frac{1}{4} \sum_{i=1}^n \|\gs{\vec{b}_i}\|_2^2 \text{ ,}
\]
Thus there exists $i \in [n]$ such that $\mu(\lat) \leq
(\sqrt{n}/2)\|\gs{\vec{b}_i}\|_2$. Now let $W = {\rm
span}(\vec{b}_1,\dots,\vec{b}_{i-1})^\perp$ and note that $\dim(W) = n-i+1$. By
definition of an HKZ basis, we have that
\begin{align*}
\mu(\lat) &\leq (\sqrt{n}/2) \|\gs{\vec{b}_i}\|_2 
          = (\sqrt{n}/2) \lambda_1(\pi_W(\lat)) \\
          &\leq (\sqrt{n}/2) \sqrt{n-i+1} \det(\pi_W(\lat))^{1/(n-i+1)} \quad
\left(\text{ by Minkowski's first theorem }\right) \\
          &= (\sqrt{n}/2) \frac{\sqrt{n-i+1}}{\det(\lat^* \cap W)^{1/(n-i+1)}}
           \text{ ,}
\end{align*}
as needed.
\end{proof}

\subsection{Conjectures under direct sums}
\label{sec:cross-product}

The main goal of this section is to prove that both the main conjecture and
the KL conjecture behave well under direct sums of lower dimensional lattices. In
particular, they cannot be disproved by direct sum constructions.

To formalize this, define $C_\eta(k,n)$ and $C_{KL}(k,n)$ to be the least numbers such
that for all lattices of the form $\lat = \lat_1 \oplus \cdots \oplus \lat_n$, where
$\lat_i$, $i \in[n]$, is a lattice of dimension at most $k$, we have that
\[
\eta(\lat) \leq C_\eta(k,n) \max_{W \text{ lattice subspace of } \lat^*}
\frac{1}{\det(\lat^* \cap W)^{1/\dim(W)}} \text{ ,}
\]
and
\[
\mu(\lat) \leq C_{KL}(k,n) \max_{W \text{ lattice subspace of } \lat^*}
\frac{\sqrt{\dim(W)}}{\det(\lat^* \cap W)^{1/\dim(W)}} \text{ .}
\]

We now show that the direct sum version of $C_\eta$ is easily bounded.
\begin{lemma} 
\label{lem:main-cross}
For all $k,n \in \N$, $C_\eta(k,n) \leq \sqrt{\log_2(n/\ln(3/2))} C_\eta(k)$.
\end{lemma}
\begin{proof}
Take $\lat = \lat_1 \oplus \dots \oplus \lat_n$, where $\dim(\lat_i) \leq
k$, $i \in [n]$. Let $s = \sqrt{\log_2(c n)}
\max_{i \in [n]} \eta(\lat_i)$, where $c = 1/\ln(3/2)$. We first show that
$\eta(\lat) \leq s$. In particular, it suffices to show that
$\rho_{1/s^2}(\lat^*) \leq 3/2$. By Lemma~\ref{lem:mass-decrease},
\begin{align*}
\rho_{1/s^2}(\lat^*) &= \prod_{i=1}^n \rho_{1/s^2}(\lat^*_i) 
                     = \prod_{i=1}^n (1+\rho_{1/s^2}(\lat^*_i \setminuszero)) \\
                     &\leq \prod_{i=1}^n (1 + (1/2)^{\log_2 (c n)})
                     =(1 + 1/(cn))^n \leq e^{1/c} = 3/2 \text{.}  
\end{align*}
Therefore, by definition of $C_\eta(k)$, we have that
\[
\eta(\lat) \leq \sqrt{\log_2(c n)} \max_{i \in [n]} \eta_i(\lat_i)
           \leq \sqrt{\log_2(c n)} C_k(\eta) \max_{i \in [n]} \max_{W \text{
lattice subspace of } \lat^*_i} \frac{1}{\det(\lat^*_i \cap W)^{1/\dim(W)}}
\text{ .}
\]
To finish the proof, we note that lattice subspaces of the $\lat_i$s are
naturally identified with subspaces of $\lat$.
\end{proof}

Next, we show the corresponding result for the direct sum version of $C_{KL}$.
\begin{lemma}
\label{lem:kl-cross}
For all $k,n \in \N$, $C_{KL}(k,n) \leq 2\sqrt{\log_2(kn) +1} C_{KL}(k)$.
\end{lemma}
\begin{proof}
Take $\lat = \lat_1 \oplus \dots \oplus \lat_n$, where $\dim(\lat_i) \leq k$, $i
\in [n]$. Since $\lat$ is a direct sum, $\mu(\lat)^2 = \sum_{i=1}^n
\mu(\lat_i)^2$. Since each $\lat_i$, $i \in [n]$, has dimension at most $k$,
there exists lattice subspaces $W_1,\dots,W_n$ of $\lat_1^*,\dots,\lat_n^*$, with
$d_i = \dim(W_i) \leq k$, $i \in [n]$, satisfying
\begin{equation}
\label{eq:kl-cross1}
\mu(\lat)^2 \leq C_{KL}(k)^2 \sum_{i=1}^n \frac{d_i}{\det(\lat_i^* \cap
W_i)^{2/d_i}} \text{ .}
\end{equation}
From here, applying Lemma~\ref{lem:rev-am-gm}, we have that
\begin{equation}
\label{eq:kl-cross2}
\sum_{i=1}^n \frac{d_i}{\det(\lat_i^* \cap W_i)^{2/d_i}} \leq 4\ceil{\log_2 d_{[m]}}
\max_{S \subseteq [n]} d_S\left(\prod_{i \in S} \frac{1}{\det(\lat_i^* \cap
W_i)^2}\right)^{1/d_S} \text{.}
\end{equation}
Now for each $S \subseteq [m]$, we can associate the lattice subspace $Z =
Z_1 \oplus \dots \oplus Z_m$ of $\lat^*$, where $Z_i = W_i$ if $i \in S$ and
$Z_i = \set{\vec{0}}$ otherwise, satisfying $\dim(Z) = d_S$ and $\det(\lat^*
\cap Z) = \prod_{i \in S} \det(\lat_i^* \cap W_i)$. The lemma now follows
by combining~\eqref{eq:kl-cross1} and~\eqref{eq:kl-cross2}, and noting that $4\ceil{\log_2
d_{[m]}} \leq 4(\log_2(kn) + 1)$.
\end{proof}

\subsection{Conjectures for random lattices}
\label{sec:randomsanity}
In this section, we show that the conjectures are easy for ``random
lattices''. For the general Kannan-Lov{\'a}sz conjecture, as mentioned earlier,
this already follows from a classical theorem of Rogers~\cite[Theorem 2]{Rogers1958}.
For completeness, we give a simple proof for the $\ell_2$ case. We note
that random lattices are the main class of worst-case examples for essentially
all known transference theorems, in particular, for Khinchine's flatness
theorem. Thus we find it interesting that they are essentially the easiest
lattices for the conjectures considered here.

We now show that for random lattices, getting determinantal bounds on the $\ell_2$
covering radius and smoothing parameter is easy, verifying both conjectures for
this class of lattices.

\begin{lemma}
Let $\lat$ be a random $n$-dimensional lattice. Then $\eta(\lat) = O(1/\det(\lat^*)^{1/n})$ and $\mu(\lat) = O(\sqrt{n}
/\det(\lat^*)^{1/n})$ with overwhelming probability.
\end{lemma}
\begin{proof}
Note that since $\det(\lat) = \det(\lat^*) = 1$ by definition, it suffices to
show that $\eta(\lat) = O(1)$ and $\mu(\lat) = O(\sqrt{n})$. By
Theorem~\ref{thm:smooth-mu-bnd}, we already have that $\mu(\lat) \leq
O(\sqrt{n})\eta(\lat)$, thus it suffices to prove that $\eta(\lat) = O(1)$.

We now show that $\eta(\lat) \leq 2$ with probability at least $2^{-n+1}$.  For
this purpose, we show that $\rho_{1/2^2}(\lat^* \setminuszero) \leq 1/2$
with probability at least $2^{-n+1}$. By Theorem~\ref{thm:haar}, recall that 
$\lat^*$ and $\lat$ are identically distributed. From here, we have that by~\eqref{eq:siegelforfunctions},
\begin{align*}
\E[\rho_{1/4}(\lat^* \setminuszero)] 
&= \E[\sum_{\vec{y} \in \lat^* \setminuszero} e^{-\pi \|2\vec{y}\|^2}] \\
&= \int_{\R^n} e^{-\pi \|2\vec{y}\|^2} dy = 2^{-n} \text{.} 
\end{align*}
Hence, by Markov's inequality $\Pr[\rho_{1/4}(\lat^* \setminuszero) \geq 1/2]
\leq 2^{-n+1}$, as needed.
\end{proof}


\bibliographystyle{alphaabbrvprelim}
\bibliography{reverse_minkowski}

\end{document}